\documentclass[english]{article}
\usepackage[T1]{fontenc}
\usepackage[latin9]{inputenc}
\usepackage{geometry}
\usepackage{refstyle}
\usepackage{float}
\usepackage{url}
\usepackage{amsmath}
\usepackage{amsthm}
\usepackage{amssymb}
\usepackage{graphicx}
\usepackage[authoryear]{natbib}

\makeatletter

\AtBeginDocument{\providecommand\secref[1]{\ref{sec:#1}}}
\AtBeginDocument{\providecommand\subsecref[1]{\ref{subsec:#1}}}
\AtBeginDocument{\providecommand\eqref[1]{\ref{eq:#1}}}
\AtBeginDocument{\providecommand\figref[1]{\ref{fig:#1}}}
\AtBeginDocument{\providecommand\Subsecref[1]{\ref{Subsec:#1}}}
\AtBeginDocument{\providecommand\propref[1]{\ref{prop:#1}}}
\AtBeginDocument{\providecommand\corref[1]{\ref{cor:#1}}}
\providecommand{\tabularnewline}{\\}
\RS@ifundefined{subsecref}
  {\newref{subsec}{name = \RSsectxt}}
  {}
\RS@ifundefined{thmref}
  {\def\RSthmtxt{theorem~}\newref{thm}{name = \RSthmtxt}}
  {}
\RS@ifundefined{lemref}
  {\def\RSlemtxt{lemma~}\newref{lem}{name = \RSlemtxt}}
  {}

\numberwithin{equation}{section}
\numberwithin{figure}{section}
\theoremstyle{plain}
\newtheorem{thm}{\protect\theoremname}[section]
\theoremstyle{plain}
\newtheorem{prop}[thm]{\protect\propositionname}
\theoremstyle{plain}
\newtheorem{cor}[thm]{\protect\corollaryname}

\usepackage{amsmath}
\usepackage{amsthm}
\usepackage{url}
\usepackage{hyperref}

\newref{sec}{name=Section~,Name=Section~}
\newref{subsec}{name=Section~,Name=Section~}
\newref{subsubsec}{name=Section~,Name=Section~}
\newref{eq}{name=Equation~,Name=Equation~}
\newref{prop}{name=Proposition~,Name=Proposition~}
\newref{thm}{name=Theorem~,Name=Theorem~}
\newref{lem}{name=Lemma~,Name=Lemma~}
\newref{cor}{name=Corollary~,Name=Corollary~}
\newref{defn}{name=Definition~,Name=Definition~}
\newref{fig}{name=Figure~,Name=Figure~}

\makeatother

\usepackage{babel}
\providecommand{\corollaryname}{Corollary}
\providecommand{\propositionname}{Proposition}
\providecommand{\theoremname}{Theorem}

\begin{document}
\title{Curvature-Induced Force Fields in Hyperelasticity}
\author{Victor Dods}
\date{2026.01.14}
\maketitle
\begin{abstract}
Originally motivated by creating first-person computer visualizations
within Riemannian manifolds -- the author was led to study deformable-body
mechanics, as rigid-body mechanics is not available in a generic Riemannian
manifold due to its lack of nontrivial isometry group. Hyperelasticity
is a particularly nice sub-category of continuum mechanics in which
a deformable, elastic body's behavior is determined by a stored energy
density function. This allows problems to be posed variationally,
and powerful tools brought to bear on studying and solving them.

This article presents numerical simulations of static solutions to
a particular class of problems in hyperelastic mechanics in 2-dimensional
Riemannian manifolds in which a flat hyperelastic body $B$ is embedded
into a region $\Omega$ in a nowhere-flat surface $S$ of revolution
$z=z\left(r\right)$ such that $\left|K\left(r\right)\right|$ decreases
as $r\to\infty$, where $K$ denotes the Gaussian curvature of $S$.
For example, the funnel $z=-r^{-1}$ or the paraboloid $z=\frac{1}{2}r^{2}$.
Because $B$ is flat, the body can't achieve a zero-stored-energy
configuration, and restorative forces arise in the body to move it
toward a region of lower stored energy -- meaning, toward a flatter
configuration.

With the addition of a gravitational potential $U\left(r\right)=z\left(r\right)$
on $S$, forces act on the body to pull it toward $r=0$. If the hyperelastic
material has sufficient stiffness and the body remains within the
region $\Omega$, then the body has an equilibrium configuration in
which the body's deformation-response forces perfectly cancel the
gravitational forces. Such a configuration represents a kind of \textquotedbl levitation\textquotedbl{}
phenomenon within this surface.

The numerical implementation of this problem will be detailed and
the resulting numerical solutions and various consequences discussed.
All code and data is available at \url{https://github.com/vdods/jello}.
Updates and other relevant additional work will be posted there.
\end{abstract}
\begin{center}
\includegraphics[width=6in]{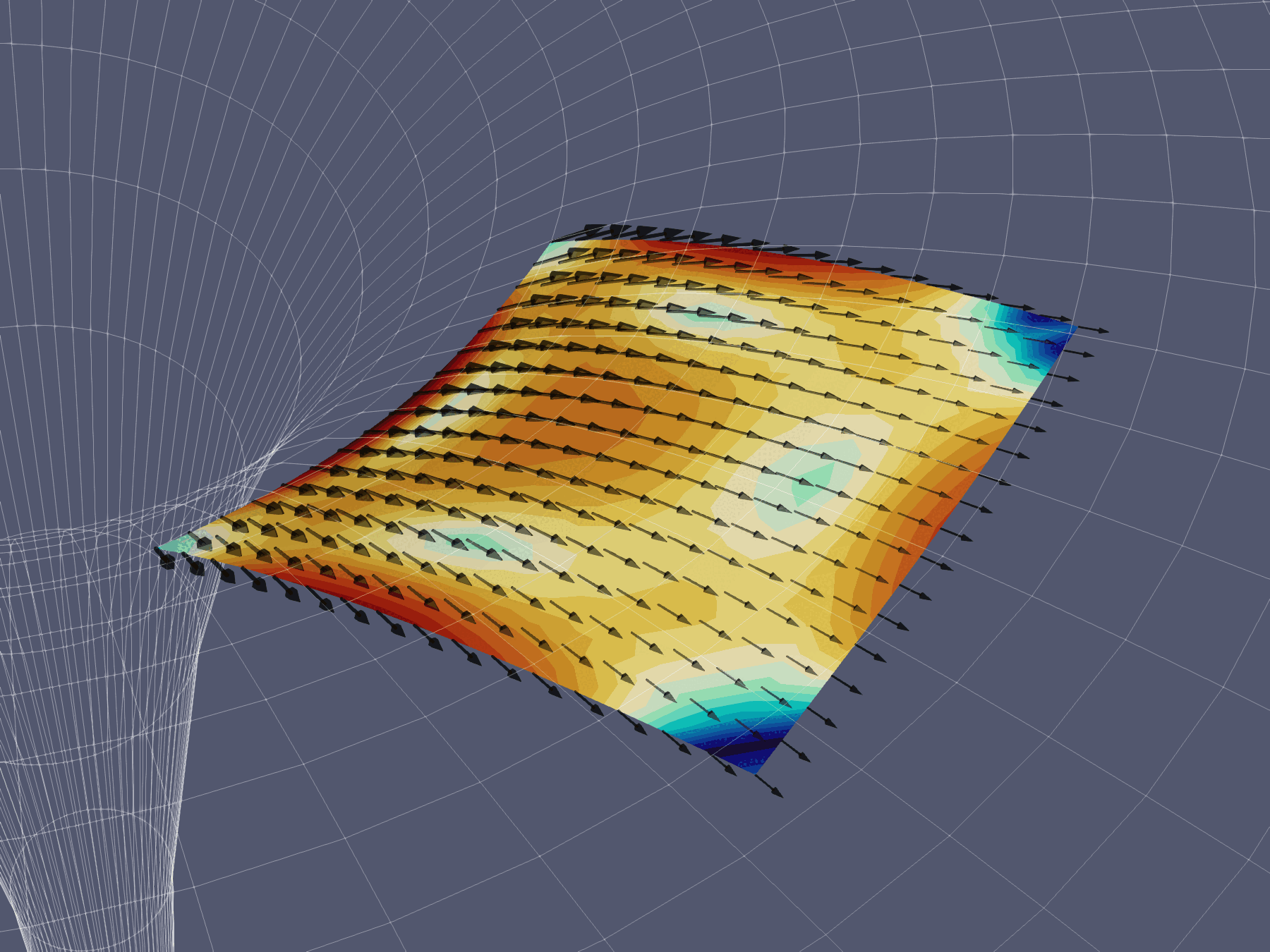}
\par\end{center}

\section{Motivation}

\subsection{Intrinsic Visualization of Curved Space}

The author's motivation for studying the mechanics of hyperelastic
bodies within Riemannian manifolds originated in work in 2003 and
2008 on \textquotedbl first person\textquotedbl{} visualization of
manifolds -- what a subject would see if they were physically located
within a manifold.

Modeling the optics of a Riemannian manifold has a clear generalization
from that of classical optics, in that \textquotedbl photons\textquotedbl{}
follow the manifold's geodesics. Work on this dates back to 1911 with
Einstein predicting gravitational lensing \citet{einstein_1911,einstein_1916}.
Modern work on intrinsic visualization of curved spaces includes visualization
of hyperbolic space \citet{gunn_2002}, constant-curvature spaces
\citet{Weeks2002RealtimeRI}, interactive simulation of relativity
\citet{Savage2006RealTR}, hyperbolic space in virtual reality \citet{HartHawksleyMatsumotoSegerman2017},
and some other non-Euclidean spaces in virtual reality \citet{VelhoDaSilvaNovello2020,NOVELLO2020},
and notably the spectacular visualization of a black hole in the movie
\emph{Interstellar} \citet{Thorne2014}.

Because of the complex and varied character of geodesics, there are
many optical phenomena that occur in Riemannian manifolds that do
not occur in Euclidean spaces, such as the following.
\begin{itemize}
\item Objects may have significantly nonlinear optical magnification/minification.
For example, within a sphere, an object located at the antipodal point
relative to the subject will produce an image that covers the subject's
entire visual field.
\item Each point on the surface of a single object may have multiple images
in the subject's visual field. For example, in a flat torus, a single
object will produce a lattice of images within the subject's visual
field.
\item Caustics (in the sense of optics) may be produced due to the curvature
of the manifold (i.e. not from a surface of reflection or refraction),
leading to the concentration and dilution of light.
\item Photons may have a closed trajectory (in the sense of closed geodesics).
In particular, this means that the subject may be able to see the
back of their own head/body. Photons may return to their initial positions
but not their initial directions, meaning that the subject may be
able to see any part of their own head/body.
\item If the manifold is non-orientable, then objects may appear\footnote{Chiral reversal will also be true of the objects themselves, if they
traverse certain loops in the non-orientable manifold.} chirally reversed from some perspectives.
\end{itemize}
\begin{figure}[H]
\begin{centering}
\includegraphics[totalheight=0.25\textwidth]{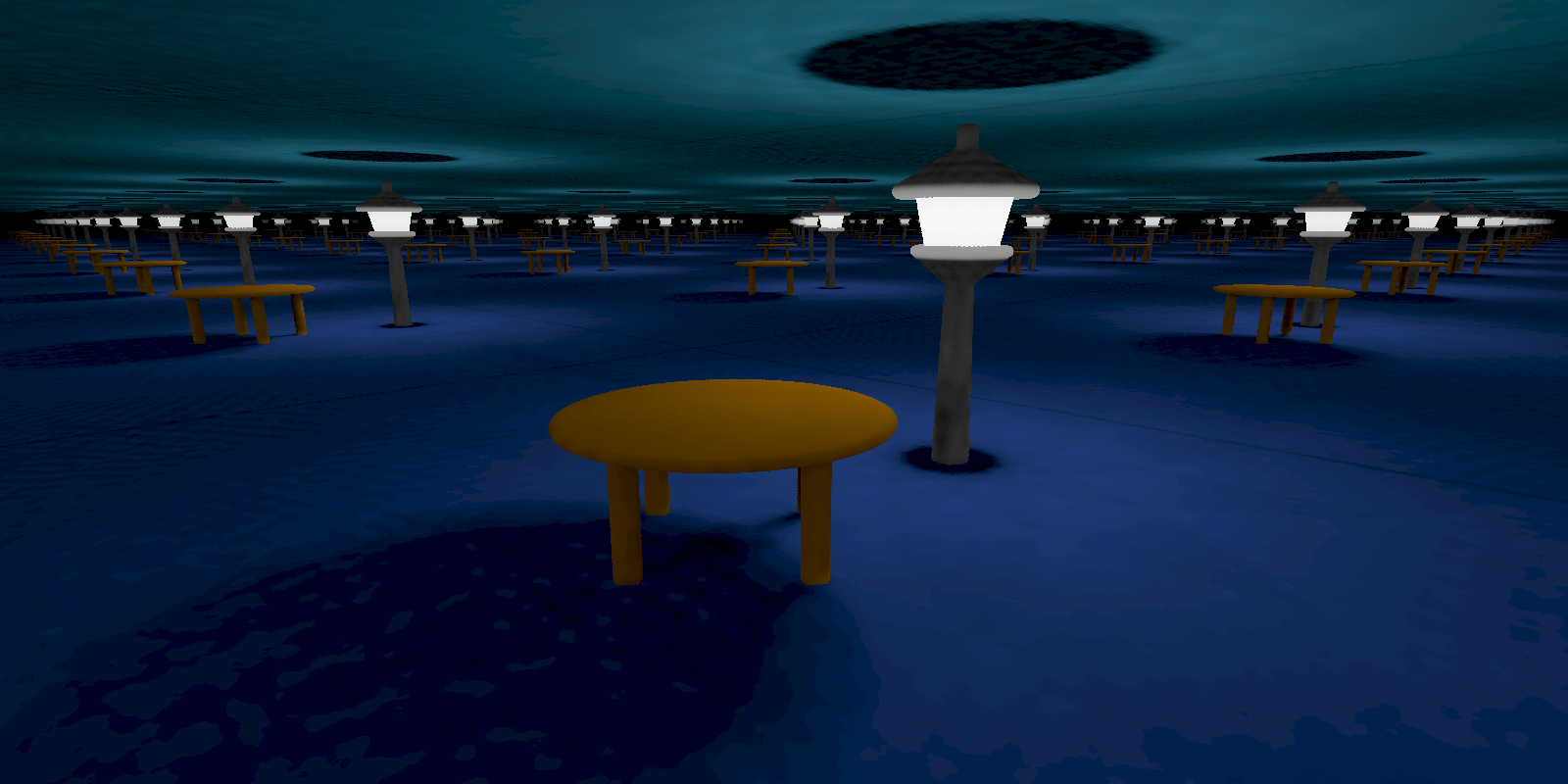}
\includegraphics[totalheight=0.25\textwidth]{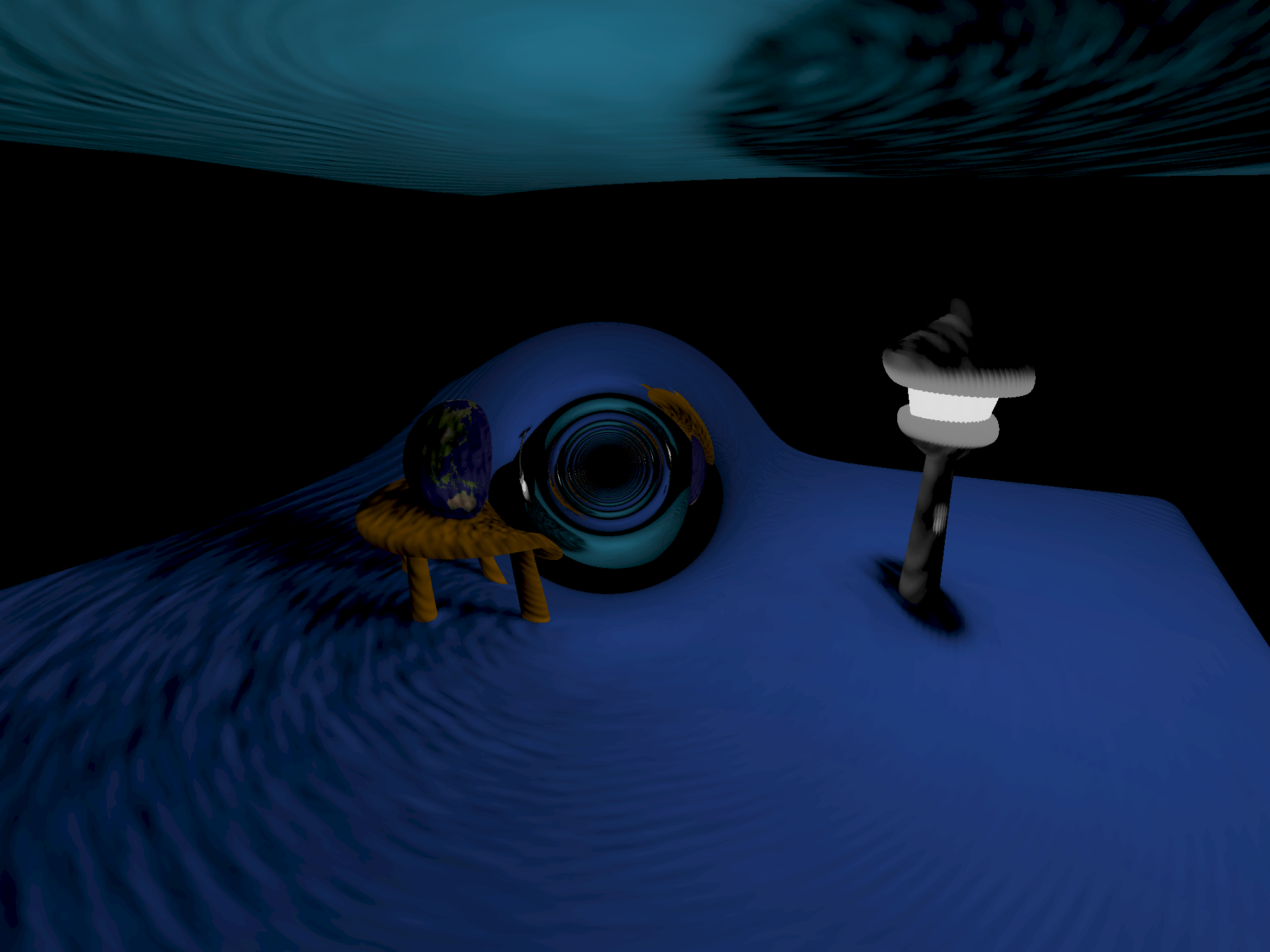}
\par\end{centering}
\caption{Images made using the author's 2008 curved-space ray-tracer \textquotedbl Einstein's
Eye\textquotedbl . It used voxels to model objects and used a finite
sampling of photon trajectories to model lighting. Left: A single
lamp and table in a flat torus, thereby appearing in a lattice of
images. Right: A lamp, table, and globe, with a funnel-shaped spatial
distortion in the center of the room.}
\end{figure}

\subsection{Vision Is Boring If There's Nothing To See}

Producing a subject's first person view from within a manifold is
not particularly interesting if there's nothing to see. While certainly
one could visualize various mathematical structures within the manifold
-- a deep and interesting pursuit in its own right -- there is a
natural motivation to have \textquotedbl physically realistic\textquotedbl{}
objects within the manifold to visualize. For the author, this motivation
came from a background and interest in video game development. The
following are two approaches for generalizing conventional object
representation to Riemannian manifolds.
\begin{itemize}
\item \textbf{Voxels} could be described as \textquotedbl volumetric pixels\textquotedbl ,
and can be conglomerated to form discretizations of arbitrary continuous
objects. In the usual Euclidean setting, a voxel is usually a tiny
cube or sphere. Because the concept of regular polytope that tiles
space has no natural generalization to a generic Riemannian manifold,
using a sphere-shaped voxel is the reasonable choice. However, computing
the precise shape of a geodesic sphere is nontrivial and computationally
expensive. Using a local approximation of a tiny geodesic sphere in
the form of a coordinate ellipsoid is a reasonable compromise, but
this approach does cause the specifics of the coordinate charts to
\textquotedbl leak through\textquotedbl{} to the visualization, and
worse yet, if a voxel transitions from one chart to another, there
will be a discontinuous change in its shape.
\item \textbf{Models constructed from meshes} makes up much of the object
geometry used in computer graphics (in particular, triangulated surfaces).
These models could be placed in the manifold using \textquotedbl non-physical\textquotedbl ,
but still geometrically meaningful, embeddings. For example, mapping
the model into the manifold using the Riemannian exponential. The
origin and orientation of the Riemannian exponential map and the scale
of the model would be the configuration parameters. However, this
approach also suffers from the discontinuous change in shape as it
crosses coordinate chart transitions, because the specific shape of
the surface elements defined by the mesh (e.g. triangles) will be
coordinate-dependent. Furthermore, this approach would produce degenerate,
self-overlapping, and orientation-reversed regions in the mapped model
if it includes a conjugate point of the center.
\end{itemize}
Another approach would be to use \textquotedbl physical\textquotedbl{}
embeddings of objects (either defined in Euclidean space, or in the
manifold itself).

\subsection{What Does \textquotedbl Physical\textquotedbl{} Mean In This Context?}

Euclidean space enjoys a high degree of symmetry, in that its isometry
group is comprised of translations and rotations. This allows for
a meaningful definition of \textquotedbl rigid transformation\textquotedbl ,
and therefore also of \textquotedbl rigid motion\textquotedbl .
From this isometry group, Noether's theorem provides the conservation
laws for linear and angular momenta, giving a tractable and (relatively)
simple subcategory of body mechanics called \textquotedbl rigid body
mechanics\textquotedbl .

A generic Riemannian manifold has a trivial isometry group, and therefore
does not enjoy any kind of rigid body mechanics. Thus it is necessary
to consider the more general category of \textquotedbl continuum
mechanics\textquotedbl . Because this theory is so vast, it would
be wise to consider a more tractable sub-category. Here is a coarse
map of some of the categories within continuum mechanics, each arrow
denoting an inclusion into a more general category.
\[
\begin{array}{ccccc}
\text{continuum mechanics} & \leftarrow & \text{solid mechanics} & \leftarrow & \text{elastic mechanics}\\
\uparrow &  & \uparrow &  & \uparrow\\
\text{fluid mechanics} &  & \text{plastic mechanics} &  & \text{\textbf{hyperelastic mechanics}}
\end{array}
\]

Continuum mechanics is the analysis of the physics of continuous bodies
(i.e. not particle systems). Solid and fluid mechanics are special
cases of continuum mechanics which are distinguished by a body having
a preferred \textquotedblleft rest shape\textquotedblright{} (a solid
does, while a fluid does not, have a rest shape). Solid mechanics
is further specialized into the study of materials which accrue permanent
deformation (plastic) and materials which tend to revert to their
fixed rest shape (elastic). Finally, hyperelastic mechanics is a particular
sub-category of elastic mechanics in which the elastic behavior is
uniquely determined by a stored energy function.

\section{Hyperelastic Mechanics in Riemannian Manifolds}

By construction, hyperelastic mechanics is amenable to use of the
calculus of variations. Various covariant forms of the calculus of
variations have been presented in \citet{Abraham&Marsden}, \citet{Marsden&Hughes},
and \citet{Dods2012,Dods2022}, the latter of which originated the
specific decomposition of $TE\to E$, corresponding definition of
partial covariant derivatives and their use in a Riemannian Calculus
of Variations, showing that the partial covariant derivatives correctly
generalize the ordinary partial derivatives of the Calculus of Variations
on vector spaces.

\subsection{Riemannian Calculus of Variations\protect\label{subsec:Riemannian-Calculus-of-Variations}}

A terse summary of Riemannian Calculus of Variations will be given
here. Details can be found in \citet{Dods2022}. For the reader's
convenience, proofs of the formulas given here can be found in \secref{Proofs}.

While \textquotedbl body\textquotedbl , \textquotedbl space\textquotedbl ,
\textquotedbl body configuration\textquotedbl , and \textquotedbl deformation
tensor\textquotedbl{} are terms specific to continuum mechanics, the
constructions given here are fully general for Lagrangians depending
on the 1-jet of a map between Riemannian manifolds, which is the proper
generalization of first-order Lagrangians in the standard calculus
of variations.

\subsubsection{Preliminaries}

Let $\left(B,g\right)$ and $\left(S,h\right)$ be Riemannian manifolds
of equal dimension, with $B$ compact, having a single connected component.
$\left(B,g\right)$ and $\left(S,h\right)$ represent the rest/reference
configuration of a continuous body, and the space in which it is embedded,
respectively. Let $d\mu_{B}$ and $d\mu_{\partial B}$ denote the
Riemannian volume forms on $\left(B,g\right)$ and $\left(\partial B,\overline{g}\right)$
respectively, where $\overline{g}$ is the metric induced on $\partial B$
by inclusion. A \textbf{body configuration} is an embedding\footnote{This article really only considers immersions, as detecting and preventing
non-local self-intersection of a body is a nontrivial problem whose
solution isn't needed here.}
\begin{align*}
\phi & \in C^{1}\left(B,S\right)
\end{align*}
defining where in space each body point is mapped. Other regularity
classes, such as $C^{2}\left(B,S\right)$ or $H^{1}\left(B,S\right)$,
are considered when appropriate. A \textbf{deformation tensor} is
an element of the vector bundle
\begin{align*}
\pi\colon TS\otimes T^{*}B & \to S\times B.
\end{align*}
An element of $TS\otimes T^{*}B$ is sometimes called a two-point
tensor \citet[pg. 70]{Marsden&Hughes}, given that it has \textquotedbl one
leg\textquotedbl{} at each of two points (in this case, points in
two different manifolds). For brevity, let $E:=TS\otimes T^{*}B$.
A body's \textbf{deformation tensor field} is
\begin{align*}
\nabla^{B\to S}\phi\colon B & \to\phi^{*}TS\otimes T^{*}B
\end{align*}
which is simply the tangent map $T\phi\colon TB\to TS$ expressed
as a tensor field. Note that for each $\phi$, $\phi^{*}TS\otimes T^{*}B$
is a dependent type that includes naturally into $E$, so $\nabla^{B\to S}\phi$
can also be considered as a map of non-dependent type $B\to E$. Define
\textbf{Lagrangian density}
\begin{align*}
L\colon E & \to\mathbb{R}
\end{align*}
and corresponding \textbf{Lagrangian} (action functional)
\begin{align}
\mathcal{L}\left(\phi\right) & :=\int_{B}L\circ\nabla^{B\to S}\phi\,d\mu_{B}.\label{eq:lagrangian-action-functional}
\end{align}

Let
\begin{align*}
\pi_{S} & :=\pi_{S}^{S\times B}\circ\pi\colon E\to S, & \pi_{B} & :=\pi_{B}^{S\times B}\circ\pi\colon E\to B,\\
\sigma & :=\nabla^{E\to S}\pi_{S}\in\Gamma\left(\pi_{S}^{*}TS\otimes T^{*}E\right), & \beta & :=\nabla^{E\to B}\pi_{B}\in\Gamma\left(\pi_{B}^{*}TB\otimes T^{*}E\right).
\end{align*}
The covariant derivatives $\nabla^{TS}$ and $\nabla^{TB}$ naturally
induce a covariant derivative $\nabla^{E}$, which in turn produces
a vertical \textquotedbl projection\textquotedbl
\begin{align*}
\Phi & \in\Gamma\left(\pi^{*}E\otimes T^{*}E\right),\text{ defined by} & e^{*}\Phi\cdot e^{\prime}\left(0\right) & =\nabla_{\frac{d}{dt}\mid_{t=0}}^{\left(\pi\circ e\right)^{*}E}e,
\end{align*}
where $e\colon\left(-\epsilon,\epsilon\right)\to E$ is any $C^{1}$
curve representing the tangent vector $e^{\prime}\left(0\right)\in T_{e\left(0\right)}E$.
Effectively, $\Phi$ turns an ordinary derivative into a covariant
one. Then
\begin{align*}
\sigma\oplus\beta\oplus\Phi\colon TE & \to\pi_{S}^{*}TS\oplus\pi_{B}^{*}TB\oplus\pi^{*}E
\end{align*}
is a vector bundle isomorphism which splits $TE$ into three sub-bundles
\begin{align*}
\ker\left(\beta\oplus\Phi\right) & \cong\pi_{S}^{*}TS, & \ker\left(\sigma\oplus\Phi\right) & \cong\pi_{B}^{*}TB, & \ker\left(\sigma\oplus\beta\right) & \cong\pi^{*}E
\end{align*}
These bundles represent, respectively, variations along the spatial
manifold, along the body manifold, and (vertically) along deformation
gradients. Formalizing this, the differential $dL\in\Gamma\left(T^{*}E\right)$
can be decomposed into the \textbf{partial covariant derivatives}
\begin{align*}
L_{,\sigma} & \in\Gamma\left(\pi_{S}^{*}T^{*}S\right), & L_{,\beta} & \in\Gamma\left(\pi_{B}^{*}T^{*}B\right), & L_{,\Phi} & \in\Gamma\left(\pi^{*}E^{*}\right)=\Gamma\left(\pi_{S}^{*}T^{*}S\otimes\pi_{B}TB\right)
\end{align*}
 of $L$, which are defined uniquely by
\begin{align*}
dL & =L_{,\sigma}\cdot\sigma+L_{,\beta}\cdot\beta+L_{,\Phi}:\Phi,
\end{align*}
where $:$ was used in the last term because $E=TS\otimes T^{*}B$
has two factors. $L_{,\Phi}$ is just the fiber derivative of $L$,
making it relatively easy to compute. Computing $L_{,\sigma}$ involves
pairing $dL$ against $e^{\prime}\left(0\right)\in T_{e\left(0\right)}E$,
which is a variation of a curve $e\colon\left(-\epsilon,\epsilon\right)\to E$
such that $0=e^{*}\beta\cdot e^{\prime}\left(0\right)=\left(\pi_{B}\circ e\right)^{\prime}\left(0\right)$
and $0=e^{*}\Phi\cdot e^{\prime}\left(0\right)=\nabla_{\frac{d}{dt}\mid_{t=0}}^{\left(\pi\circ e\right)^{*}E}e$.
Computing $L_{,\beta}$ is analogous.

\subsubsection{First and Second Variations}

Proofs for the claims in this section can be found in \subsecref{Proofs-for-RCOV}.

In the following, assume maps have sufficient regularity for the expressions
involving derivatives to be well-defined. Let $\psi\in\Gamma\left(\phi^{*}TS\right)$
(a vector field along $\phi$) denote a variation of $\phi$. For
brevity, let\footnote{The specific type analysis here is $\left(\nabla\phi\right)^{*}L_{,\sigma}\in\Gamma\left(\left(\nabla\phi\right)^{*}\pi_{S}^{*}T^{*}S\right)\cong\Gamma\left(\left(\pi_{S}\circ\nabla\phi\right)^{*}T^{*}S\right)=\Gamma\left(\phi^{*}T^{*}S\right)$.
Similar for $\left(\nabla\phi\right)^{*}L_{,\Phi}$.}
\begin{align*}
L_{,\sigma}^{\nabla\phi} & :=\left(\nabla\phi\right)^{*}L_{,\sigma}\in\Gamma\left(\phi^{*}T^{*}S\right), & L_{,\Phi}^{\nabla\phi} & :=\left(\nabla\phi\right)^{*}L_{,\Phi}\in\Gamma\left(\phi^{*}T^{*}S\otimes TB\right).
\end{align*}

\begin{prop}
\label{prop:first-variation-weak-form}The \textbf{first variation
of $\mathcal{L}$ in weak form} is
\begin{align*}
D\mathcal{L}\left(\phi\right)\cdot\psi & =\int_{B}L_{,\sigma}^{\nabla\phi}\cdot\psi+L_{,\Phi}^{\nabla\phi}:\nabla\psi\,d\mu_{B}.
\end{align*}
\end{prop}

\begin{prop}
\label{prop:first-variation-bulk-boundary-form}The \textbf{first
variation of $\mathcal{L}$ in bulk + boundary form} is
\begin{align*}
D\mathcal{L}\left(\phi\right)\cdot\psi & =\int_{B}\left(L_{,\sigma}^{\nabla\phi}-\text{\emph{div}}L_{,\Phi}^{\nabla\phi}\right)\cdot\psi\,d\mu_{B}+\int_{\partial B}\left(\iota^{*}L_{,\Phi}^{\nabla\phi}\cdot\nu\right)\cdot\iota^{*}\psi\,d\mu_{\partial B},
\end{align*}
where $\iota\colon\partial B\to B$ is the inclusion and $\nu\in\Gamma\left(\iota^{*}T^{*}B\right)$
is the outward unit conormal field on $\partial B$. Note that $\iota^{*}L_{,\Phi}^{\nabla\phi}\in\Gamma\left(\iota^{*}\phi^{*}T^{*}S\otimes\iota^{*}TB\right)$
and $\iota^{*}\psi\in\Gamma\left(\iota^{*}\phi^{*}TS\right)$, which
are the restrictions of $L_{,\Phi}^{\nabla\phi}$ and $\psi$ to $\partial B$.
\end{prop}

The partial covariant derivative allows the covariant Hessian $\nabla^{2}L\in\Gamma\left(T^{*}E\otimes T^{*}E\right)$
to be decomposed uniquely into a \textquotedbl matrix\textquotedbl{}
of second partial covariant derivatives. In the following, the permutation
superscripts indicate permutation of the tensor factors. For example,
recalling that $\sigma\in\Gamma\left(\pi_{S}^{*}TS\otimes T^{*}E\right)$,
it follows that $\sigma^{\left(1\,2\right)}\in\Gamma\left(T^{*}E\otimes\pi_{S}^{*}TS\right)$.
\begin{align*}
\nabla^{2}L & =\left[\begin{array}{ccc}
\sigma^{\left(1\,2\right)}\cdot & \beta^{\left(1\,2\right)}\cdot & \Phi^{\left(1\,3\right)\left(2\,4\right)}:\end{array}\right]\left[\begin{array}{ccc}
L_{,\sigma\sigma} & L_{,\sigma\beta} & L_{,\sigma\Phi}\\
L_{,\beta\sigma} & L_{,\beta\beta} & L_{,\beta\Phi}\\
L_{,\Phi\sigma} & L_{,\Phi\beta} & L_{,\Phi\Phi}
\end{array}\right]\left[\begin{array}{c}
\cdot\sigma\\
\cdot\beta\\
:\Phi
\end{array}\right]
\end{align*}
which, due to the symmetry of $\nabla^{2}L$, enjoys symmetries
\begin{align*}
\left(L_{,\sigma\sigma}\right)^{ij} & =\left(L_{,\sigma\sigma}\right)^{ji}, & \left(L_{,\sigma\beta}\right)^{ij} & =\left(L_{,\beta\sigma}\right)^{ji}, & \left(L_{,\sigma\Phi}\right)^{ijk} & =\left(L_{,\Phi\sigma}\right)^{jki},\\
 &  & \left(L_{,\beta\beta}\right)^{ij} & =\left(L_{,\beta\beta}\right)^{ji}, & \left(L_{,\beta\Phi}\right)^{ijk} & =\left(L_{,\Phi\beta}\right)^{jki},\\
 &  &  &  & \left(L_{,\Phi\Phi}\right)^{ijk\ell} & =\left(L_{,\Phi\Phi}\right)^{k\ell ij}.
\end{align*}
The iterated subscript notation is justified by the fact that the
second partial covariant derivatives can also be computed directly
from the first partial covariant derivatives. For example, $\nabla L_{,\sigma}\in\Gamma\left(\pi_{S}^{*}T^{*}S\otimes T^{*}E\right)$
decomposes uniquely as
\begin{align*}
\nabla L_{,\sigma} & =L_{,\sigma\sigma}\cdot\sigma+L_{,\sigma\beta}\cdot\beta+L_{,\sigma\Phi}:\Phi.
\end{align*}
Finally, for brevity, let
\begin{align*}
L_{,\sigma\sigma}^{\nabla\phi} & :=\left(\nabla\phi\right)^{*}L_{,\sigma\sigma}\in\Gamma\left(\phi^{*}T^{*}S\otimes\phi^{*}T^{*}S\right), & L_{,\sigma\Phi}^{\nabla\phi} & :=\left(\nabla\phi\right)^{*}L_{,\sigma\Phi}\in\Gamma\left(\phi^{*}T^{*}S\otimes\phi^{*}T^{*}S\otimes TB\right),\\
L_{,\Phi\sigma}^{\nabla\phi} & :=\left(\nabla\phi\right)^{*}L_{,\Phi\sigma}\in\Gamma\left(\phi^{*}T^{*}S\otimes TB\otimes\phi^{*}T^{*}S\right), & L_{,\Phi\Phi}^{\nabla\phi} & :=\left(\nabla\phi\right)^{*}L_{,\Phi\Phi}\in\Gamma\left(\phi^{*}T^{*}S\otimes TB\otimes\phi^{*}T^{*}S\otimes TB\right).
\end{align*}

\begin{prop}
\label{prop:second-variation-weak-form}If $D\mathcal{L}\left(\phi\right)=0$,
then the \textbf{second variation of $\mathcal{L}$ in weak form}
is
\begin{align*}
\nabla^{2}\mathcal{L}\left(\phi\right):\left(\psi\otimes\psi\right) & =\int_{B}\left[\begin{array}{cc}
\psi\cdot & \nabla\psi:\end{array}\right]\left[\begin{array}{cc}
L_{,\sigma\sigma}^{\nabla\phi} & L_{,\sigma\Phi}^{\nabla\phi}\\
L_{,\Phi\sigma}^{\nabla\phi} & L_{,\Phi\Phi}^{\nabla\phi}
\end{array}\right]\left[\begin{array}{c}
\cdot\psi\\
:\nabla\psi
\end{array}\right]+L_{,\Phi}^{\nabla\phi}:\left(\phi^{*}R^{TS}\vdots\left(\psi\otimes\psi\otimes\nabla\phi\right)\right)\,d\mu_{B},
\end{align*}
where $\nabla^{2}\mathcal{L}:=\nabla D\mathcal{L}$ is the covariant
Hessian of $\mathcal{L}$ (see \citet{Eliasson}) and
\begin{align*}
R^{TS} & \in\Gamma\left(TS\otimes T^{*}S\otimes T^{*}S\otimes T^{*}S\right)
\end{align*}
is the Riemannian curvature tensor on $TS$ using the conventions
\begin{align*}
\omega\cdot R^{TS}\vdots\left(Z\otimes X\otimes Y\right) & =\omega\cdot R_{\text{op}}^{TS}\left(X,Y\right)Z,\\
R_{\text{op}}^{TS}\left(X,Y\right) & :=\nabla_{X}^{TS}\nabla_{Y}^{TS}-\nabla_{Y}^{TS}\nabla_{X}^{TS}-\nabla_{\left[X,Y\right]}^{TS},
\end{align*}
for all $\omega\in\Gamma\left(T^{*}S\right)$ and $X,Y,Z\in\Gamma\left(TS\right)$.
\end{prop}

\subsubsection{Symmetries of the Lagrangian}

Let $G_{t}\colon S\to S$ denote a one-parameter family of maps, and
let $\phi_{t}:=G_{t}\circ\phi$. If $L\circ\nabla\phi_{t}=L\circ\nabla\phi$
for all $t$, then $\mathcal{L}\left(\phi_{t}\right)=\mathcal{L}\left(\phi\right)$,
showing that $G_{t}$ describes a symmetry of the problem, i.e.
\begin{align*}
0 & =\frac{d}{dt}\mathcal{L}\left(\phi\right)\mid_{t=0}=\frac{d}{dt}\mathcal{L}\left(\phi_{t}\right)\mid_{t=0}=D\mathcal{L}\left(\phi\right)\cdot\frac{d}{dt}\phi_{t}\mid_{t=0}.
\end{align*}
In particular,
\begin{align*}
\frac{d}{dt}\phi_{t}\mid_{t=0} & =\left(\frac{d}{dt}G_{t}\mid_{t=0}\right)\circ\phi\in\Gamma\left(\phi^{*}TS\right)
\end{align*}
is a vector field along $\phi$ that is the infinitesimal generator
of the symmetry at $\phi$, and $\frac{d}{dt}G_{t}\mid_{t=0}$ is
a map of the form $S\to TS$ that can be used to quantify the symmetries
of the problem in terms of the spatial manifold alone.

Now assume that $\phi$ is a critical point of $\mathcal{L}$, i.e.
$D\mathcal{L}\left(\phi\right)=0$. Let $\psi:=\frac{d}{dt}\phi_{t}\mid_{t=0}$
be an infinitesimal generator of the symmetry at $\phi$. Then
\begin{align*}
0 & =\frac{d}{dt}\frac{d}{ds}\mathcal{L}\left(\phi\right)\mid_{s=0}\mid_{t=0}\\
 & =\frac{d}{dt}\frac{d}{ds}\mathcal{L}\left(\phi_{t+s}\right)\mid_{s=0}\mid_{t=0}\\
 & =\frac{d}{dt}D\mathcal{L}\left(\phi_{t}\right)\cdot\left(\frac{d}{ds}\phi_{t+s}\mid_{s=0}\right)\mid_{t=0}\\
 & =\nabla^{2}\mathcal{L}\left(\phi\right):\left(\left(\frac{d}{ds}\phi_{s}\mid_{s=0}\right)\otimes\left(\frac{d}{dt}\phi_{t}\mid_{t=0}\right)\right)+D\mathcal{L}\left(\phi\right)\cdot\nabla_{\frac{d}{dt}}\left(\frac{d}{ds}\phi_{t+s}\mid_{s=0}\right)\mid_{t=0}\\
 & =\nabla^{2}\mathcal{L}\left(\phi\right):\left(\psi\otimes\psi\right),
\end{align*}
showing that $\psi$ is an eigenvector of $\nabla^{2}\mathcal{L}\left(\phi\right)$
with eigenvalue 0, as expected. Thus the stability analysis for an
equilibrium solution $\phi$ should exclude the subspace generated
by infinitesimal symmetries of the problem.

\subsection{Hyperelastic Mechanics\protect\label{subsec:Hyperelastic-Mechanics}}

As a sub-category of continuum mechanics, hyperelastic mechanics is
defined by the existence of a stored energy density function that
uniquely determines the body's response to deformation. This makes
it amenable to a variational formulation.

This section will apply the Riemannian Calculus of Variations to pose
a general form of the static problem addressed by this article. The
concrete problem follows by making particular choices for the various
manifolds and functions involved. As before, let $\phi\colon B\to S$
be a body configuration, having as much regularity as is needed for
the relevant expressions.

\subsubsection{A Particular Lagrangian\protect\label{subsec:A-Particular-Lagrangian}}

Let the \textbf{stored energy density} of the hyperelastic body be
\begin{align*}
W\colon E & \to\mathbb{R}.
\end{align*}
Consider spatial forces arising due to a \textbf{mass-specific potential}
(e.g. gravitational potential)
\begin{align*}
U\colon S & \to\mathbb{R},
\end{align*}
defined throughout space, that couples with the body's \textbf{mass
density}
\begin{align*}
\rho\colon B & \to\mathbb{R}_{+}
\end{align*}
to produce the body's \textbf{potential energy density}
\begin{align*}
V\colon S\times B & \to\mathbb{R},\\
\left(s,b\right) & \mapsto U\left(s\right)\rho\left(b\right),
\end{align*}
where the pair $\left(s,b\right)$ denotes a specific mapping of a
body point $b\in B$ to a spatial point $s\in S$. Define the Lagrangian
density to be the sum of these energy densities, i.e.
\begin{align*}
L\colon E & \to\mathbb{R},\\
F & \mapsto W\left(F\right)+V\left(\pi\left(F\right)\right).
\end{align*}
This gives the Lagrangian (action functional)
\begin{align*}
\mathcal{L}\left(\phi\right) & :=\int_{B}L\circ\nabla^{B\to S}\phi\,d\mu_{B}
\end{align*}
just as defined in \eqref{lagrangian-action-functional}.

The formulas for the variations of this specific Lagrangian are straightforward
corollaries of the general formulas. Proofs for these claims can be
found in \subsecref{Proofs-for-Hyperelastic-Mechanics}. Let $\psi\in\Gamma\left(\phi^{*}TS\right)$
be a variation of body configuration $\phi\colon B\to S$.
\begin{cor}[First variation of specific $\mathcal{L}$ in weak form]
\label{cor:mechanical-first-variation-in-weak-form}
\begin{align*}
D\mathcal{L}\left(\phi\right)\cdot\psi & =\int_{B}\left(W_{,\sigma}^{\nabla\phi}+\rho\phi^{*}dU\right)\cdot\psi+W_{,\Phi}^{\nabla\phi}:\nabla\psi\,d\mu_{B}.
\end{align*}
\end{cor}

\begin{cor}[First variation of specific $\mathcal{L}$ in bulk + boundary form]
\label{cor:mechanical-first-variation-in-bulk-boundary-form} 
\begin{align*}
D\mathcal{L}\left(\phi\right)\cdot\psi & =\int_{B}\left(W_{,\sigma}^{\nabla\phi}+\rho\phi^{*}dU-\text{div}W_{,\Phi}^{\nabla\phi}\right)\cdot\psi\,d\mu_{B}+\int_{\partial B}\left(\iota^{*}W_{,\Phi}^{\nabla\phi}\cdot\nu\right)\cdot\iota^{*}\psi\,d\mu_{\partial B}.
\end{align*}
\end{cor}

\begin{cor}[Second variation of specific $\mathcal{L}$ in weak form]
\label{cor:mechanical-second-variation-in-weak-form}
\begin{align*}
\nabla^{2}\mathcal{L}\left(\phi\right):\left(\psi\otimes\psi\right) & =\int_{B}\left[\begin{array}{cc}
\psi\cdot & \nabla\psi:\end{array}\right]\left[\begin{array}{cc}
W_{,\sigma\sigma}^{\nabla\phi}+\rho\phi^{*}\nabla^{2}U & W_{,\sigma\Phi}^{\nabla\phi}\\
W_{,\Phi\sigma}^{\nabla\phi} & W_{,\Phi\Phi}^{\nabla\phi}
\end{array}\right]\left[\begin{array}{c}
\cdot\psi\\
:\nabla\psi
\end{array}\right]+W_{,\Phi}^{\nabla\phi}:\left(\phi^{*}R^{TS}\vdots\left(\psi\otimes\psi\otimes\nabla\phi\right)\right)\,d\mu_{B}
\end{align*}
\end{cor}

\subsubsection{Spatially Invariant and Homogeneous Materials}

Because $\sigma$ and $\beta$ encode the \textquotedbl along the
spatial manifold\textquotedbl{} and \textquotedbl along the body
manifold\textquotedbl{} sub-bundles of $TE$, the conditions $W_{,\sigma}\equiv0$
and $W_{,\beta}\equiv0$ define a \textbf{spatially invariant} material
and a \textbf{homogeneous} material, respectively. While a physically
realistic $W$ will necessarily be defined in terms of the metrics
on $\left(B,g\right)$ and $\left(S,h\right)$, this is not a-priori
an obstacle to the conditions $W_{,\sigma}\equiv0$ and $W_{,\beta}\equiv0$
due to metric compatibility and the fact that $W_{,\sigma}$ and $W_{,\beta}$
are partial covariant derivatives. It is certainly possible to define
stored energy densities that represent materials that are spatially
non-invariant (e.g. due to a charge-specific potential) or non-homogeneous
(e.g. due to a non-uniform mass density).

As a special case, if it is assumed that $W_{,\sigma}\equiv0$, then
the first variation of $\mathcal{L}$ in weak form is
\begin{align*}
D\mathcal{L}\left(\phi\right)\cdot\psi & =\int_{B}\rho\phi^{*}dU\cdot\psi+W_{,\Phi}^{\nabla\phi}:\nabla\psi\,d\mu_{B}
\end{align*}
and the second variation of $\mathcal{L}$ in weak form is

\begin{align*}
\nabla^{2}\mathcal{L}\left(\phi\right):\left(\psi\otimes\psi\right) & =\int_{B}\psi\cdot\left(\phi^{*}\nabla^{2}U\right)\cdot\psi+\nabla\psi:W_{,\Phi\Phi}^{\nabla\phi}:\nabla\psi+W_{,\Phi}^{\nabla\phi}:\left(\phi^{*}R^{TS}\vdots\left(\psi\otimes\psi\otimes\nabla\phi\right)\right)\,d\mu_{B}.
\end{align*}

\subsubsection{A Physically Reasonable Material\protect\label{subsec:A-Physically-Reasonable-Material}}

Recall that $g\in\Gamma\left(T^{*}B\otimes T^{*}B\right)$ and $h\in\Gamma\left(T^{*}S\otimes T^{*}S\right)$
are the Riemannian metrics on $B$ and $S$, respectively, and $g^{-1}\in\Gamma\left(TB\otimes TB\right)$
is the inverse metric on $B$. Let
\begin{align*}
C\colon TS\otimes T^{*}B & \to TB\otimes T^{*}B,\\
F & \mapsto g^{-1}\left(\pi_{B}\left(F\right)\right)\cdot F^{\left(1\,2\right)}\cdot h\left(\pi_{S}\left(F\right)\right)\cdot F,
\end{align*}
noting that the permutation superscript indicates permutation of tensor
factors. In particular, $F^{\left(1\,2\right)}\in T^{*}B\otimes TS$
is the tensor transpose of $F$. Then $C\left(F\right)$ is the Cauchy
strain tensor corresponding to deformation tensor $F$. By construction,
$C\left(F\right)$ is symmetric and positive-semidefinite (positive-definite
when $F$ has trivial kerel), and quantifies the squared local changes
in distance due to the deformation represented by $F$. If the principal
stretches of $F$ (obtained by its polar decomposition) are $\lambda_{1},\dots,\lambda_{n}$,
then the eigenvalues of $C\left(F\right)$ are $\lambda_{1}^{2},\dots,\lambda_{n}^{2}$.
A hyperelastic material that is frame-indifferent and isotropic has
an energy density function that factors through $C$ and furthermore,
is a function of the tensor invariants of $C\left(F\right)$ -- the
first and last of which are $\text{tr}C\left(F\right)$ and $\det C\left(F\right)$.

Following \citet{lehmich_neff_lankeit_2013}, the \textbf{uni-constant
compressible Neo-Hooke material} is defined by

\begin{align*}
W\left(F\right) & :=\alpha\left(\text{tr}\left(C\left(F\right)-I\right)-\log\det C\left(F\right)\right),
\end{align*}
where $\alpha>0$ is the \textquotedbl shear modulus\textquotedbl{}
or \textquotedbl modulus of rigidity\textquotedbl . The function
$C\mapsto\alpha\left(\text{tr}\left(C-I\right)-\log\det C\right)$
is convex in $C$, but $W$ is only polyconvex in $F$, a notion introduced
in \citet{ball_1976}. This material is \textquotedbl physically
reasonable\textquotedbl{} in the sense that forces arise within the
body to restore it to its rest configuration (at which the forces
are zero), that those forces increase with the magnitude of deformation,
and that the forces become infinite when compression or elongation
becomes infinite.

In terms of the principal stretches $\lambda_{1},\dots,\lambda_{n}$
of $F$,
\begin{align*}
W\left(F\right) & =\overline{W}\left(\lambda_{1},\dots,\lambda_{n}\right):=\alpha\left(w\left(\lambda_{1}^{2}\right)+\dots+w\left(\lambda_{n}^{2}\right)\right),
\end{align*}
where $w\left(\lambda\right):=\lambda-1-\log\lambda$, which is a
convex function with global minimum at $\lambda=1$, and $w\to\infty$
as $\lambda\to0$ or $\lambda\to\infty$. The function $\overline{W}$
is convex in $\lambda_{1},\dots,\lambda_{n}$, which gives a nice
way to understand the behavior of $W$.
\begin{prop}[Spatial invariance and homogeneity of uni-constant compressible Neo-Hooke
material]
\label{prop:W-is-spatially-invariant-and-homogeneous}$W$ defines
a spatially invariant material, i.e. $W_{,\sigma}\equiv0$. $W$ defines
a homogeneous material, i.e. $W_{,\beta}\equiv0$.
\end{prop}

\section{Problem Statement}

The specific problem involves embedding a flat, hyperelastic body
$B$ with uniform density $\rho>0$ into a surface $S$ of revolution
$z=z\left(r\right)$ (in cylindrical coordinates) that has a conservative
central force field meant to represent a gravitational force, then
finding an equilibrium solution in which the forces arising within
the body responding to its deformation, which tend to push it toward
flatter regions of the surface, perfectly cancel the gravitational
forces on the body. A solution to this problem represents a kind of
curvature-induced \textquotedbl levitation\textquotedbl{} phenomenon
within this surface, and therefore could fairly be called a \textbf{curvature
levitator}.

By construction, the problem is radially symmetric. The central force
field arises from a potential $U\colon S\to\mathbb{R}$ depending
only on $r$, thus $U$ is a function of $r$; $U=U\left(r\right)$.
Rotations about the center $r=0$ are isometries of $S$, so in particular,
the Gaussian curvature $K\colon S\to\mathbb{R}$ of the surface is
a function of $r$; $K=K\left(r\right)$.

The existence of a solution to the problem depends in particular on
the following.
\begin{itemize}
\item The surface has a region, large in comparison with the body, in which
$\left|K\right|$ is strictly increasing along the central force field.
\item The body's stored energy (i.e. integral of stored energy density over
the body) is zero iff it is in an undeformed state.
\item The body's stored energy density function is coercive, meaning that
the stored energy increases with more deformation.
\item The body is sufficiently stiff.
\end{itemize}
Let $X:=\mathbb{R}^{2}$ with standard coordinates $x=\left(x^{1},x^{2}\right)$,
and let $B:=\left[-\frac{1}{2},\frac{1}{2}\right]\times\left[-\frac{1}{2},\frac{1}{2}\right]\subset X$,
and let $g\left(x\right):=I$ (Euclidean metric), representing the
flat reference configuration for the hyperelastic body. Let $S\subset\mathbb{R}^{3}$
be the surface of revolution defined by $z=z\left(r\right)$. Let
$Y:=\mathbb{R}^{2}$ with $y=\left(y^{1},y^{2}\right)$ denoting graph
coordinates on $S$. The labels $X$ and $Y$, referring to the model
vector spaces for $B$ and $S$ respectively, will help in semantic
distinction of the different coordinate systems. In the $Y$ coordinates,
$r=r\left(y\right):=\sqrt{\left(y^{1}\right)^{2}+\left(y^{2}\right)^{2}}$.
With $h$ and $K$ denoting the metric and Gaussian curvature on $S$,
respectively, it follows that
\begin{align*}
h\left(y\right) & =I+z^{\prime}\left(r\right)^{2}r^{-2}y\otimes y, & K\left(r\right) & =\frac{z^{\prime}\left(r\right)z^{\prime\prime}\left(r\right)}{r\left(z^{\prime}\left(r\right)^{2}+1\right)^{2}},
\end{align*}
where $y\otimes y$ is the outer product; $\left(y\otimes y\right)^{ij}=y^{i}y^{j}$.

A deformation tensor has the form $F\in Y\otimes X^{*}$ (linear map
from $X$ to $Y$), and the stored energy density has the form
\begin{align*}
W\colon B\times S\times Y\otimes X^{*} & \to\mathbb{R},\\
\left(x,y,F\right) & \mapsto\alpha\left(\text{tr}C\left(x,y,F\right)-2-\log\det C\left(x,y,F\right)\right),
\end{align*}
where the Cauchy strain tensor is given by
\begin{align*}
C\colon B\times S\times Y\otimes X^{*} & \to X\otimes X^{*},\\
\left(x,y,F\right) & \mapsto g^{-1}\left(x\right)\cdot F^{*}\cdot h\left(y\right)\cdot F.
\end{align*}
Define gravitational potential simply as the elevation function on
$S$: 
\begin{align*}
U\colon S & \to\mathbb{R},\\
y & \mapsto z\left(r\left(y\right)\right),
\end{align*}
noting that $U=U\left(r\right)$ by construction. Let the body have
uniform mass density $\rho>0$. The potential energy density is then
\begin{align*}
V\colon S & \to\mathbb{R},\\
y & \mapsto\rho U\left(y\right),
\end{align*}
noting that the uniformity of the body's mass density implies that
$V$, which would normally have type $S\times B\to\mathbb{R}$ as
described in \subsecref{A-Particular-Lagrangian}, is independent
of the body point and hence has type $S\to\mathbb{R}$. $V=V\left(r\right)$
by construction. Define the Lagrangian density as the sum of energy
densities.
\begin{align*}
L\colon B\times S\times Y\otimes X^{*} & \to\mathbb{R},\\
\left(x,y,F\right) & \mapsto W\left(x,y,F\right)+V\left(y\right).
\end{align*}
Note that $L$ is symmetric under rotations about $0\in Y$; if $R\in Y\otimes Y^{*}$
is such a rotation, then $L\left(x,R\cdot y,R\cdot F\right)=L\left(x,y,F\right)$.

A body configuration, in coordinates, has the form $\phi\colon B\to Y$.
The Lagrangian (action functional) is then
\begin{align*}
\mathcal{L}\colon\mathcal{D} & \to\mathbb{R},\\
\phi & \mapsto\int_{B}L\left(x,\phi\left(x\right),D\phi\left(x\right)\right)\,d\mu_{B}\left(x\right),
\end{align*}
where $\mathcal{D}$ is an appropriate subset of $H^{1}\left(B,Y\right)$
that excludes the maps $\phi$ for which the integrand is not defined
due to the $-\log\det C\left(x,\phi\left(x\right),D\phi\left(x\right)\right)$
term, and potentially the $V\left(\phi\left(x\right)\right)$ term
if it has a singularity. In particular, $\mathcal{D}$ includes $C^{1}$
immersions that avoid $U$-term singularities. In practice this domain
restriction will not be a problem, as there is an infinite energy
penalty that \textquotedbl protects\textquotedbl{} against the singular
configurations while searching for a minimizer.

The problem is to find a stable solution to the static problem, i.e.
a body configuration $\phi$ such that $D\mathcal{L}\left(\phi\right)=0$
and $D^{2}\mathcal{L}\left(\phi\right)$ is positive definite modulo
the expected symmetries. By construction, rotation about $0\in Y$
is a symmetry of the problem. Its infinitesimal generator at the body
configuration $\phi$ is $x\mapsto\left[\begin{array}{cc}
0 & -1\\
1 & 0
\end{array}\right]\phi\left(x\right)$. There may be other symmetries, depending on the specific form of
$z\left(r\right)$.

Here are several specific surfaces in which the problem will be solved.
\begin{center}
\begin{tabular}{|c|c|c|c|}
\hline 
Surface Name & $z\left(r\right)$ & $h\left(y\right)$ & $K\left(r\right)$\tabularnewline
\hline 
\hline 
Generic & $z\left(r\right)$ & $I+z^{\prime}\left(r\right)^{2}r^{-2}y\otimes y$ & $\frac{z^{\prime}\left(r\right)z^{\prime\prime}\left(r\right)}{r\left(z^{\prime}\left(r\right)^{2}+1\right)^{2}}$\tabularnewline
\hline 
Funnel & $-r^{-1}$ & $I+r^{-6}y\otimes y$ & $-\frac{2r^{2}}{\left(r^{4}+1\right)^{2}}$\tabularnewline
\hline 
Paraboloid & $\frac{1}{2}r^{2}$ & $I+y\otimes y$ & $\frac{1}{\left(r^{2}+1\right)^{2}}$\tabularnewline
\hline 
Flamm's Paraboloid & $2\sqrt{r_{s}\left(r-r_{s}\right)}$ & $I+\frac{r_{s}}{r-r_{s}}r^{-2}y\otimes y$ & $-\frac{r_{s}}{2r^{3}}$\tabularnewline
\hline 
\end{tabular}
\par\end{center}

In Flamm's paraboloid, $r_{s}$ denotes the Schwarzschild radius of
the gravitational source. In addition, the problem will be solved
in the following constant-curvature surface in order to show that
the levitation effect indeed requires a gradient of curvature.
\begin{center}
\begin{tabular}{|c|c|c|c|}
\hline 
Surface Name & $z\left(r\right)$ & $h\left(y\right)$ & $K\left(r\right)$\tabularnewline
\hline 
\hline 
Spherical cup of radius $R$ & $-\sqrt{R^{2}-r^{2}}$ & $I+\left(R^{2}-r^{2}\right)^{-1}y\otimes y$ & $\frac{1}{R^{2}}$\tabularnewline
\hline 
\end{tabular}
\par\end{center}

\section{Finite Element Method Constructions}

A finite element method will be used to find approximate solutions
to the problem numerically. This involves defining a mesh on $B$
and constructing a basis of functions each with local support limited
to a small number of mesh elements. Because the domain in question
is a rectangle, it's possible to construct the basis functions as
the tensor product of univariate basis functions defined over a subdivided
interval.

\subsection{Univariate 1-Jet Control Functions\protect\label{subsec:Univariate-1-Jet-Control-Functions}}

The univariate \textquotedbl reference elements\textquotedbl{} will
be constructed as follows. Consider the unit interval $I:=\left[0,1\right]$
with standard coordinate $t$. Let $s:=1-t$. The functions
\begin{align*}
v_{0}^{0}\left(t\right) & :=s^{3}+3s^{2}t, & v_{0}^{1}\left(t\right) & :=3st^{2}+t^{3},\\
v_{1}^{0}\left(t\right) & :=s^{2}t, & v_{1}^{1}\left(t\right) & :=-st^{2}
\end{align*}
form a basis for the 4-dimensional $\mathbb{R}$-vector space of cubic
polynomials on $I$ (this includes polynomials of degree less than
3), such that each function controls exactly one component of the
jet at one endpoint of $I$. These functions are called the cubic
Hermite spline basis functions. In particular,
\begin{align}
\left(\frac{d}{dt}\right)^{j}v_{\ell}^{k}\left(t\right)\mid_{t=i} & =\delta^{ik}\delta_{j\ell}\text{ for }i,j,k,\ell\in\left\{ 0,1\right\} \label{eq:jet-control-functions-equations}
\end{align}
A function $f$ written in this basis has the form $f\left(t\right)=f_{i}^{j}v_{j}^{i}\left(t\right)$
for $f_{i}^{j}\in\mathbb{R}$, where $f_{i}^{j}$ specifies the value
of the $j$th derivative of $f$ at the endpoint $t=i$. This gives
a very nice way to interpolate univariate functions on $I$ based
on boundary 1-jet data. This scheme can be extended to jets of order
$m$ by solving the system \eqref{jet-control-functions-equations}
on the space of polynomials of degree (up to) $2n+1$, where $i,k\in\left\{ 0,1\right\} $
and $j,\ell\in\left\{ 0,\dots,m\right\} $.

\subsection{Univariate Finite Element Function Space}

The univariate finite element function space will consist of piecewise-cubic
polynomials on an interval that are globally $C^{1}$ (i.e. cubic
on each piece and $C^{1}$ at piece boundaries). 

Consider the interval be $K:=\left[a_{0},a_{1}\right]$ where $a_{0}<a_{1}$.
Subdivide $K$ into $n\in\mathbb{N}$ segments of equal length. Use
$\left|\cdot\right|$ notation to denote the length of an interval.
Let $\mathcal{P}_{n}:=\left\{ 0,\dots,n\right\} $, and let $b_{p}:=a_{0}+p\left|K\right|n^{-1}$
for $p\in\mathcal{P}_{n}$, denoting the sequence of segment boundary
points. Let $\mathcal{Q}_{n}:=\left\{ 0,\dots,n-1\right\} $. The
$q$th segment, for $q\in\mathcal{Q}_{n}$ is $K_{q}:=\left[b_{q},b_{q+1}\right]$.
Let $K_{n}:=\left\{ K_{q}\mid q\in\mathcal{Q}_{n}\right\} $. For
each $q\in\mathcal{Q}_{n}$, let $\psi_{q}\colon K_{q}\to I,\,x\mapsto b_{q}+\left(x-b_{q}\right)\left|K_{q}\right|^{-1}$,
denoting the mapping from $K_{q}$ to the reference interval $I$.
Let $J_{1}\left(K_{n}\right)$ denote the function space of $K_{n}$-piecewise-cubic
polynomials that are globally $C^{1}$, noting that $\dim J_{1}\left(K_{n}\right)=2\left(n+1\right)$.
For $p\in\mathcal{P}_{n}$ and $\ell\in\left\{ 0,1\right\} $, define
\begin{align*}
w_{\ell}^{p}\colon K & \to\mathbb{R},\\
x & \mapsto\begin{cases}
\frac{1}{\ell!}\left|K_{p-1}\right|^{\ell}v_{\ell}^{1}\circ\psi_{p-1} & \text{if }p-1\in\mathcal{Q}_{n}\text{ and }x\in K_{p-1},\\
\frac{1}{\ell!}\left|K_{p}\right|^{\ell}v_{\ell}^{0}\circ\psi_{p} & \text{if }p\in\mathcal{Q}_{n}\text{ and }x\in K_{p},\\
0 & \text{otherwise.}
\end{cases}
\end{align*}
The factor of $\frac{1}{\ell!}\left|K_{\cdot}\right|^{\ell}$ is necessary
to map jet data in $K$ to jet data in the reference element $I$.
While the $\frac{1}{\ell!}$ factor seems redundant, it is necessary
in order to extend this finite element scheme to higher-order jets.

These functions span $J_{1}\left(K_{n}\right)$, and each one controls
exactly one component of a function's 1-jet at exactly one segment
boundary point. This gives a very nice way to interpolate univariate
functions on $K$ based on the sampling of the 1-jet at the segment
boundary points. $J_{1}\left(K_{n}\right)$ has a natural representation
as $\mathbb{R}^{\mathcal{P}_{n}\times\left\{ 0,1\right\} }$, where
the $\left(p,\ell\right)$th component gives the coefficient for $w_{\ell}^{p}$.
Thus, a function $f$ can be written in this basis as $f=f_{p}^{\ \ell}w_{\ell}^{p}$.
This representation is particularly suited for computers.

It should be noted that $J_{1}\left(K_{n}\right)$ is a vector subspace
of $J_{1}\left(K_{2n}\right)$, since each cubic defined on $K_{q}$
naturally defines a cubic on each of the two halves of $K_{q}$.

\subsection{Bivariate Finite Element Function Space}

The bivariate finite element function space will consist of piecewise-bicubic
polynomials on a rectangle that are globally $C^{1}$ (i.e. bicubic
on each piece and $C^{1}$ at piece boundaries).

Let $B:=K^{1}\times K^{2}$, where $K^{i}=\left[a_{0}^{i},a_{1}^{i}\right]$
is an interval with $a_{0}^{i}<a_{1}^{i}$, and let $n^{i}\in\mathbb{N}$
denote the number of subdivisions for $K^{i}$. Let 
\begin{align*}
B_{n^{1}n^{2}} & :=K_{n^{1}}^{1}\times K_{n^{2}}^{2}\equiv\left\{ K_{q^{1}}^{1}\times K_{q^{2}}^{2}\subset B\mid q^{1}\in\mathcal{Q}_{n^{1}}\text{ and }q^{2}\in\mathcal{Q}_{n^{2}}\right\} ,
\end{align*}
and let $J_{1}\left(B_{n^{1}n^{2}}\right)$ denote the $B_{n^{1}n^{2}}$-piecewise
bicubic polynomials on $B$ that are globally $C^{1}$. Then
\begin{align*}
J_{1}\left(B_{n^{1}n^{2}}\right) & =J_{1}\left(K_{n^{1}}^{1}\right)\otimes J_{1}\left(K_{n^{2}}^{2}\right),
\end{align*}
noting that $\dim J_{1}\left(B_{n^{1}n^{2}}\right)=4\left(n^{1}+1\right)\left(n^{2}+1\right)$.
Let $\left(w_{\ell}^{p}\right)$ and $\left(z_{\ell}^{p}\right)$
denote the 1-jet control bases in $J_{1}\left(K_{n^{1}}^{1}\right)$
and $J_{1}\left(K_{n^{2}}^{2}\right)$ respectively. $J_{1}\left(B_{n^{1}n^{2}}\right)$
has a natural representation as $\mathbb{R}^{\mathcal{P}_{n^{1}}\times\mathcal{P}_{n^{2}}\times\left\{ 0,1\right\} \times\left\{ 0,1\right\} }$.
A function $f\in J_{1}\left(B_{n^{1}n^{2}}\right)$ is written in
this basis as
\begin{align*}
f\left(x^{1},x^{2}\right) & =f_{ij}^{\ \ k\ell}w_{k}^{i}\left(x^{1}\right)z_{\ell}^{j}\left(x^{2}\right),
\end{align*}
where $f_{ij}^{\ \ k\ell}$ gives the value of $\partial_{x^{1}}^{k}\partial_{x^{2}}^{\ell}f$
at the $\left(i,j\right)$th vertex in the rectangular mesh $B_{n^{1}n^{2}}$.
To put this more concretely, with $\left(v^{1},v^{2}\right)$ denoting
the $\left(i,j\right)$th vertex, the jet data is
\begin{align*}
f_{ij}^{\ \ 00} & =f\left(v^{1},v^{2}\right), & f_{ij}^{\ \ 10} & =\frac{\partial f}{\partial x^{1}}\left(v^{1},v^{2}\right), & f_{ij}^{\ \ 01} & =\frac{\partial f}{\partial x^{2}}\left(v^{1},v^{2}\right), & f_{ij}^{\ \ 11} & =\frac{\partial^{2}f}{\partial x^{1}\partial x^{2}}\left(v^{1},v^{2}\right).
\end{align*}
$J_{1}\left(\Omega\right)$ is a finite-dimensional subspace of $C^{1}\left(\Omega,\mathbb{R}\right)$.
So as to have a more self-encapsulated basis for $J_{1}\left(B_{n^{1}n^{2}}\right)$,
let
\begin{align*}
b_{\ \ k\ell}^{ij} & :=w_{k}^{i}\otimes z_{\ell}^{j},
\end{align*}
so that a function $f\in J_{1}\left(B_{n^{1}n^{2}}\right)$ is written
as $f=f_{ij}^{\ \ k\ell}b_{\ \ k\ell}^{ij}$. For ease of notation,
the basis symbols $b_{\ \ k\ell}^{ij}$ will not have the $n^{1}n^{2}$
parameter decorations as $B_{n^{1}n^{2}}$ does, and those parameter
values will be determined from context.

Let $F_{N}:=J_{1}\left(B_{NN}\right)$, noting that the subdivisions
along each axis are made at the same rate. This construction uses
what is known as the Bogner-Fox-Schmit Rectangle \citet{Valdman2020},
and is known to be a $H^{2}$-conforming element, meaning that an
element of $H^{2}\left(B,\mathbb{R}\right)$ can be arbitrarily-well
approximated by an element of $F_{N}$ for some $N\in\mathbb{N}$.

The space $F_{N}$ only provides $\mathbb{R}$-valued functions. However,
this can easily be extended to vector- and manifold-valued functions
by choosing a coordinate system, say $Y$, and declaring each coordinate
function to be in $F_{N}$. This coordinatized function space has
a particularly nice representation as $V_{N}:=Y\otimes F_{N}$, having
dimension $4\left(N+1\right)^{2}\left(\dim Y\right)$. It follows
that $V_{N}$ is a $H^{2}\left(B,Y\right)$-conforming element, and
is sufficient for use in this problem.

\section{Numerical Implementation}

A significant difficulty in producing numerical solutions to mathematical
problems is found in implementing the various formulas that define
the problem. There are several stages that pose challenges here, where
failure in one stage precludes the success of the next stage.
\begin{enumerate}
\item Derive the correct mathematical formulas.
\item Implement concrete numerical computation of the formulas correctly
in code, accounting for numerical ill-conditioning as appropriate.
\item Write test code to verify that each implementation satisfies the expected
properties and constraints. Test code must also account for numerical
ill-conditioning.
\end{enumerate}
Thus, in order to improve (or even make tractable) the process of
producing correct code, any tools that reduce the amount of code that
needs to be written and tested are extremely valuable.

Many problems involve computing derivatives of various kinds. However,
testing the correctness of the numerical computation of a derivative
is difficult. The usual finite-difference approximation of the derivative,
\begin{align*}
f^{\prime}\left(x\right) & =\frac{f\left(x+h\right)-f\left(x\right)}{h}+O\left(h\right),
\end{align*}
has extremely poor numerical conditioning -- subtracting two floating
point numbers that are similar in magnitude introduces a very large
round-off error. Thus, even writing tests to verify correctness of
derivative computations is difficult.

\subsection{Automatic Differentiation}

The methods described in this section, in the author's opinion, are
of critical importance, and are an indispensable tool for any person
working in numerical computation.

\subsubsection{The Dual Numbers and Hyper-Dual Numbers\protect\label{subsec:The-Dual-and-Hyper-Dual-Numbers}}

The method of automatic differentiation relies on the notion of \textbf{dual
numbers} \citet{rehner2021}. This algebraic structure allows one
to perform exact computations of derivatives. The dual numbers are
a commutative algebra $\mathbb{D}:=\left\langle 1,\epsilon\right\rangle $
(i.e. the free $\mathbb{R}$-vector space with basis elements $1$
and $\epsilon$), where $\epsilon$ is the unit for infinitesimals,
i.e. $\epsilon^{2}:=0$. $\mathbb{D}$ can rightly be thought of as
the space of 1-jets of real-valued functions, and could fairly (and
more accurately) be called the \textbf{1-jet numbers}. Simply by using
this kind of number in place of a floating point value in a numerical
computation results in exact computation of the function and its derivative
-- no need to even select an $h$ value. A function $f\in C^{1}\left(\mathbb{R},\mathbb{R}\right)$
is uniquely extended to $\mathbb{D}\to\mathbb{D}$ by 
\begin{align*}
f\left(x+y\epsilon\right) & =f\left(x\right)+f^{\prime}\left(x\right)y\epsilon.
\end{align*}
The value $x+\epsilon\in\mathbb{D}$ can be thought of as the 1-jet
of the identity function on $\mathbb{R}$ evaluated at $x$.

The \textbf{hyper-dual numbers}, introduced in \citet{fike_2013_hyperdual},
make multivariate and higher-order derivatives possible. The hyper-dual
numbers $\mathbb{H}$ extend the idea to have two infinitesimals $\epsilon_{1},\epsilon_{2}$,
each squaring to zero, and $\epsilon_{1}\epsilon_{2}=\epsilon_{2}\epsilon_{1}\neq0$.
A hyper-dual number has the form $x+y^{1}\epsilon_{1}+y^{2}\epsilon_{2}+z\epsilon_{1}\epsilon_{2}$
for $x,y^{1},y^{2},z\in\mathbb{R}$. Then a function $f\in C^{2}\left(\mathbb{R},\mathbb{R}\right)$
is uniquely extended to $\mathbb{H}\to\mathbb{H}$ by
\begin{align*}
f\left(x+y^{1}\epsilon_{1}+y^{2}\epsilon_{2}+z\epsilon_{1}\epsilon_{2}\right) & =f\left(x\right)+f^{\prime}\left(x\right)y^{1}\epsilon_{1}+f^{\prime}\left(x\right)y^{2}\epsilon_{2}+\left(f^{\prime\prime}\left(x\right)y^{1}y^{2}+f^{\prime}\left(x\right)z\right)\epsilon_{1}\epsilon_{2},
\end{align*}
and it follows that
\begin{align*}
f\left(x+\epsilon_{1}+\epsilon_{2}\right) & =f\left(x\right)+f^{\prime}\left(x\right)\epsilon_{1}+f^{\prime}\left(x\right)\epsilon_{2}+f^{\prime\prime}\left(x\right)\epsilon_{1}\epsilon_{2},
\end{align*}
meaning that the $2$-jet of $f$ at $x$ is computed in a single
call to $f$. The peculiar form of the second order term in the extension
is due to the need for extensions to satisfy the chain rule. The hyper-dual
numbers can also be used for multivariate automatic differentiation.
If $U$ and $V$ are finite-dimensional vector spaces, then $F\in C^{2}\left(U,V\right)$
is uniquely extended to $\mathbb{H}\otimes U\to\mathbb{H}\otimes V$
by
\begin{align*}
F\left(x+y^{1}\epsilon_{1}+y^{2}\epsilon_{2}+z\epsilon_{1}\epsilon_{2}\right) & =F\left(x\right)+DF\left(x\right)\cdot y^{1}\epsilon_{1}+DF\left(x\right)\cdot y^{2}\epsilon_{2}+\left(D^{2}F\left(x\right):\left(y^{1}\otimes y^{2}\right)+DF\left(x\right)\cdot z\right)\epsilon_{1}\epsilon_{2},
\end{align*}
where $x,y^{1},y^{2},z\in U$. Then, 
\begin{align*}
F\left(x+y^{1}\epsilon_{1}+y^{2}\epsilon_{2}\right) & =F\left(x\right)+DF\left(x\right)\cdot y^{1}\epsilon_{1}+DF\left(x\right)\cdot y^{2}\epsilon_{2}+D^{2}F\left(x\right):\left(y^{1}\otimes y^{2}\right)\epsilon_{1}\epsilon_{2},
\end{align*}
showing that $F\left(x\right)$, $F_{,y^{1}}\left(x\right)$, $F_{,y^{2}}\left(x\right)$,
and $F_{,y^{1}y^{2}}\left(x\right)$ are computed in a single call
to $F$.

Note that one call to $F$ \textbf{does} \textbf{not} give the full
2-jet of $F$ if $\dim U>1$. Evaluating $DF\left(x\right)$ requires
$\dim X$ evaluations of $F$, and evaluating $D^{2}F\left(x\right)$
requires ${\dim X \choose 2}$ evaluations of $F$. Thus there is
redundant computation of $F\left(x\right)$ when computing $DF\left(x\right)$
and there is redundant computation of $F\left(x\right)$ and $DF\left(x\right)$
when computing $D^{2}F\left(x\right)$. However, this method provides
a qualitative utility that can't be understated -- one must only
implement the base function $F$, and the implementations of $DF$
and $D^{2}F$ are provided automatically. Given how complicated derivatives
of expressions of tensor calculus can get, this advantage should not
be ignored.

\subsection{Numerical Optimization}

The static hyperelastic problem addressed in this article is a local
minimization problem on $H^{1}\left(B,Y\right)$, which is an infinite-dimensional
space. With the finite-dimensional function spaces $V_{N}$ constructed,
let $\mathcal{L}_{N}:=\mathcal{L}\mid_{V_{N}}$ denote the restricted
Lagrangian. The problem becomes a sequence of local minimization problems
for each $\mathcal{L}_{N}$. A solution $\phi_{N}\in V_{N}$ is defined
by $D\mathcal{L}_{N}\left(\phi_{N}\right)=0$ and $D^{2}\mathcal{L}_{N}\left(\phi_{N}\right)$
being positive-definite modulo expected symmetries, both conditions
subject to numerical tolerance. Here, the appropriate second derivative
is the ordinary Hessian $D^{2}\mathcal{L}_{N}$, since the domain
of $\mathcal{L}_{N}$ is a finite-dimensional vector space. $\phi_{N}$
can then be mapped into $V_{N+1}$ to form the initial condition for
the local minimization problem for $\mathcal{L}_{N+1}$. At some point,
this iterative process will exhaust the numerical well-conditioning,
the computer's capabilities, or the patience of the mathematician,
and the process will be terminated. Nonlinear numerical optimization
takes some finesse, as there is no one-size-fits-all optimization
algorithm that works for every problem. The code that produced the
solutions presented in this article can be found at \url{https://github.com/vdods/jello}.

\subsection{Computation of $\mathcal{L}$}

Recall that $X=\mathbb{R}^{2}$, $Y=\mathbb{R}^{2}$, $\left(x^{1},x^{2}\right)$
are the standard coordinates on $X$, and $\left(y^{1},y^{2}\right)$
are the standard coordinates on $Y$. In particular, $x^{1},x^{2}\in X^{*}$
and $y^{1},y^{2}\in Y^{*}$ are the standard bases for $X^{*}\cong\mathbb{R}^{2}$
and $Y^{*}\cong\mathbb{R}^{2}$ respectively. Let $y_{1},y_{2}\in Y$
denote the basis dual to $y^{1},y^{2}\in Y^{*}$, noting that this
is the standard basis for $Y$.

With $H^{1}\left(B,S\right)$ coordinatized as $H^{1}\left(B,Y\right)$,
$\mathcal{L}$ takes the form
\begin{align*}
\mathcal{L}\left(\phi\right) & =\int_{B}L\left(x,\phi\left(x\right),D\phi\left(x\right)\right)\,d\mu_{B}\left(x\right).
\end{align*}
Thus in order to compute $\mathcal{L}\left(\phi\right)$, it's necessary
to compute the deformation tensor field $D\phi\colon B\to Y\otimes X^{*}$.
Recall that $b_{\ \ k\ell}^{ij}\colon B\to\mathbb{R}$ has the explicit
form
\begin{align*}
b_{\ \ k\ell}^{ij}\left(x^{1},x^{2}\right) & =w_{k}^{i}\left(x^{1}\right)z_{\ell}^{j}\left(x^{2}\right).
\end{align*}
Because $\phi$ will be represented as
\begin{align*}
\phi & =\phi_{\ ij}^{p\ \ k\ell}y_{p}\otimes b_{\ \ k\ell}^{ij}\in V_{N}\equiv Y\otimes J_{1}\left(B_{NN}\right)
\end{align*}
where $p$ indexes $Y$, it follows that
\begin{align*}
D\phi=y_{p}\otimes x^{q}\otimes\partial_{q}\phi^{p} & =\phi_{\ ij}^{p\ \ k\ell}y_{p}\otimes x^{q}\otimes\partial_{x^{q}}b_{\ \ k\ell}^{ij}\in Y\otimes X^{*}\otimes J_{0}\left(B_{NN}\right),
\end{align*}
where $J_{0}\left(B_{NN}\right)$ denotes the space of piecewise bicubic
polynomials that are globally $C^{0}$. Expressing the latter factor
concretely,
\begin{align*}
\partial_{x^{1}}b_{\ \ k\ell}^{ij}\left(x^{1},x^{2}\right) & =\left(w_{k}^{i}\right)^{\prime}\left(x^{1}\right)z_{\ell}^{j}\left(x^{2}\right), & \partial_{x^{2}}b_{\ \ k\ell}^{ij}\left(x^{1},x^{2}\right) & =w_{k}^{i}\left(x^{1}\right)\left(z_{\ell}^{j}\right)^{\prime}\left(x^{2}\right).
\end{align*}
Because $w_{k}^{i}$ and $z_{\ell}^{j}$ have known expressions as
piecewise polynomials, their derivatives do too. Thus $D\phi$ will
be computed directly, instead of through automatic differentiation.
Thus now the Lagrangian has the form
\begin{align*}
\mathcal{L}\left(\phi\right) & =\int_{B}L\left(x,\phi_{\ ij}^{p\ \ k\ell}b_{\ \ k\ell}^{ij}\left(x\right),\phi_{\ ij}^{p\ \ k\ell}\partial_{x^{q}}b_{\ \ k\ell}^{ij}\left(x\right)\right)\,d\mu_{B}\left(x\right).
\end{align*}
This is not quite computable, because even though $\phi\in V_{N}$,
the integrand can be highly nonlinear and therefore might not necessarily
reside in any computable function space -- remember, computers need
concrete, finite-dimensional function subspaces.

\subsubsection{Integration by Quadrature\protect\label{subsec:Integration-by-Quadrature}}

Integration by quadrature -- which would be better named \textbf{approximate
integration} -- is the standard method for computing an approximation
of such an integral. Essentially it is choosing a discrete measure
on a finite subset of the domain $\Omega$ of integration -- a set
of \textquotedbl sampling\textquotedbl{} points $x_{1},\dots,x_{n}\in\Omega$
and corresponding weights $w_{1},\dots,w_{n}\in\mathbb{R}_{+}$ --
giving a finite-dimensional approximation of the actual measure on
the domain.
\begin{align*}
\int_{\Omega}f\left(x\right)dx & \approx\sum_{i=1}^{n}f\left(x_{i}\right)w_{i}.
\end{align*}
The weights are chosen such that using the quadrature rule to \textquotedbl integrate
1\textquotedbl{} gives exactly the volume of $\Omega$. Roughly speaking,
the more points, the more accurate the approximation, though the choice
of point location and weight is nontrivial.

The domain $B$ is a rectangle, so a quadrature rule that is the product
of univariate quadrature rules can be used. The Gauss-Legendre quadrature
with $n$ points is constructed such that it integrates polynomials
of degree up to $2n-1$ exactly. The specific choice of $n$ depends
on the nature of the functions being integrated. In this problem,
\begin{align*}
L\left(x,\phi\left(x\right),D\phi\left(x\right)\right) & =\alpha\left(\text{tr}C\left(x,\phi\left(x\right),D\phi\left(x\right)\right)-2-\log\det C\left(x,\phi\left(x\right),D\phi\left(x\right)\right)\right)-U\left(\phi\left(x\right)\right).
\end{align*}
Because
\begin{align*}
\text{tr}C\left(x,\phi\left(x\right),D\phi\left(x\right)\right) & =\text{tr}\left(g^{-1}\left(x\right)\cdot D\phi\left(x\right)^{*}\cdot h\left(\phi\left(x\right)\right)\cdot D\phi\left(x\right)\right),
\end{align*}
this term is roughly quadratic in $D\phi$, and because $\phi$ is
bicubic, $D\phi$ is bicubic (being the sum of quadratic$\times$cubic
+ cubic$\times$quadratic), and thus $\alpha\text{tr}C\left(x,\phi\left(x\right),D\phi\left(x\right)\right)$
is roughly bisextic. The other terms in the Lagrangian density are
nonlinear, not suggesting an obvious heuristic. Thus a $5$-point
univariate quadrature rule, which integrates polynomials of degree
up to $9$, will be used, resulting in a $5\times5$ Gauss-Legendre
quadrature rule on each mesh rectangle.

For reference, on the domain $\left[-1,1\right]$, the first $5$
Gauss-Legendre quadrature rules are as follows, where $n$ is the
number of points, $x_{1},\dots,x_{n}$ are the integration points,
and $w_{1},\dots,w_{n}$ are the integration weights.
\begin{center}
\begin{tabular}{|c|c|c|}
\hline 
$n$ & $x_{i}$ & $w_{i}$\tabularnewline
\hline 
\hline 
1 & 0 & 2\tabularnewline
\hline 
2 & $-\frac{1}{\sqrt{3}},\frac{1}{\sqrt{3}}$ & 1\tabularnewline
\hline 
3 & $-\sqrt{\frac{3}{5}},0,\sqrt{\frac{3}{5}}$ & $\frac{5}{9},\frac{8}{9},\frac{5}{9}$\tabularnewline
\hline 
4 & $-\sqrt{\frac{3}{7}+\frac{2}{7}\sqrt{\frac{6}{5}}},-\sqrt{\frac{3}{7}-\frac{2}{7}\sqrt{\frac{6}{5}}},\sqrt{\frac{3}{7}-\frac{2}{7}\sqrt{\frac{6}{5}}},\sqrt{\frac{3}{7}+\frac{2}{7}\sqrt{\frac{6}{5}}}$ & $\frac{18-\sqrt{30}}{36},\frac{18+\sqrt{30}}{36},\frac{18+\sqrt{30}}{36},\frac{18-\sqrt{30}}{36}$\tabularnewline
\hline 
5 & $-\frac{1}{3}\sqrt{5+2\sqrt{\frac{10}{7}}},-\frac{1}{3}\sqrt{5-2\sqrt{\frac{10}{7}}},0,\frac{1}{3}\sqrt{5-2\sqrt{\frac{10}{7}}},\frac{1}{3}\sqrt{5+2\sqrt{\frac{10}{7}}}$ & $\frac{322-13\sqrt{70}}{900},\frac{322+13\sqrt{70}}{900},\frac{128}{225},\frac{322+13\sqrt{70}}{900},\frac{322-13\sqrt{70}}{900}$\tabularnewline
\hline 
\end{tabular}
\par\end{center}

In order to use such a quadrature rule on an arbitrary interval $\left[a,b\right]$,
the weights must be multiplied by $\frac{b-a}{2}$ and the points
transformed via the affine function $\left[-1,1\right]\to\left[a,b\right],\,x\mapsto\frac{1-x}{2}a+\frac{1+x}{2}b$.

To define a quadrature rule on a rectangle $R:=\left[a^{1},b^{1}\right]\times\left[a^{2},b^{2}\right]$
(i.e. the product of intervals), simply construct the product of quadrature
rules of each factor. The quadrature points for $R$ are $x_{uv}:=\left(x_{u}^{1},x_{v}^{2}\right)\in R$
where $x_{0}^{i},\dots,x_{n^{i}}^{i}\in\left[a^{i},b^{i}\right]$
denote the quadrature points for $\left[a^{i},b^{i}\right]$. The
weights for $R$ are $w_{uv}:=w_{u}^{1}w_{v}^{2}\in\mathbb{R}_{+}$,
where $w_{0}^{i},\dots,w_{n^{i}}^{i}\in\mathbb{R}_{+}$ denote the
weights for $\left[a^{i},b^{i}\right]$. Note that
\begin{align*}
\sum_{u=0}^{n^{1}}\sum_{v=0}^{n^{2}}w_{uv} & =\sum_{u=0}^{n^{1}}w_{u}^{1}\sum_{v=0}^{n^{2}}w_{v}^{2}=\left(b^{1}-a^{1}\right)\left(b^{2}-a^{2}\right)=\text{Area}\left(R\right)=\int_{R}1\,dx,
\end{align*}
as expected.

\subsubsection{Computable Form of $\mathcal{L}$}

In the problem at hand, $B=K^{1}\times K^{2}$, and $\dim V_{N}=\dim\left(Y\otimes J_{1}\left(K_{N}^{1}\right)\otimes J_{1}\left(K_{N}^{2}\right)\right)=2\cdot2\left(N+1\right)\cdot2\left(N+1\right)$.
Using the quadrature rule described in \subsecref{Integration-by-Quadrature},
let the integration points and weights be denoted $x_{uv}\in B$ and
$w_{uv}\in\mathbb{R}_{+}$, respectively.

Finally, the computable approximation $\mathcal{L}_{N}^{Q}$ of $\mathcal{L}$
can be written down.
\begin{align*}
\mathcal{L}\left(\phi\right) & \approx\mathcal{L}_{N}^{Q}\left(\phi\right):=\sum_{u,v=1}^{Q}L\left(x_{uv},\phi_{\ ij}^{p\ \ k\ell}\left(x_{uv}\right),\partial_{x^{q}}\phi_{\ ij}^{p\ \ k\ell}\left(x_{uv}\right)\right)w_{uv}.
\end{align*}

All code that computes needed functions and quantities should be written
such that it is \textquotedbl agnostic\textquotedbl{} as to the specific
number format being used, so long as that number format supports all
the expected operations -- algebraic operations, trigonometric functions,
exponentials, logarithms, and so forth. If done correctly, then the
hyper-dual numbers discussed in \subsecref{The-Dual-and-Hyper-Dual-Numbers}
can be used to perform automatic differentiation of $\mathcal{L}_{N}^{Q}$
to compute $\mathcal{L}_{N}^{Q}\left(\phi\right)$, $D\mathcal{L}_{N}^{Q}\left(\phi\right)$,
and $D^{2}\mathcal{L}_{N}^{Q}\left(\phi\right)$, which are used in
optimization methods to find solutions to the problem. With a careful
selection of an initial body configuration, the search for numerical
solutions can commence.

\section{Numerical Results}

Numerical optimization methods were used with $B=\left[-\frac{1}{2},\frac{1}{2}\right]\times\left[-\frac{1}{2},\frac{1}{2}\right]$
and an initial body configuration centered at $\left(y^{1},0\right)\in Y$
for varying $y^{1}\in\mathbb{R}$ values depending on the surface,
with the body rotated by an angle value chosen per surface. Each produced
a stable numerical solution to the static problem. All code and data
is available at \url{https://github.com/vdods/jello}. Updates and
other relevant additional work will be posted there.

\subsection{Funnel Surface $z=-r^{-1}$}

This solution shows a curvature levitator in a downward-pointing funnel
surface, which has vertical asymptote at $r=0$ and everywhere-negative
Gaussian curvature $K\left(r\right)=-\frac{2r}{\left(r^{4}+1\right)^{2}}$.
$K\left(r\right)$ has a global minimum at $r=7^{-1/4}\approx0.614788$,
and increases monotonically both as $r\to0$ and as $r\to\infty$.
The gravitational potential is $z$. The body is positioned fully
outside of the critical radius $r=7^{-1/4}$ so that its response
to deformation tends to push toward increasing $r$. If, on the other
hand, the body were to be entirely within the critical radius, its
deformation response forces would actually pull it in toward the center
-- thus the levitation phenomenon is not possible within the region
$r<7^{-1/4}$.

Here are plots showing the energy densities and vector fields that
describe important aspects of the solution. Modulo the expected symmetries,
$D^{2}\mathcal{L}\left(\phi\right)$ had only positive eigenvalues,
showing that the solution is stable.
\begin{center}
\begin{figure}[H]
\begin{centering}
\includegraphics[width=3.5in]{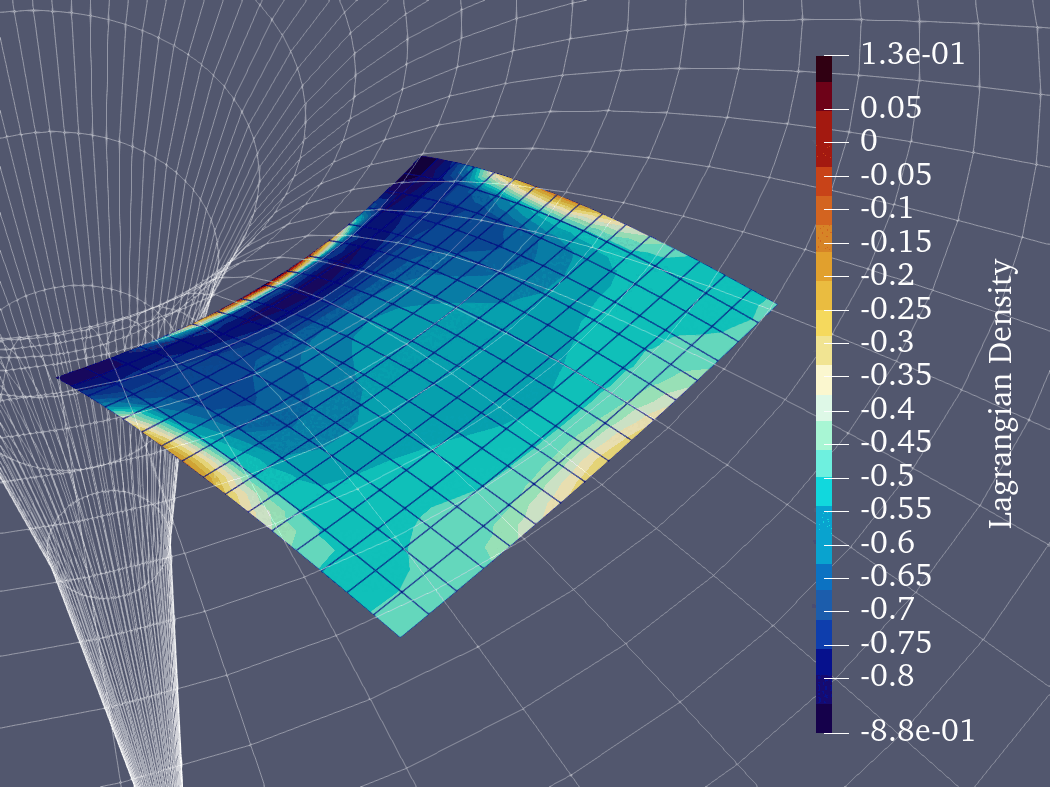}
\includegraphics[width=3.5in]{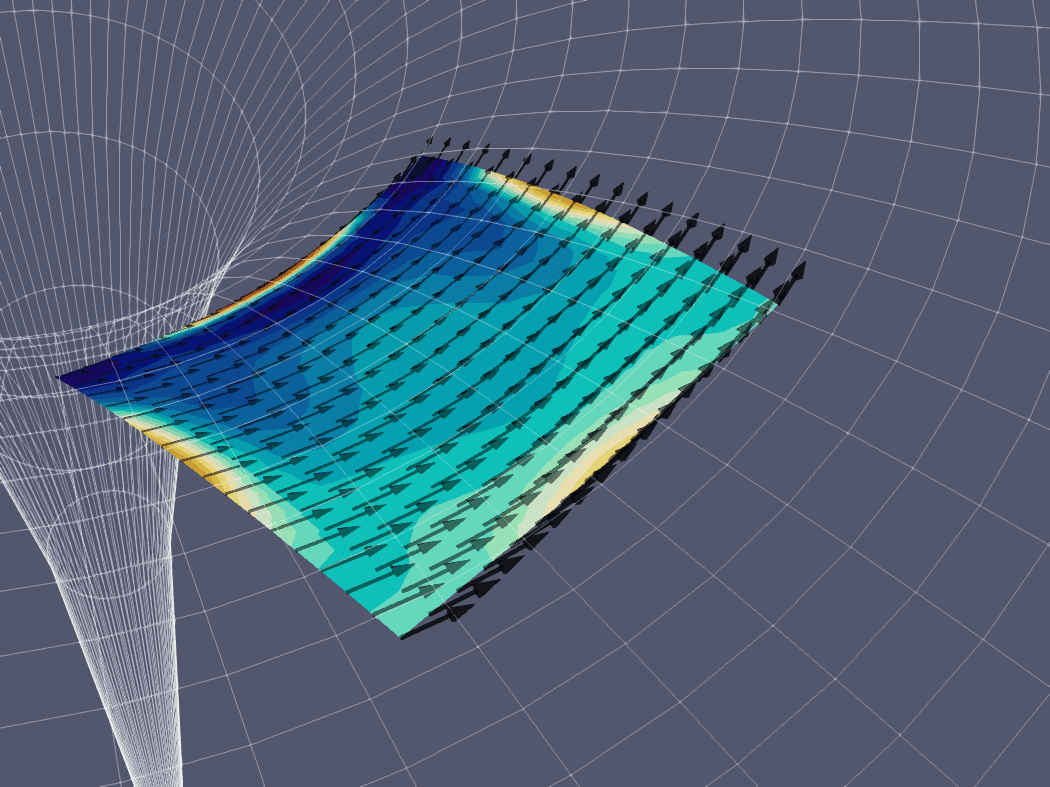}
\par\end{centering}
\caption{Left: The Lagrangian density of the body. Right: The vector field
corresponding to the expected, single zero eigenvalue of $D^{2}\mathcal{L}\left(\phi\right)$,
which is an infinitesimal rotation about the central axis of the surface
and is the generator of the symmetry group $\mathbb{S}^{1}$ of the
problem.}
\end{figure}
\par\end{center}

\begin{center}
\begin{figure}[H]
\begin{centering}
\includegraphics[width=3.5in]{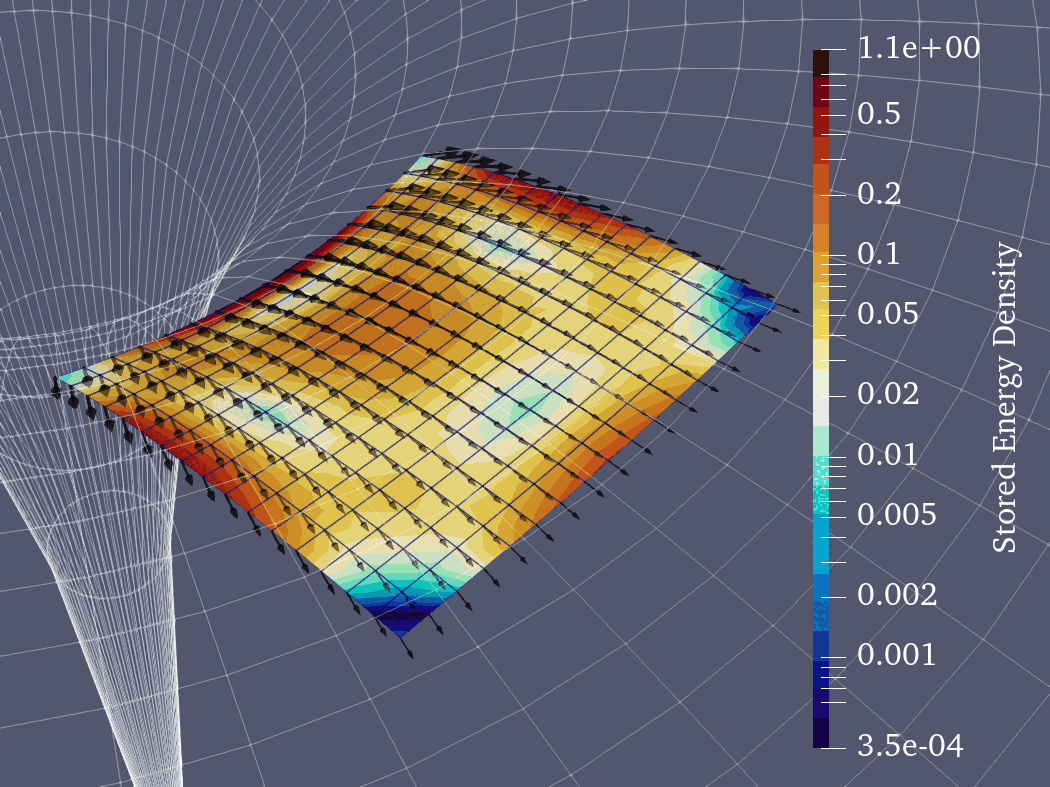}
\includegraphics[width=3.5in]{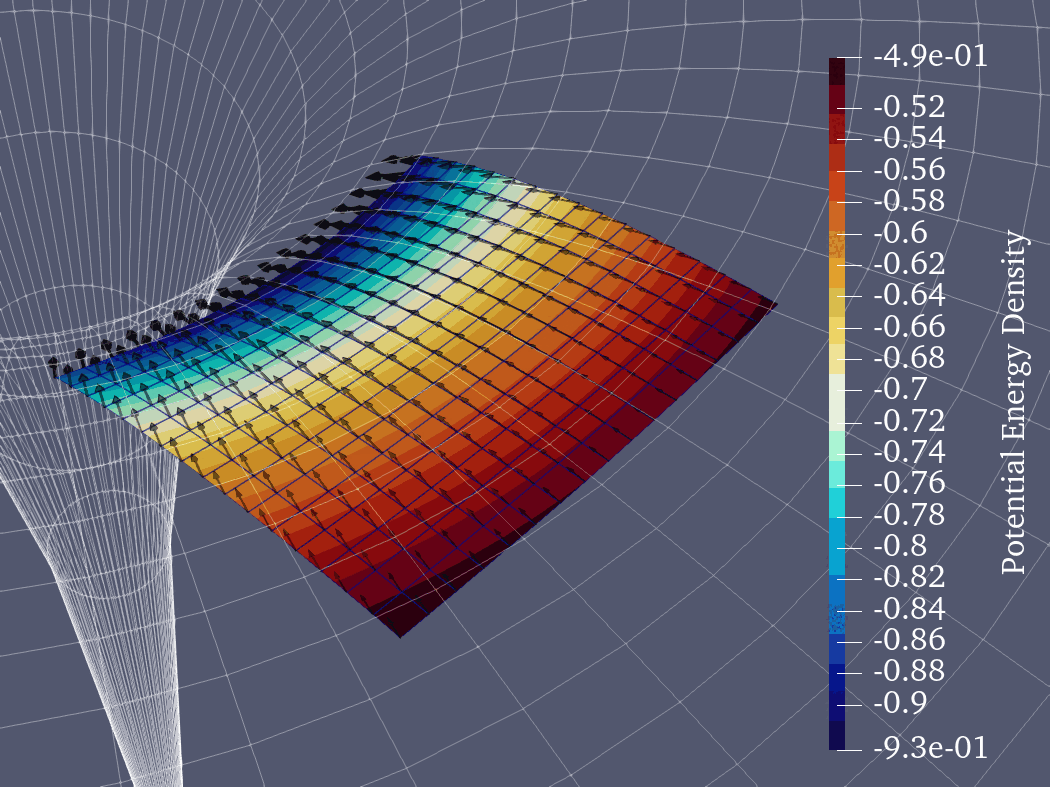}
\par\end{centering}
\caption{\protect\label{fig:funnel-stored-energy-density-potential-energy-density}Left:
The stored energy density of the body, along with the force field
arising from it. Right: The potential energy density of the body,
along with the force field arising from it. These two fields perfectly
cancel each other to within numerical tolerance.}
\end{figure}
\par\end{center}

\begin{figure}[H]
\begin{centering}
\includegraphics[width=3.5in]{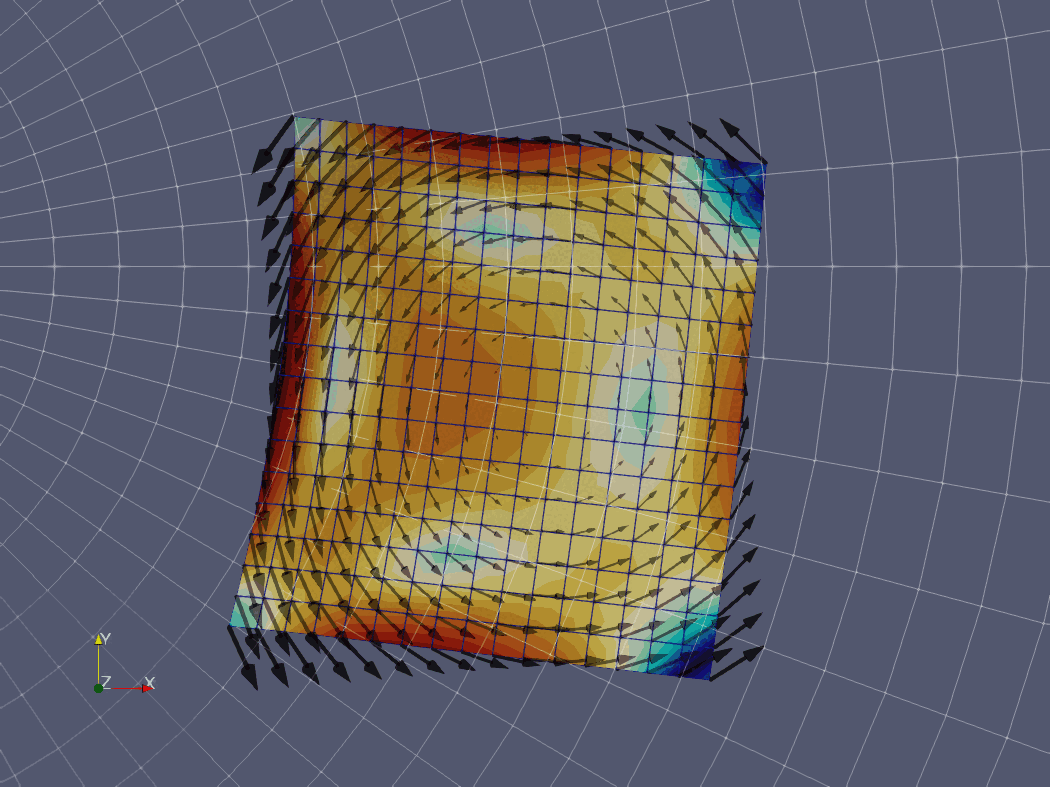}
\includegraphics[width=3.5in]{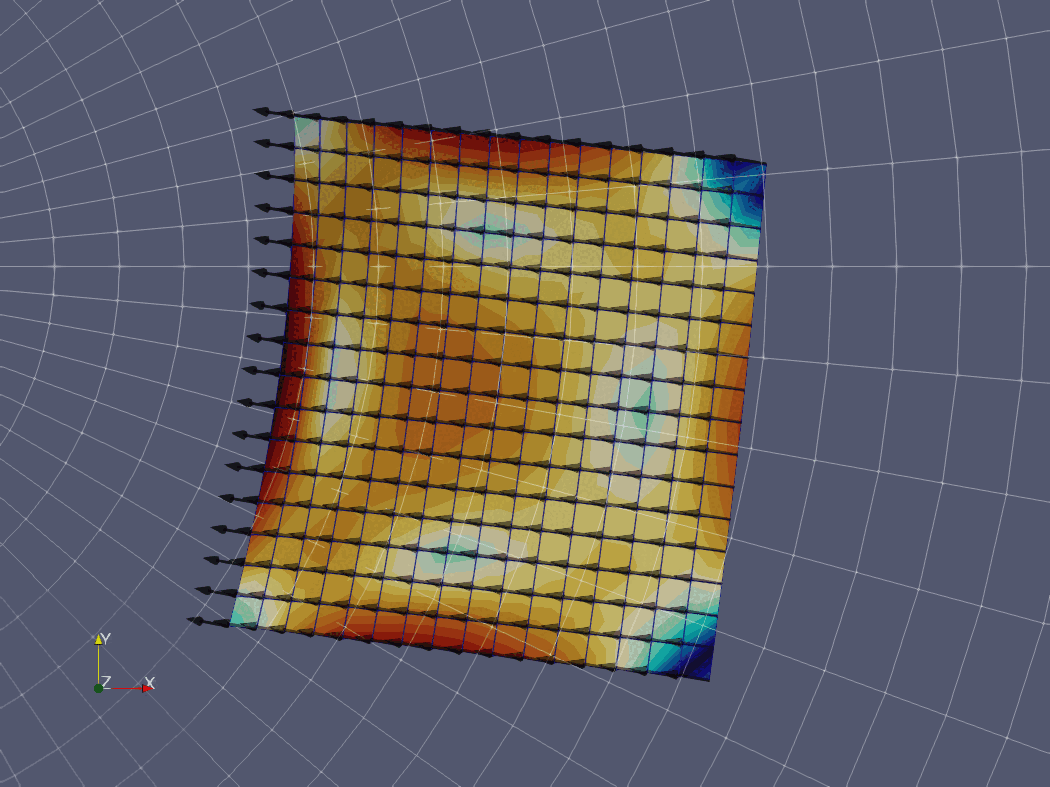}
\par\end{centering}
\begin{centering}
\includegraphics[width=3.5in]{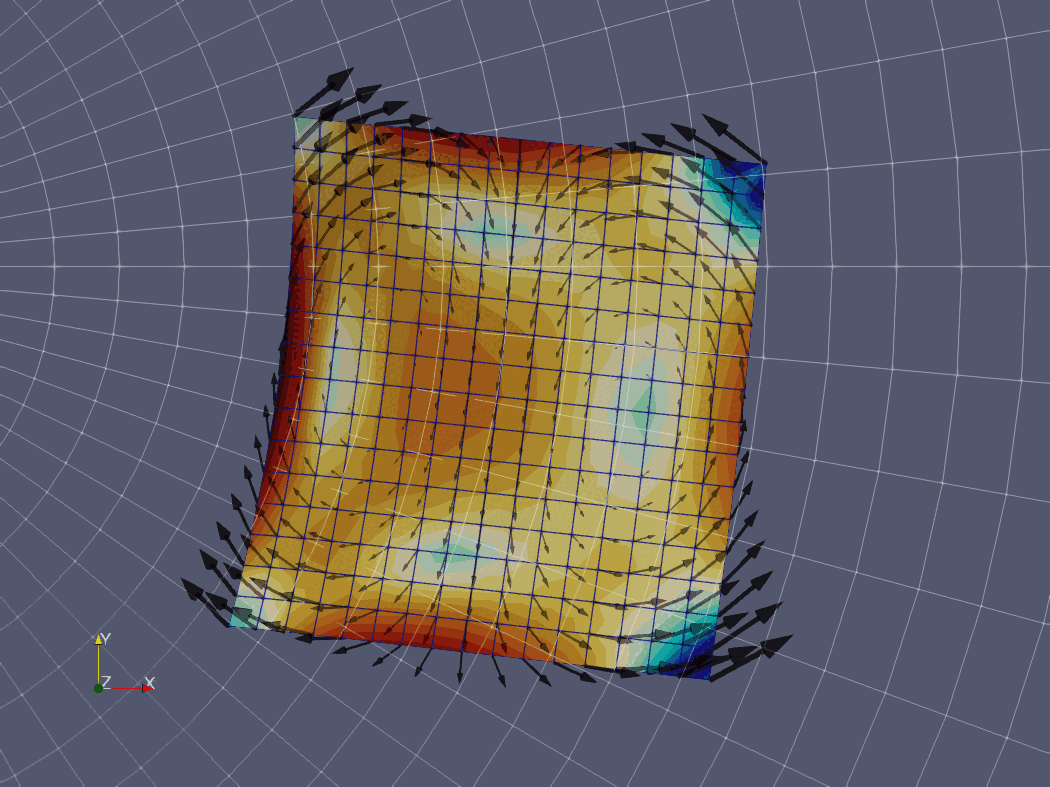}
\includegraphics[width=3.5in]{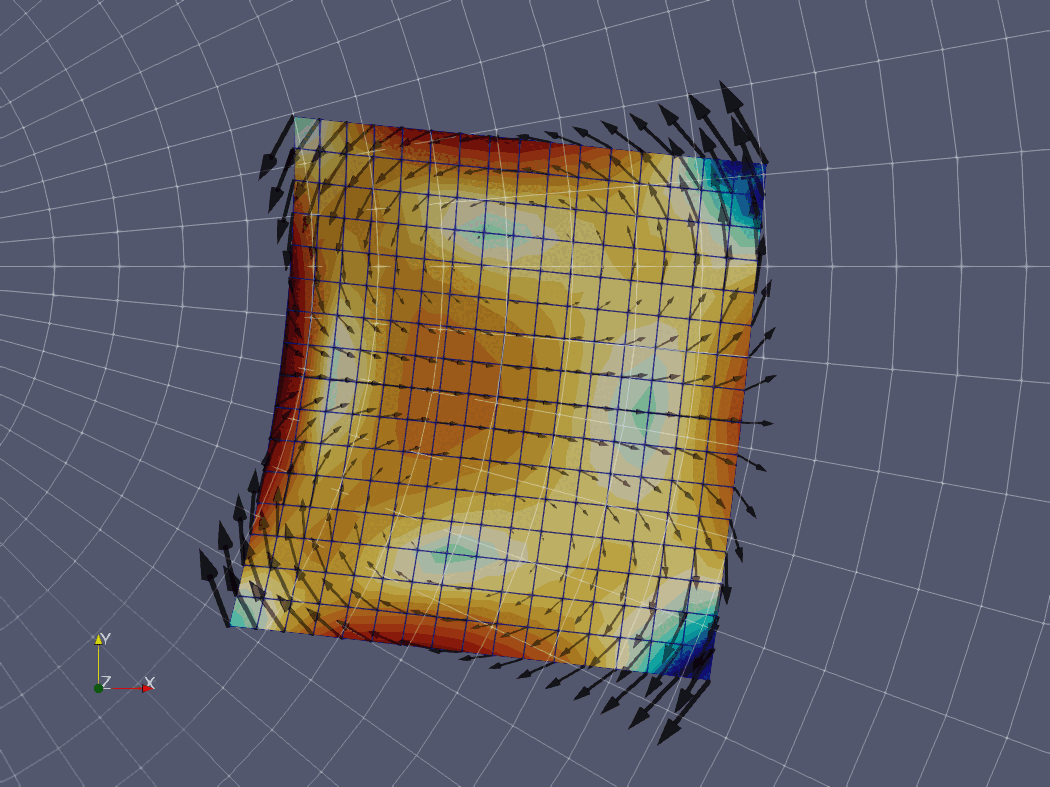}
\par\end{centering}
\caption{Top-down views of the vector fields corresponding to the smallest
four positive $D^{2}\mathcal{L}\left(\phi\right)$ eigenvalues.}
\end{figure}

\subsection{Paraboloid Surface $z=\frac{1}{2}r^{2}$}

This solution shows a curvature levitator in a parabolic cup surface,
which has everywhere-positive Gaussian curvature $K\left(r\right)=\frac{1}{\left(r^{2}+1\right)^{2}}$.
$K\left(r\right)$ has a global maximum at $r=0$ and decreases monotonically
as $r\to\infty$. The gravitational potential is $z$. Because this
surface has everywhere-positive Gaussian curvature, the levitation
phenomenon can occur anywhere in the surface, though it's possible
that a body that is \textquotedbl large\textquotedbl{} could exhibit
buckling.

Here are plots showing the energy densities and vector fields that
describe important aspects of the solution. Modulo the expected symmetries,
$D^{2}\mathcal{L}\left(\phi\right)$ had only positive eigenvalues,
showing that the solution is stable.
\begin{center}
\begin{figure}[H]
\begin{centering}
\includegraphics[width=3.5in]{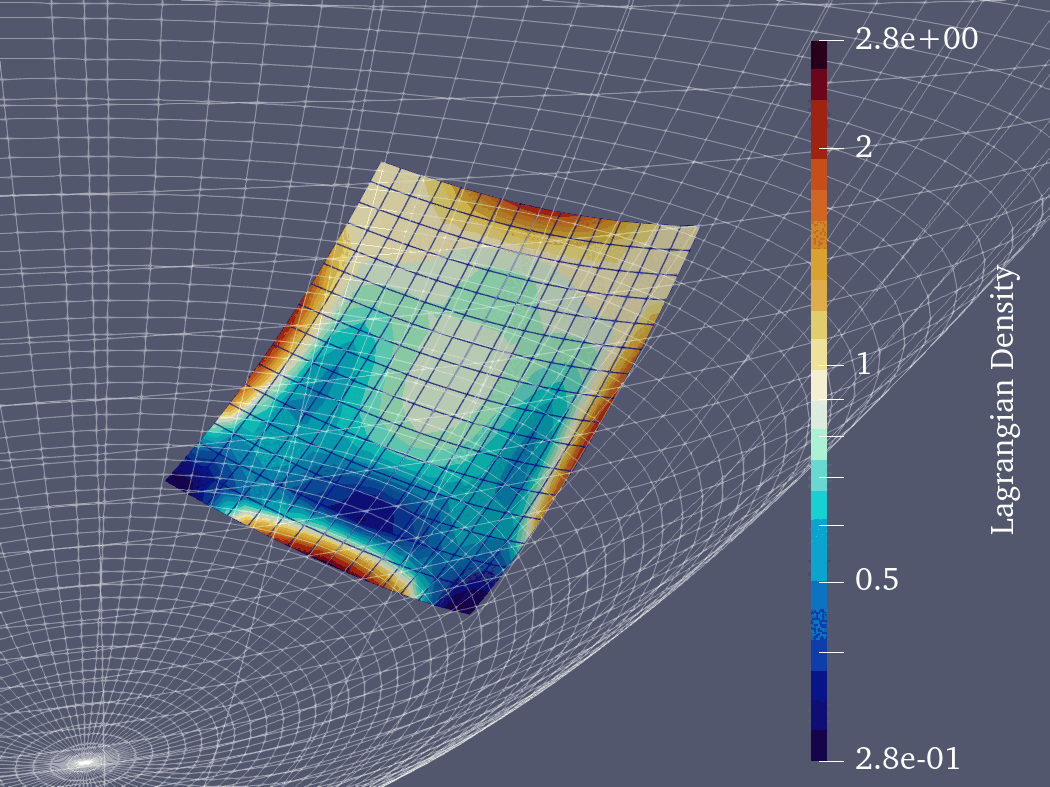}
\includegraphics[width=3.5in]{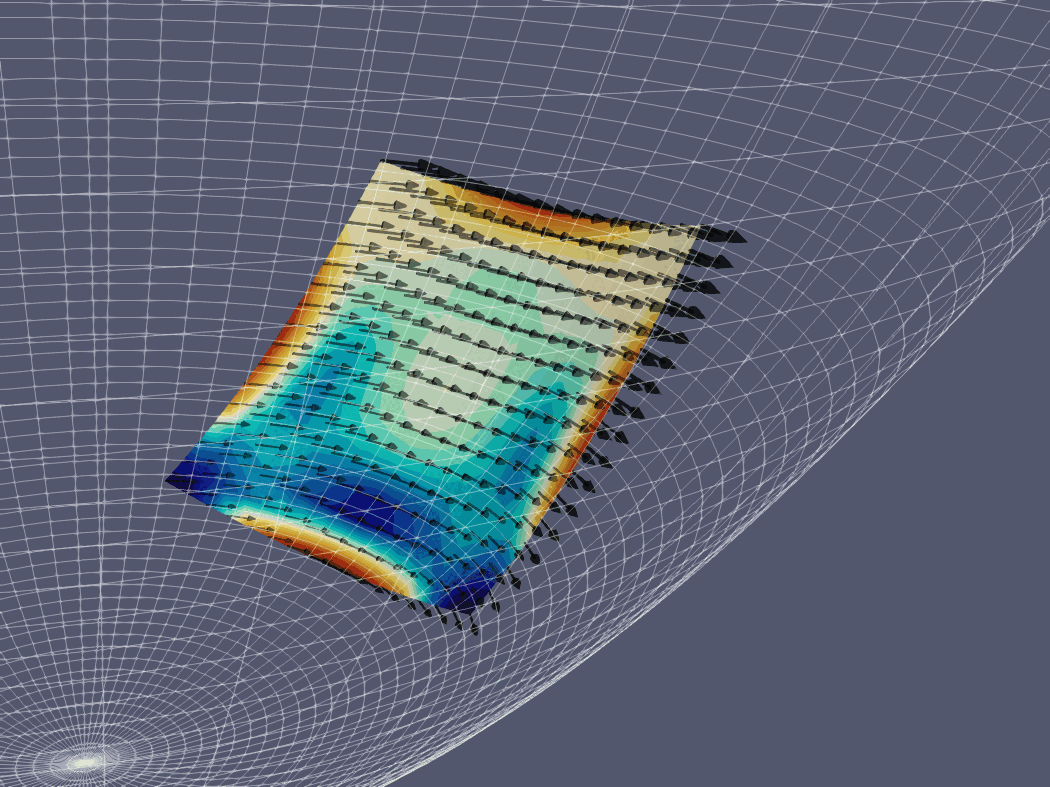}
\par\end{centering}
\caption{Left: The Lagrangian density of the body. Right: The vector field
corresponding to the expected, single zero eigenvalue of $D^{2}\mathcal{L}\left(\phi\right)$,
which is an infinitesimal rotation about the central axis of the surface
and is the generator of the symmetry group $\mathbb{S}^{1}$ of the
problem.}
\end{figure}
\par\end{center}

\begin{center}
\begin{figure}[H]
\begin{centering}
\includegraphics[width=3.5in]{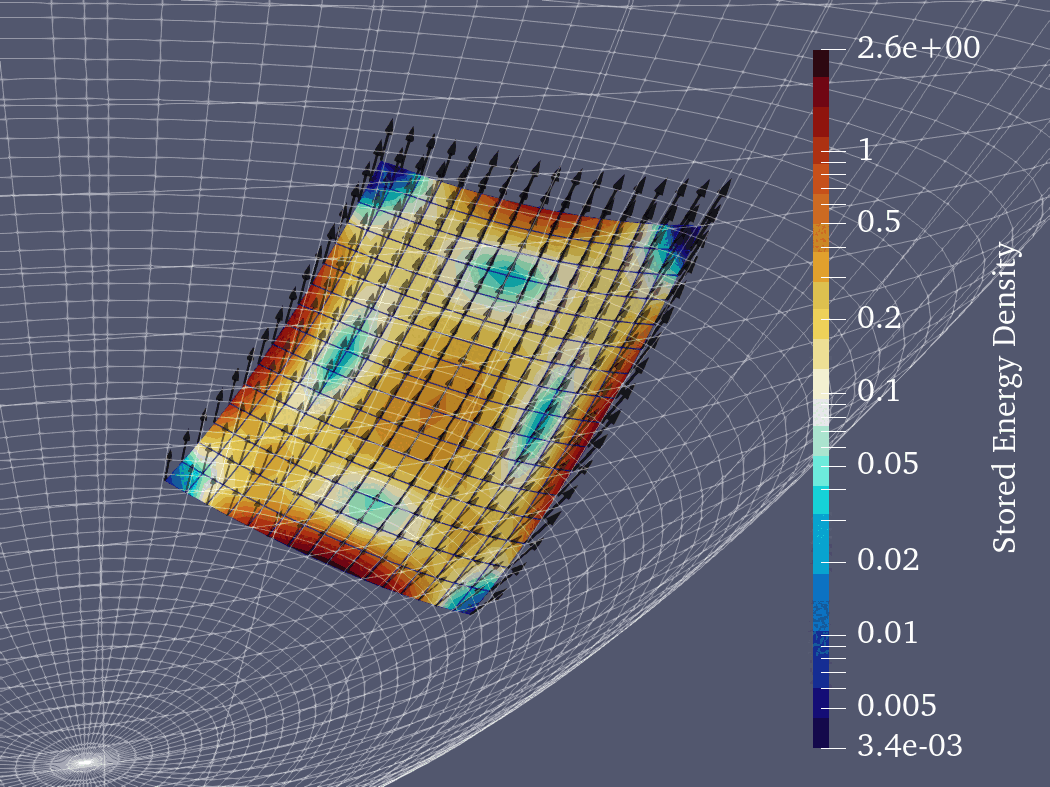}
\includegraphics[width=3.5in]{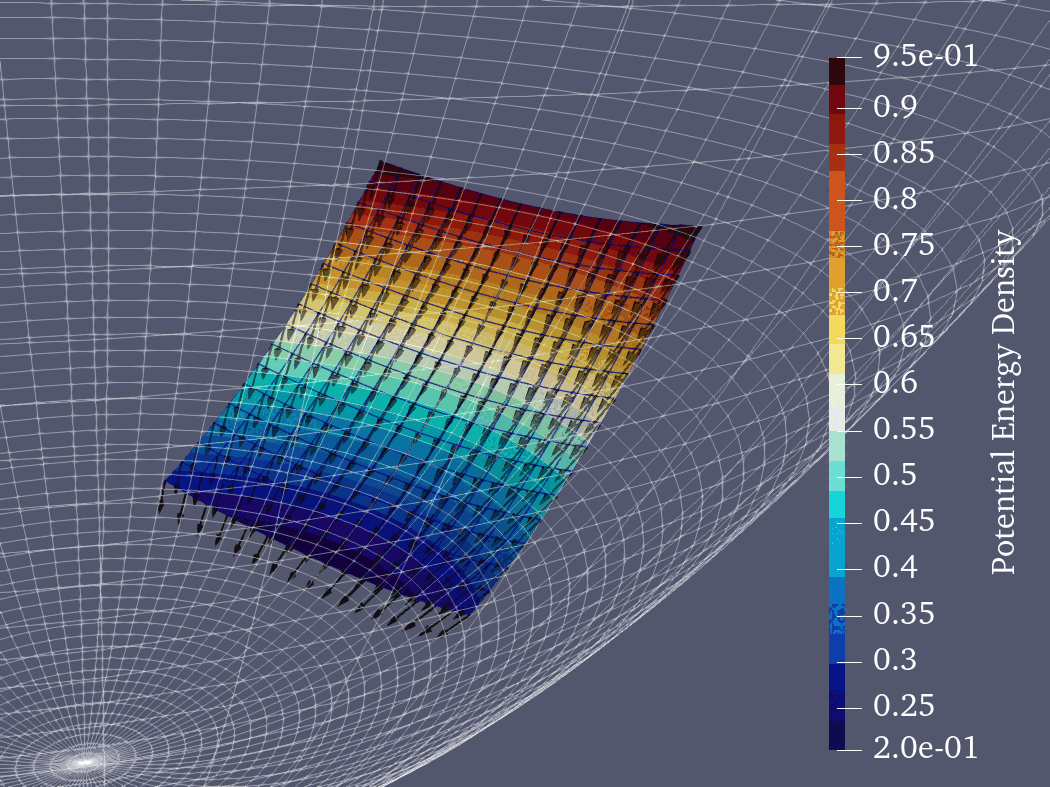}
\par\end{centering}
\caption{\protect\label{fig:paraboloid-stored-energy-density-potential-energy-density}Left:
The stored energy density of the body, along with the force field
arising from it. Right: The potential energy density of the body,
along with the force field arising from it. These two fields perfectly
cancel each other to within numerical tolerance.}
\end{figure}
\par\end{center}

\begin{figure}[H]
\begin{centering}
\includegraphics[width=3.5in]{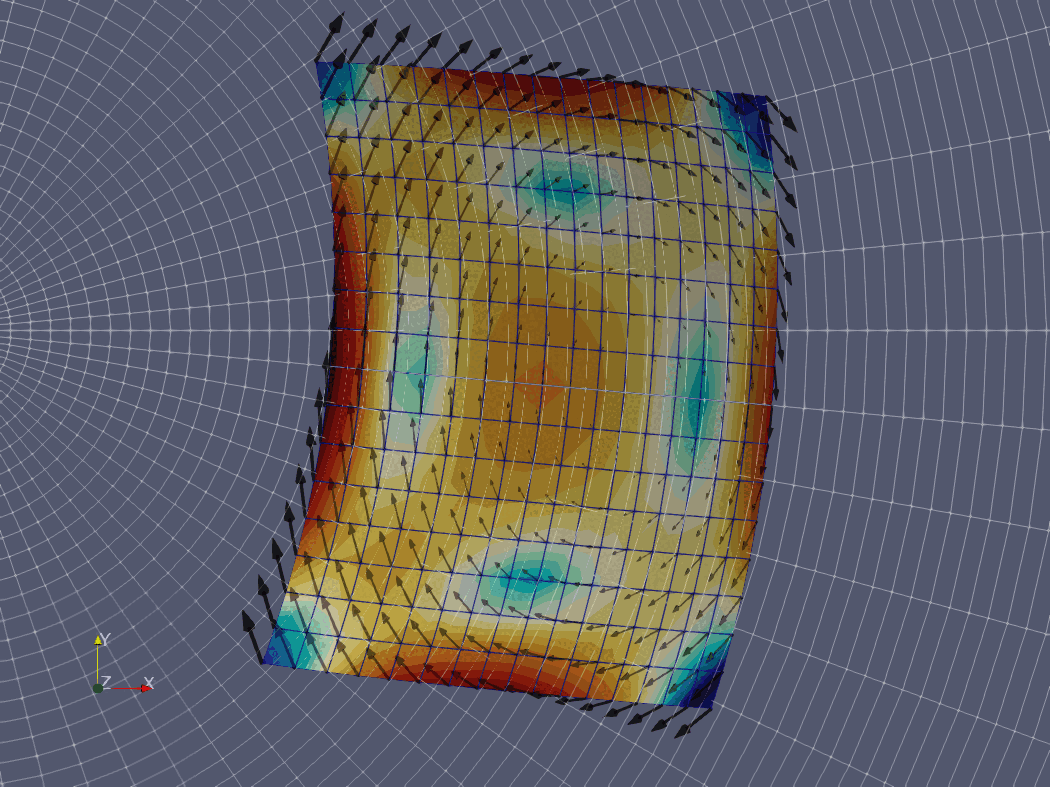}
\includegraphics[width=3.5in]{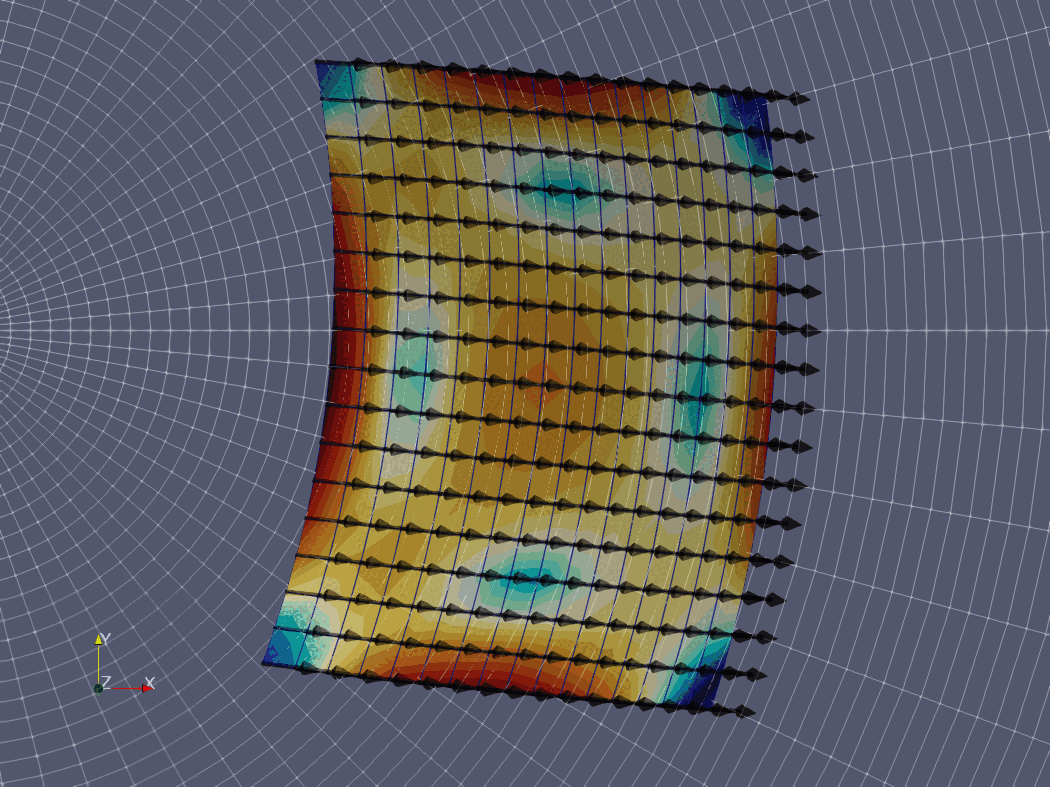}
\par\end{centering}
\begin{centering}
\includegraphics[width=3.5in]{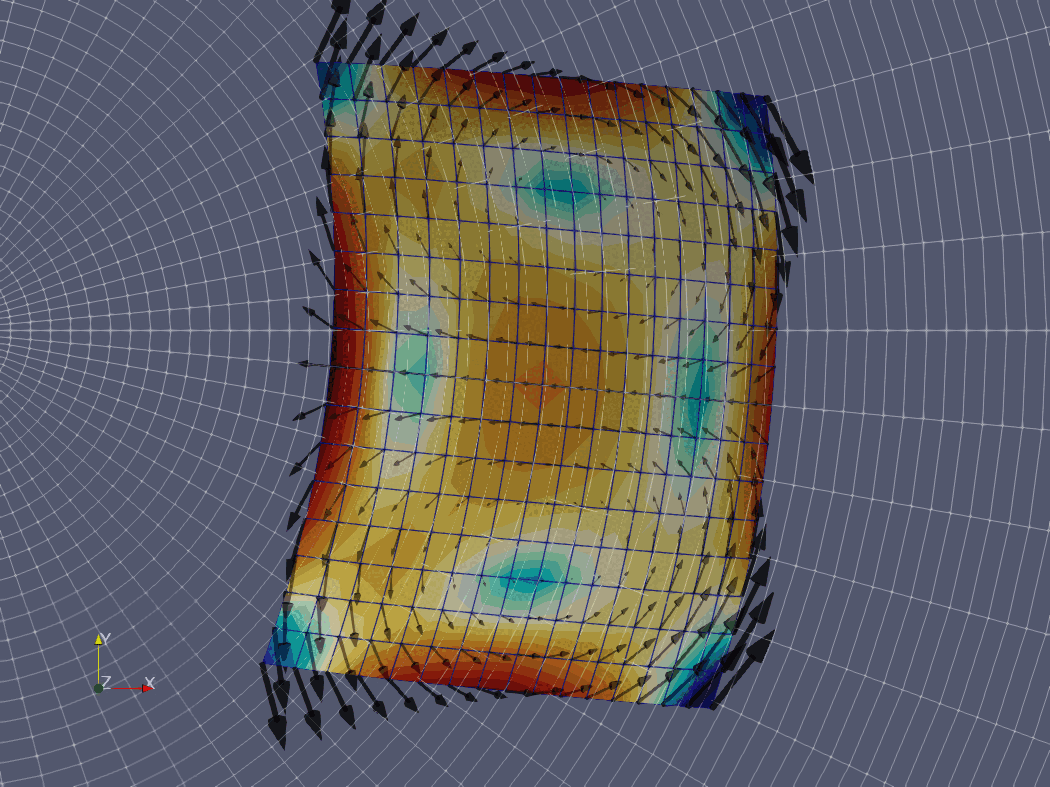}
\includegraphics[width=3.5in]{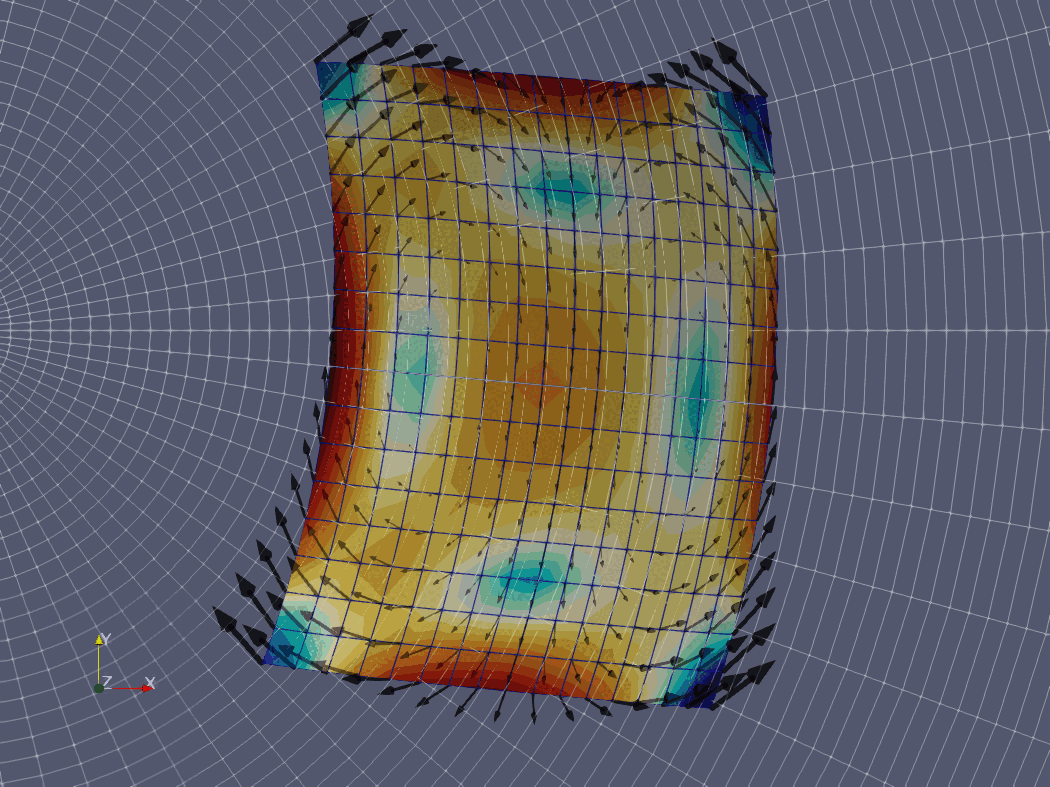}
\par\end{centering}
\caption{Top-down views of the vector fields corresponding to the smallest
four positive $D^{2}\mathcal{L}\left(\phi\right)$ eigenvalues.}
\end{figure}

\subsection{Spherical Cup Surface $z=-\sqrt{R^{2}-r^{2}}$}

This surface is the bottom half of the sphere of radius $R=5$. The
gravitational potential is $z$. The sphere has constant Gaussian
curvature $K=\frac{1}{25}$, and therefore has no curvature gradient
that a hyperelastic body can use to manifest the levitation phenomenon.
The body sits at the bottom of the cup as expected.

Here are plots showing the energy densities and vector fields that
describe important aspects of the solution. Modulo the expected symmetries,
$D^{2}\mathcal{L}\left(\phi\right)$ had only positive eigenvalues,
showing that the solution is stable.
\begin{center}
\begin{figure}[H]
\begin{centering}
\includegraphics[width=3.5in]{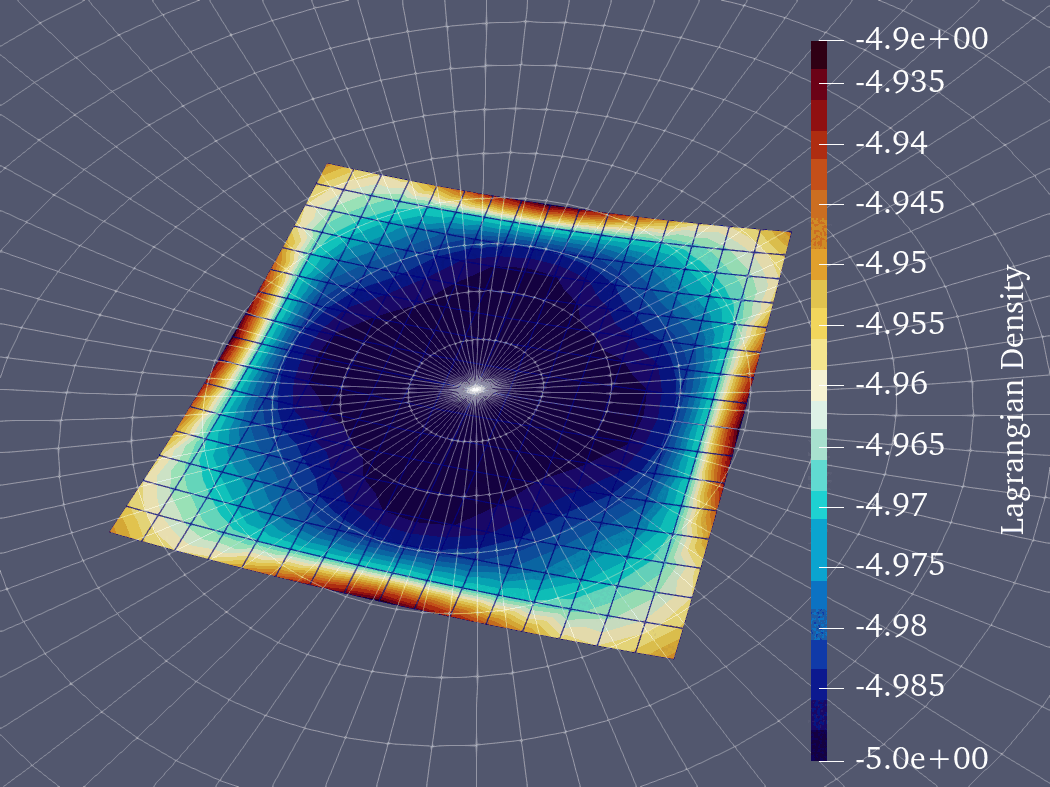}
\includegraphics[width=3.5in]{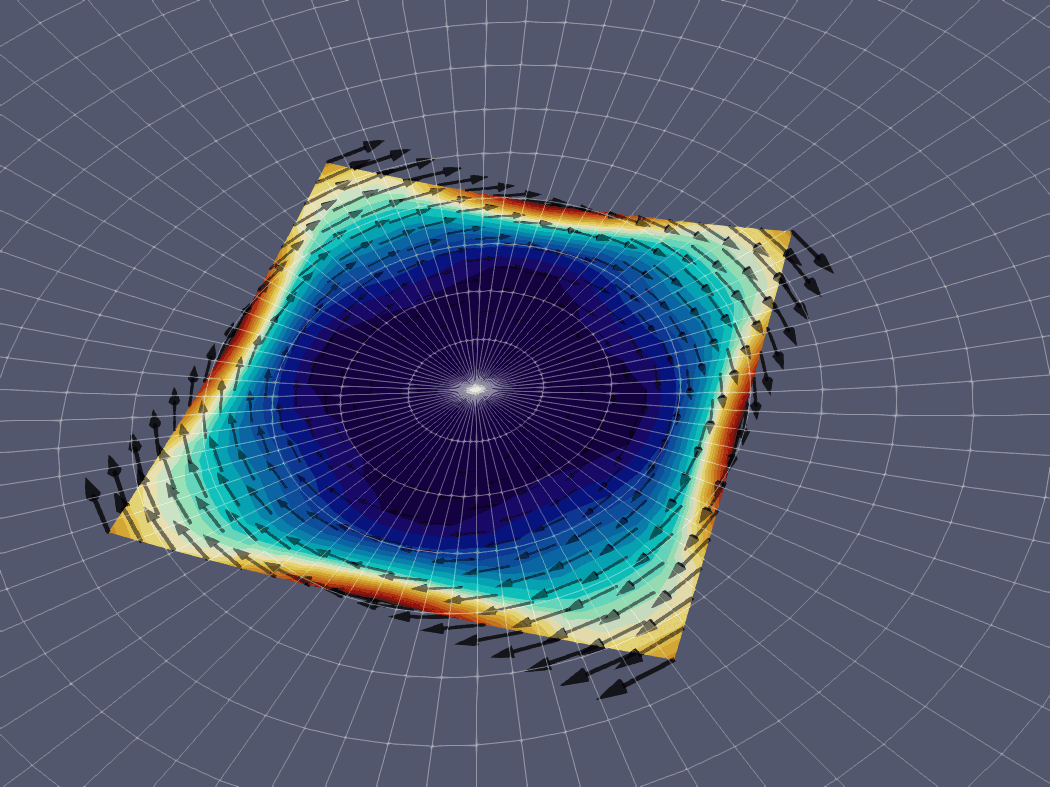}
\par\end{centering}
\caption{Left: The Lagrangian density of the body. Right: The vector field
corresponding to the expected, single zero eigenvalue of $D^{2}\mathcal{L}\left(\phi\right)$,
which is an infinitesimal rotation about the central axis of the surface
and is the generator of the symmetry group $\mathbb{S}^{1}$ of the
problem.}
\end{figure}
\par\end{center}

\begin{center}
\begin{figure}[H]
\begin{centering}
\includegraphics[width=3.5in]{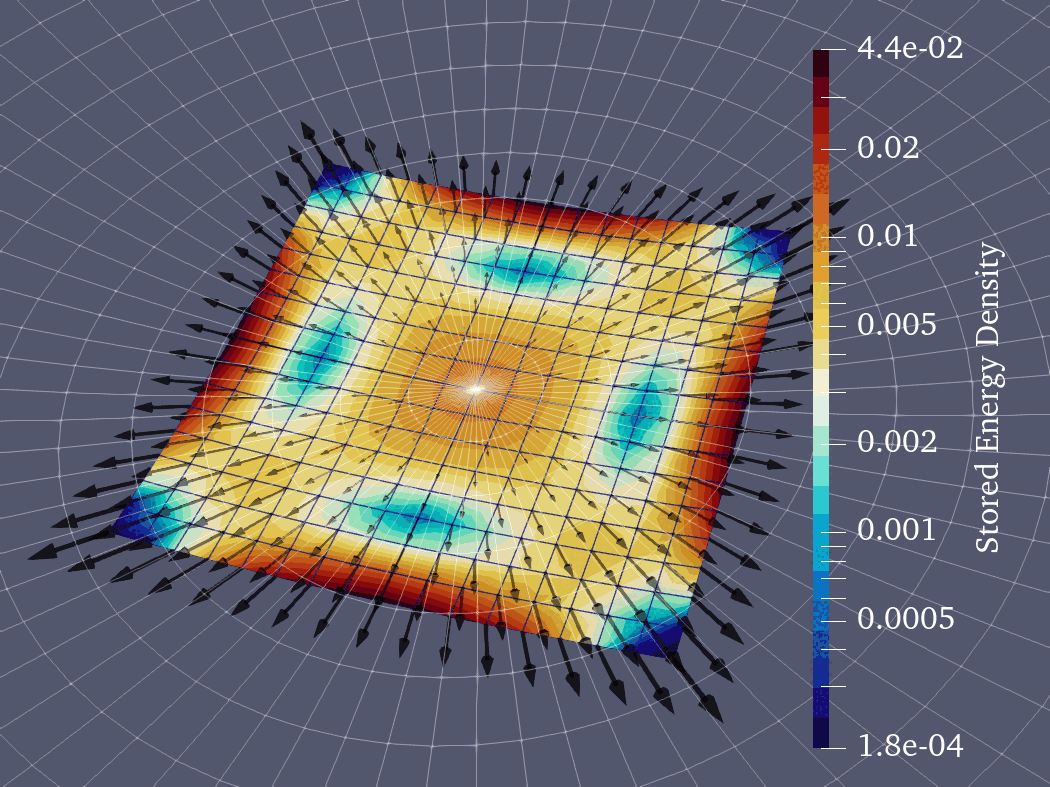}
\includegraphics[width=3.5in]{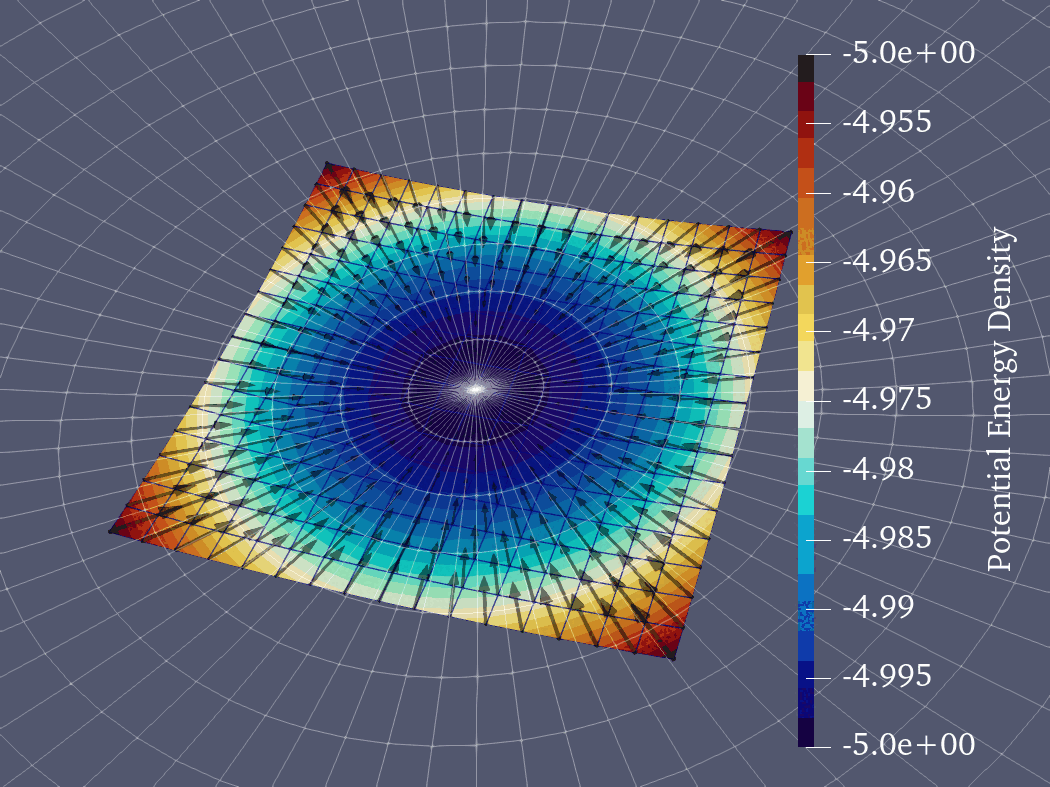}
\par\end{centering}
\caption{\protect\label{fig:spherical-cup-stored-energy-density-potential-energy-density}Left:
The stored energy density of the body, along with the force field
arising from it. Right: The potential energy density of the body,
along with the force field arising from it. These two fields perfectly
cancel each other to within numerical tolerance.}
\end{figure}
\par\end{center}

\begin{figure}[H]
\begin{centering}
\includegraphics[width=3.5in]{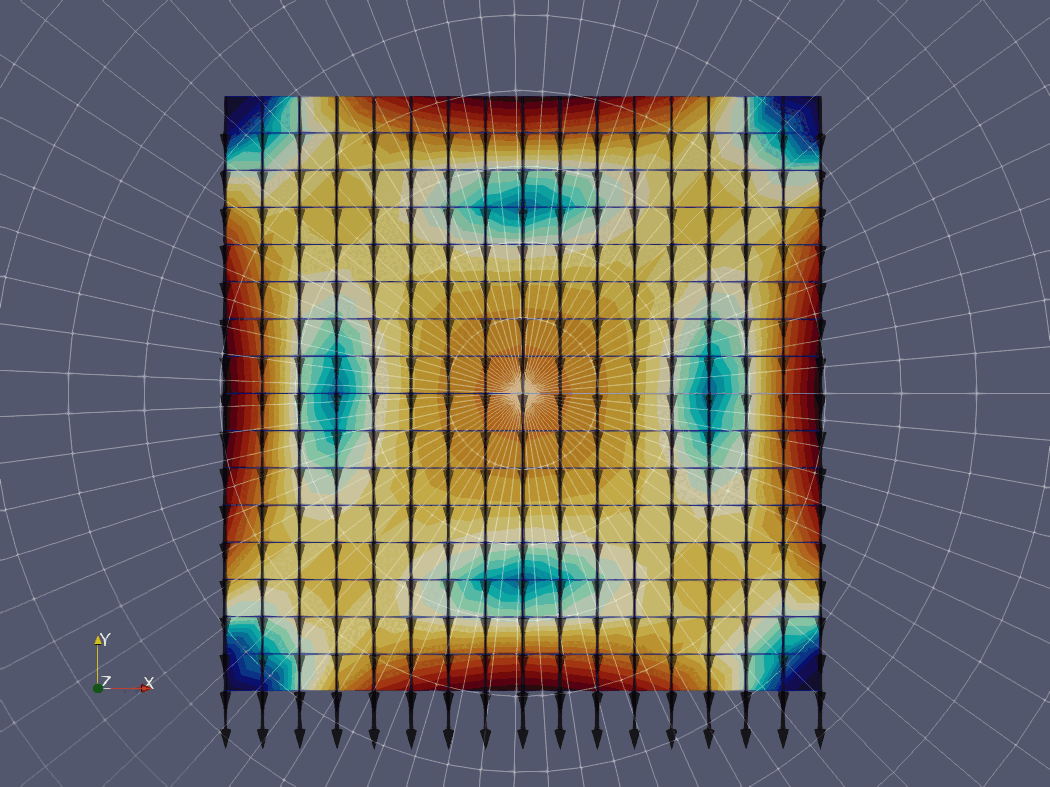}
\includegraphics[width=3.5in]{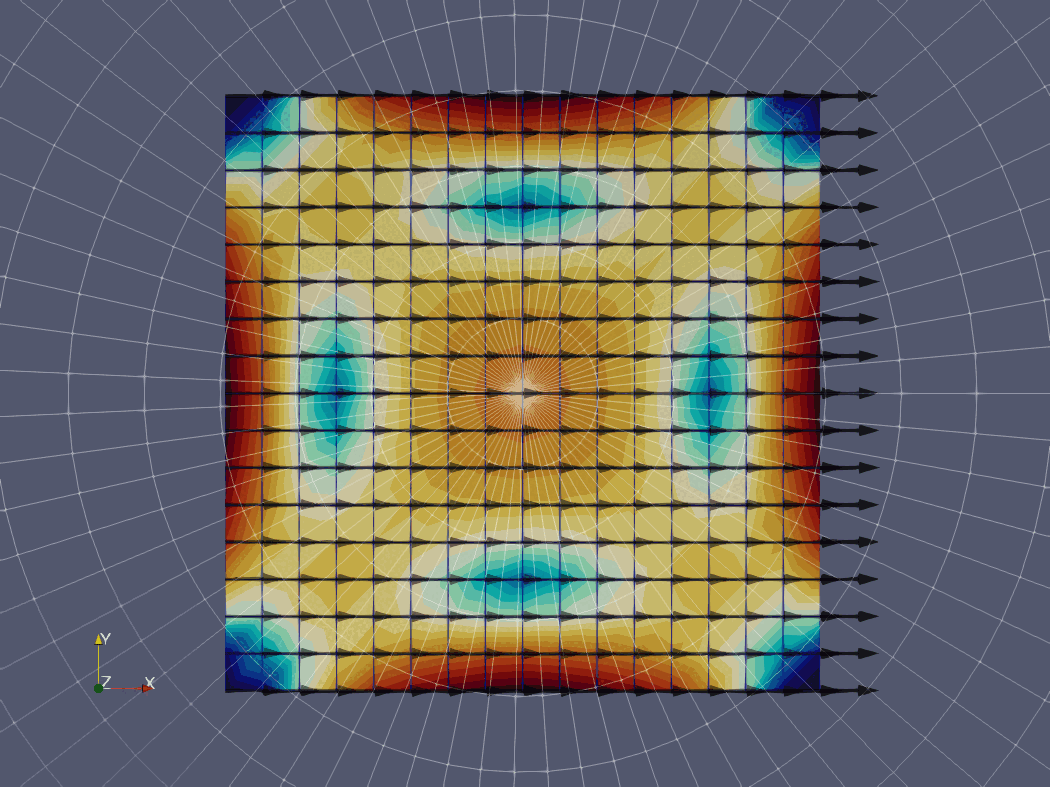}
\par\end{centering}
\begin{centering}
\includegraphics[width=3.5in]{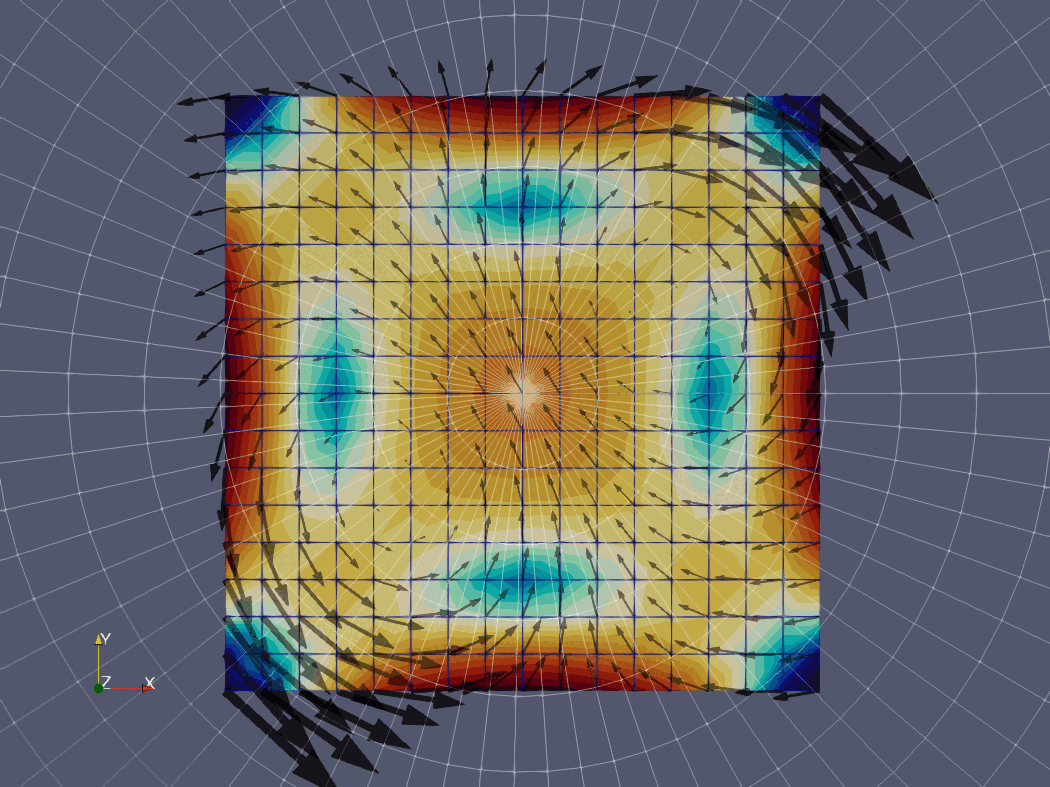}
\includegraphics[width=3.5in]{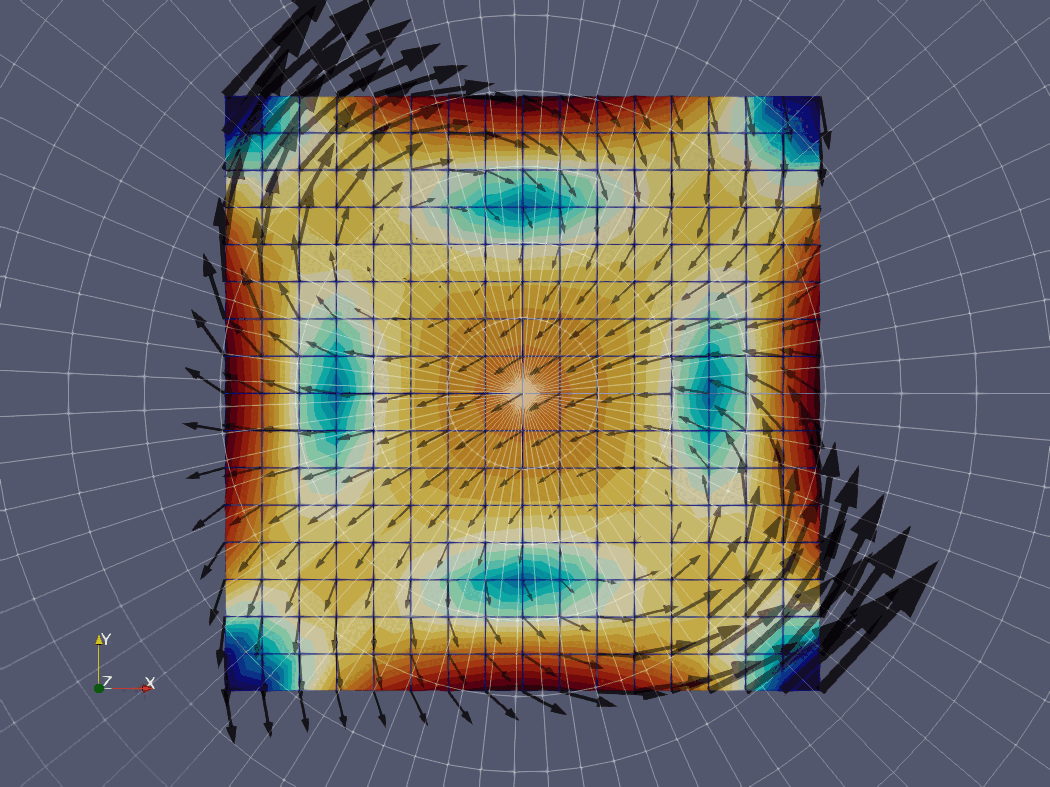}
\par\end{centering}
\caption{Top-down views of the vector fields corresponding to the smallest
four positive $D^{2}\mathcal{L}\left(\phi\right)$ eigenvalues.}
\end{figure}

\subsection{Flamm's Paraboloid Surface $z=2\sqrt{r_{s}\left(r-r_{s}\right)}$}

The Schwarzschild metric is a solution to Einstein's equations that
is spherically symmetric and features a central singularity -- a
point having mass $M$ that generates the curvature of its spacetime.
It is the archetypical model of a black hole. The Schwarzschild radius
$r_{s}=\frac{2GM}{c^{2}}$, where $G$ is Newton's gravitational constant
and $c$ is the speed of light, defines the black hole's event horizon,
inside which all time- and light-like trajectories inexorably lead
to the central singularity.

Flamm's paraboloid is an isometric immersion of a 2-dimensional, space-like
slice of the Schwarzschild spacetime (outside the event horizon) into
Euclidean space, giving a partial, visual representation of the spatial
curvature of that spacetime. The gravitational potential in the solution
presented here is $z$, just as in the other solutions, and is not
meant to model the actual gravity experienced in this spacetime. Flamm's
paraboloid has Gaussian curvature $K\left(r\right)=-\frac{r_{s}}{2r^{3}}$,
which is everywhere negative, and $\left|K\right|\to0$ monotonically
as $r\to\infty$, and thus the curvature levitator is possible anywhere
in this surface.

Here are plots showing the energy densities and vector fields that
describe important aspects of the solution. Modulo the expected symmetries,
$D^{2}\mathcal{L}\left(\phi\right)$ had only positive eigenvalues,
showing that the solution is stable.
\begin{center}
\begin{figure}[H]
\begin{centering}
\includegraphics[width=3.5in]{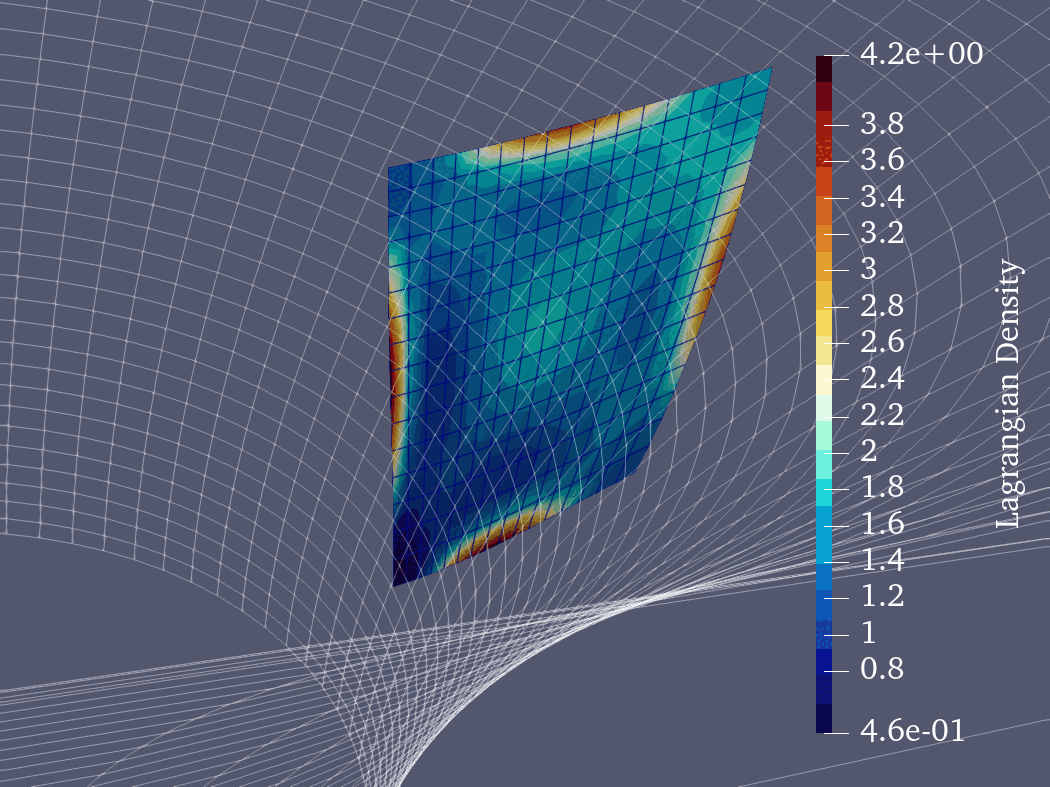}
\includegraphics[width=3.5in]{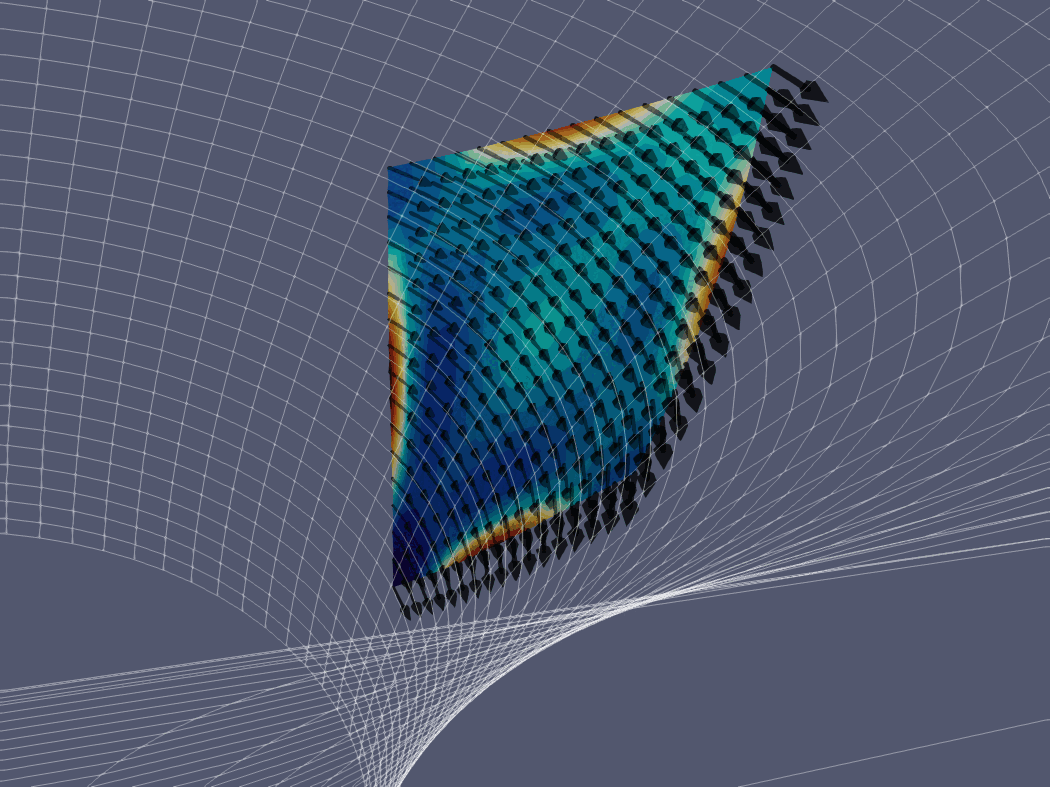}
\par\end{centering}
\caption{Left: The Lagrangian density of the body. Right: The vector field
corresponding to the expected, single zero eigenvalue of $D^{2}\mathcal{L}\left(\phi\right)$,
which is an infinitesimal rotation about the central axis of the surface
and is the generator of the symmetry group $\mathbb{S}^{1}$ of the
problem.}
\end{figure}
\par\end{center}

\begin{center}
\begin{figure}[H]
\begin{centering}
\includegraphics[width=3.5in]{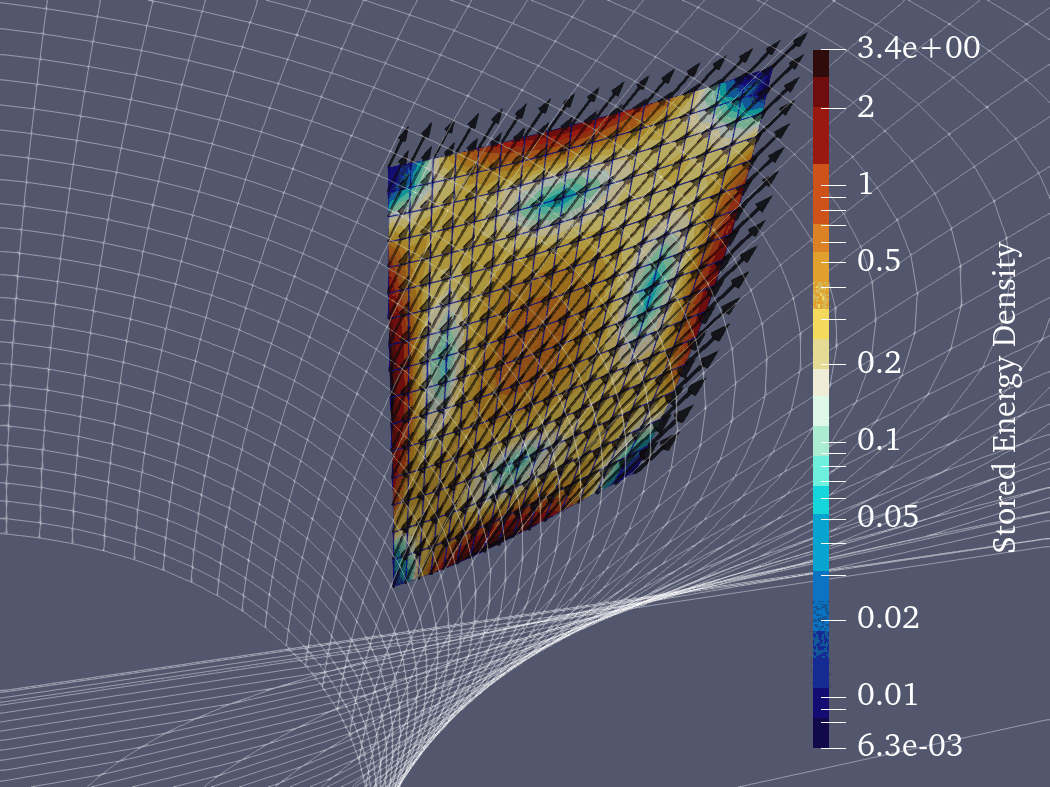}
\includegraphics[width=3.5in]{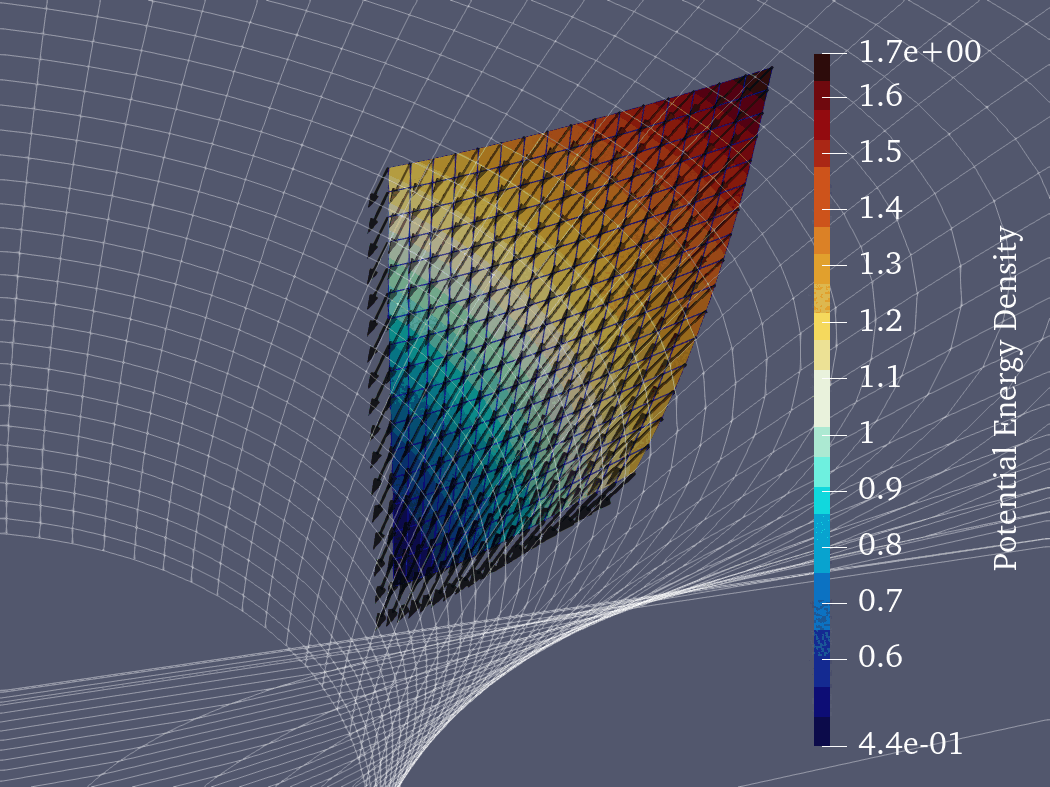}
\par\end{centering}
\caption{\protect\label{fig:spherical-cup-stored-energy-density-potential-energy-density-1}Left:
The stored energy density of the body, along with the force field
arising from it. Right: The potential energy density of the body,
along with the force field arising from it. These two fields perfectly
cancel each other to within numerical tolerance.}
\end{figure}
\par\end{center}

\begin{figure}[H]
\begin{centering}
\includegraphics[width=3.5in]{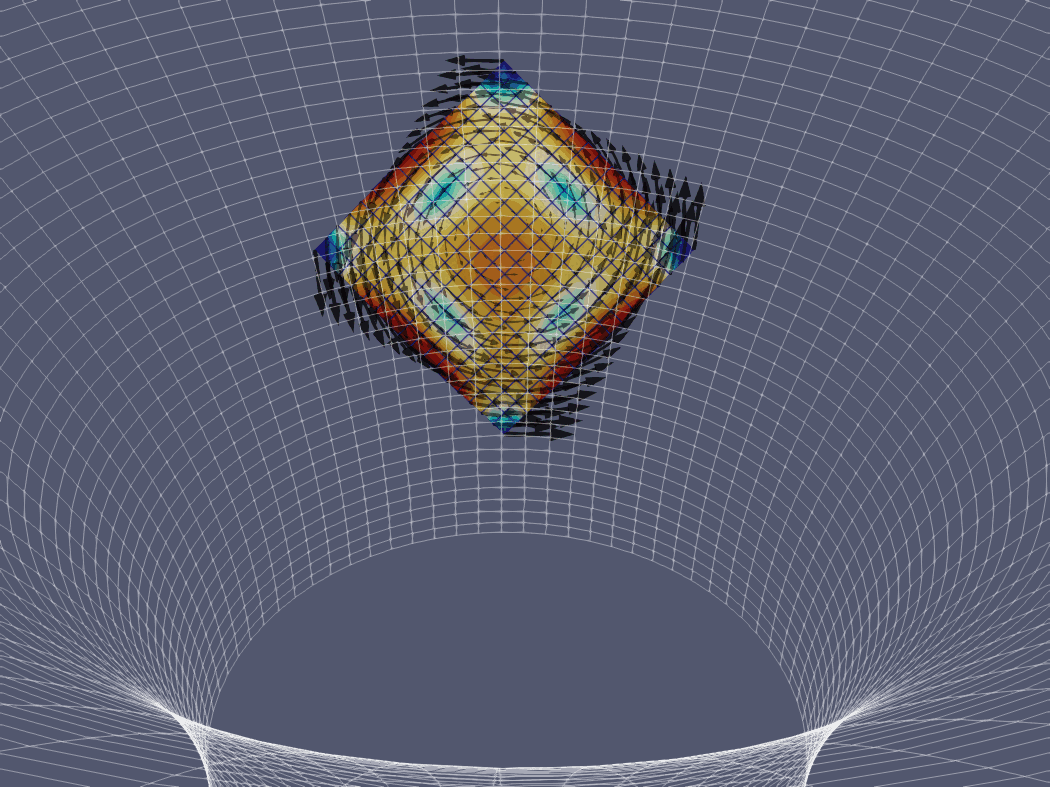}
\includegraphics[width=3.5in]{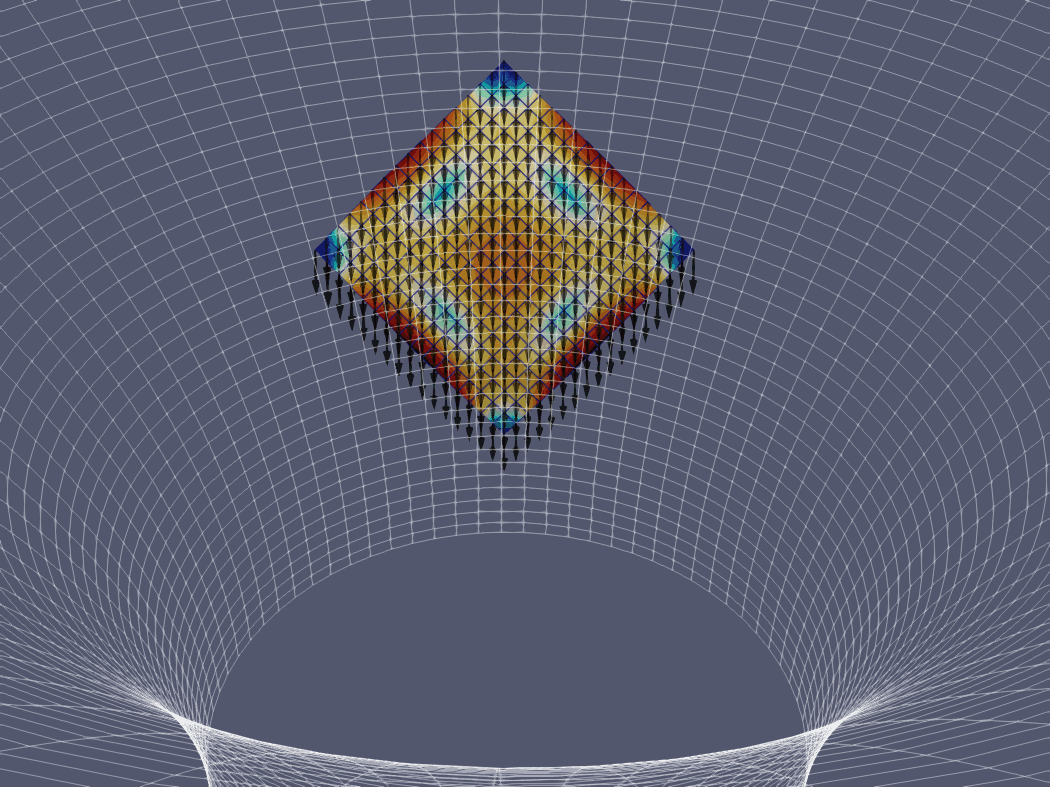}
\par\end{centering}
\begin{centering}
\includegraphics[width=3.5in]{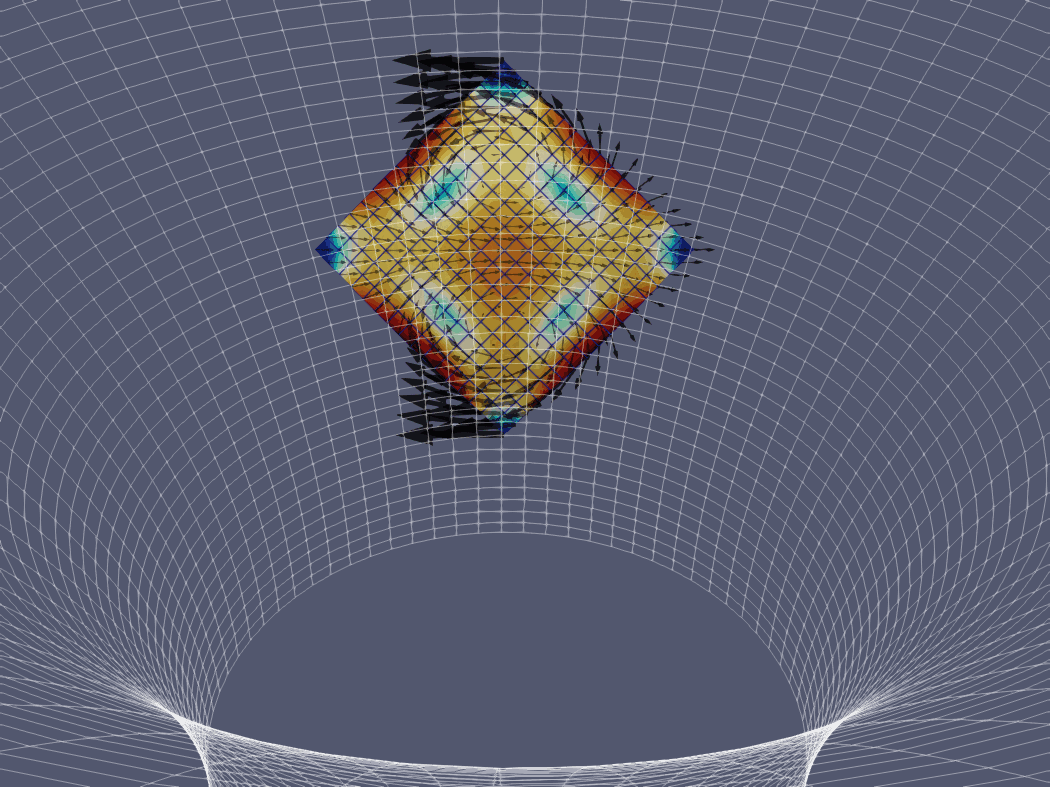}
\includegraphics[width=3.5in]{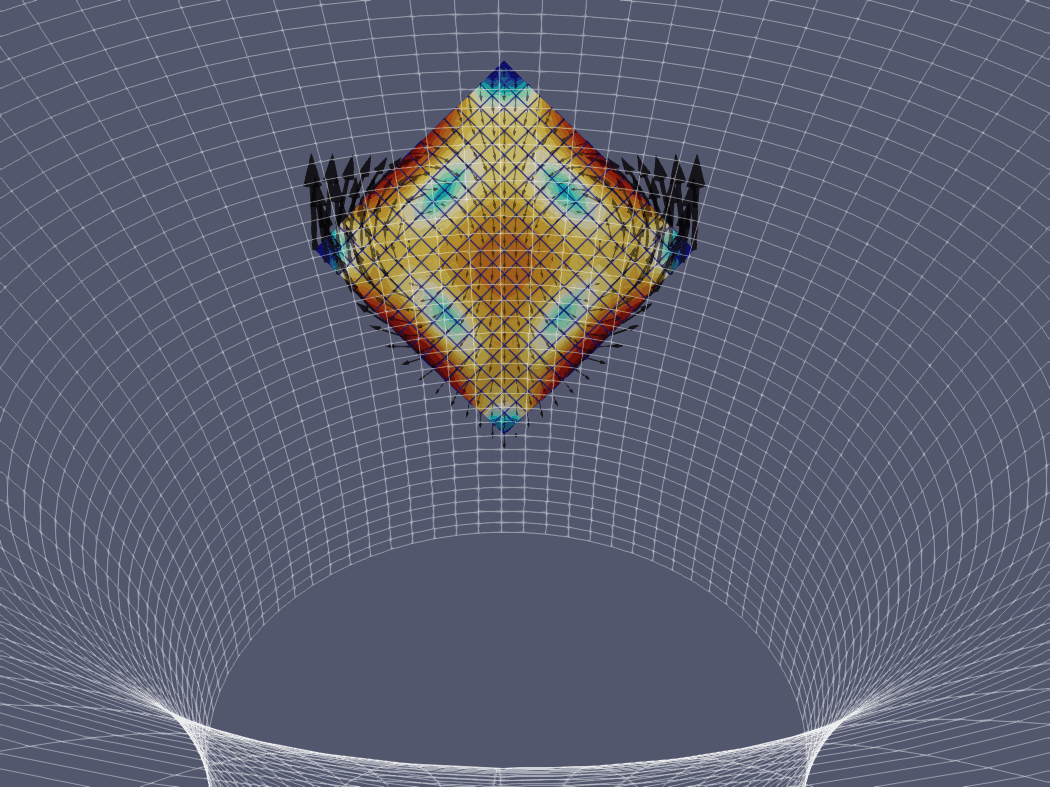}
\par\end{centering}
\caption{The vector fields corresponding to the smallest four positive $D^{2}\mathcal{L}\left(\phi\right)$
eigenvalues.}
\end{figure}

\section{Discussion}

Some of the statements in this section were derived from informal
reasoning or are otherwise speculative, and would make good subject
material for formal proof or numeric computation.

One salient feature of solid-body mechanics (this includes plastic,
elastic, and hyperelastic mechanics) is that \textbf{it takes work
to deform a body}. Because a hyperelastic body has a stored energy
function that defines its stress-strain response, i.e. is a conservative
system, work that increases the body's deformation goes into increasing
the body's stored energy. That energy is recovered as the body returns
toward its rest shape.

In the setting of Riemannian manifolds, a body's deformation can be
the result of simply being in a region of the manifold with shape
different from the rest shape of the body, completely absent of external,
spatial forces -- a phenomenon not found in the mechanics of Euclidean
spaces. This is the subject of the study of Incompatible Elasticity
\citet{kupferman_shamai_2012}. If the body moves toward a region
that deforms the body further from its rest shape, the body experiences
repulsive forces (a force field) that would help the body restore
its shape. This is what allows the curvature levitator phenomenon.
Similarly, it is possible for a freely-moving hyperelastic body to
\textquotedbl bounce off\textquotedbl{} of some region of space --
no contact with any other physical body and no external spatial forces
involved -- if that region is sufficiently different in shape than
the body and if the body's material is stiff enough. This is definitely
not a phenomenon found in Euclidean mechanics!

If the rest shape of the body were controllable, say by declaring
its rest shape to have constant curvature (Gaussian in the 2-dimensional
case, and sectional otherwise), then that curvature could be used
as a control variable to determine the \textquotedbl elevation\textquotedbl{}
at which its equilibrium is achieved. For example, in the funnel surface
$z=-r^{-1}$, if the body's rest shape were changed from zero-Gaussian-curvature
to positive-Gaussian-curvature, then the stored energy at any given
location would be higher, and the body's restorative forces would
be higher, pushing it away from the source of gravity. This could
fairly be called a \textbf{curvature elevator}.

More generally, if a body occupied a region of space with sufficiently
irregular curvature, and that body could control its rest shape with
sufficient refinement, it could potentially be made to generate forces
moving it in arbitrary directions. The control scheme would certainly
be extremely complicated, but if done, it could fairly be called a
\textbf{curvature locomotor}. 

It should be reiterated at this point that in order for a body to
generate curvature-induced forces that are commensurate with some
particular spatial force (e.g. to counteract a gravitational force),
the body's material must have sufficient stiffness -- in the case
of the material defined by $W$ from \subsecref{A-Physically-Reasonable-Material},
the stiffness is given by the \textquotedbl modulus of rigidity\textquotedbl{}
$\alpha$. Furthermore, in general the larger the body, the more deformation
it will experience, meaning that it can better take advantage of the
curvature-induced forces and its stiffness requirement will be less.
The smaller the body, the less deformation it will experience, meaning
that its stiffness requirement will be greater. In the limit, an infinitesimally-sized
body (i.e. a point) experiences no curvature-induced forces as there
is no possibility for deformation, and therefore is \textquotedbl just\textquotedbl{}
the particle mechanics of Riemannian manifolds.

Setting aside hyperelastic bodies for a moment, any deformable body
that requires work to deform will experience forces that resist its
change in shape. A plastic body is defined as not having a rest shape
-- it accepts deformations, but still requires work to do so. A plastic
body undergoing deformation as it moves through differently-curved
regions of a Riemannian manifold would experience not a bouncing-off-of,
but rather a velocity- and curvature-gradient-dependent resistance
as its rest shape changes. The energy put into deforming the body
would not be recoverable. Such a body moving under its own inertia
through differently-curved regions of a Riemannian manifold would
gradually slow down and heat up -- a kind of spatial-curvature-induced
\textquotedbl friction\textquotedbl{} -- another phenomenon not
encountered in Euclidean space.

A hyperelastic body, too, would be expected to encounter a kind of
friction, but for a different reason. As it moves through a region
of space with curvature taking it further from its rest shape, it
deforms, causing internal oscillations at the expense of some of the
kinetic energy that is pushing it into that region of space. Those
oscillations will bounce off the boundary of the body and will tend
to get more complicated (imagine circular waves bouncing off of a
polygonal/polyhedral boundary, the overall wave pattern getting more
complex with each bounce), gradually gaining higher-frequency modes
until most of the energy in the oscillations is \textquotedbl lost\textquotedbl{}
in the high frequencies. The kinetic energy in those oscillations
becomes incoherent and useless on the macro-scale -- it's heat. Thus
the body's \textquotedbl coherent\textquotedbl{} kinetic energy (i.e.
corresponding to motion in a particular overall direction) is irrecoverably
reduced, which can fairly be described as a kind of friction.

\section{Avenues for Future Work}
\begin{itemize}
\item Solve the curvature levitator problem in 3-dimensional Riemannian
manifolds.
\item Develop formal existence proofs for weak and strong curvature levitator
solutions in 2- and 3- dimensional spaces.
\item Implement body domains more complex than a rectangle. In particular,
making the body a disc would likely increase the \textquotedbl efficiency\textquotedbl{}
of the curvature levitator by getting rid of the corners of the body
-- as they do no useful work, based on the stored energy density
plot of the solution in this article. Furthermore, identify other,
internal regions of the body -- such as the four lighter regions
near the center of the body in the stored energy density plots of
\figref{funnel-stored-energy-density-potential-energy-density} and
\figref{paraboloid-stored-energy-density-potential-energy-density}
-- that are not performing much work and carve them out to make the
body lighter but without compromising its ability to resist deformation.
\item Solve the curvature elevator problem: Control the rest shape of a
body so that it can control its equilibrium elevation. Start with
constant Gaussian/sectional curvature as a control variable.
\item Solve the curvature locomotor problem: In a highly irregularly curved
manifold, control the rest shape of a body to move in arbitrary directions,
taking advantage of the curvature gradients in many directions.
\item Solve the relativistic version of this problem in 2+1 and 3+1 dimensions.
\item Implement faster and higher quality numerical optimization.
\item Implement the dynamic hyperelastic problem, ideally producing a realtime
simulation which is interactive. This would provide a great way to
gain intuition for the physics of hyperelastic bodies in Riemannian
manifolds, and finally find out what happens when you push that cubic
meter of jello into a wormhole!
\end{itemize}

\appendix

\section{Proofs\protect\label{sec:Proofs}}

The proofs in this section will make heavy use of the tensor calculus
formalism detailed in \citet{Dods2022}. It will also clarify the
exposition to use a subscript notation that indicates the specific
tensor factor(s) on which natural pairings are taken (see \citet[pg. 7]{Dods2022}).
For example,
\begin{align*}
\cdot_{\phi^{*}TS}\colon\phi^{*}T^{*}S\times\phi^{*}TS & \to\mathbb{R}, & \cdot_{\phi^{*}TS\otimes T^{*}B}\left(\phi^{*}T^{*}S\otimes TB\right)\times\left(\phi^{*}TS\otimes T^{*}B\right) & \to\mathbb{R},\\
\left(\alpha,x\right) & \mapsto\alpha\left(x\right), & \left(\alpha\otimes y,x\otimes\beta\right) & \mapsto\alpha\left(x\right)\beta\left(y\right).
\end{align*}

\subsection{Proofs for \Subsecref{Riemannian-Calculus-of-Variations} - Riemannian
Calculus of Variations\protect\label{subsec:Proofs-for-RCOV}}
\begin{proof}[Proof of \propref{first-variation-weak-form} - First variation of
general $\mathcal{L}$ in weak form]
Let $\psi\in\Gamma\left(\phi^{*}TS\right)$ denote a variation of
$\phi$. For brevity, let
\begin{align*}
\nabla\phi & :=\nabla^{B\to S}\phi\in\Gamma\left(\phi^{*}TS\otimes T^{*}B\right), & \nabla\psi & :=\nabla^{\phi^{*}TS}\psi\in\Gamma\left(\phi^{*}TS\otimes T^{*}B\right),
\end{align*}
\begin{align*}
L_{,\sigma}^{\nabla\phi} & :=\left(\nabla\phi\right)^{*}L_{,\sigma}\in\Gamma\left(\phi^{*}T^{*}S\right), & L_{,\beta}^{\nabla\phi} & :=\left(\nabla\phi\right)^{*}L_{,\beta}\in\Gamma\left(T^{*}B\right), & L_{,\Phi}^{\nabla\phi} & :=\left(\nabla\phi\right)^{*}L_{,\Phi}\in\Gamma\left(\phi^{*}T^{*}S\otimes TB\right).
\end{align*}
The equality to be proven is
\begin{align*}
D\mathcal{L}\left(\phi\right)\cdot_{\Gamma\left(\phi^{*}TS\right)}\psi & =\int_{B}L_{,\sigma}^{\nabla\phi}\cdot_{\phi^{*}TS}\psi+L_{,\Phi}^{\nabla\phi}\cdot_{\phi^{*}TS\otimes T^{*}B}\nabla\psi\,d\mu_{B}.
\end{align*}

Let $I\subset\mathbb{R}$ be an interval containing $0$, let $\epsilon$
be the standard coordinate on $\mathbb{R}$, and let $\delta_{\epsilon}:=\partial_{\epsilon}\mid_{\epsilon=0}$.
Let
\begin{align*}
z\colon B & \to B\times I,\\
b & \mapsto\left(b,0\right),
\end{align*}
denoting evaluation of the $I$ parameter at $\epsilon=0$, so that
$\delta_{\epsilon}=z^{*}\partial_{\epsilon}$. Let
\begin{align*}
\Psi\colon B\times I & \to S
\end{align*}
be a $C^{2}$ variation of $\phi$, meaning that $\Psi\left(\cdot,0\right)=\phi\left(\cdot\right)$
and $z^{*}\partial_{\epsilon}\Psi=\psi$. Let $p\colon B\times I\to B,\,\left(b,\epsilon\right)\mapsto b$.
Let $\partial_{B}$ denote differentiation along the $B$ component,
rendering a direction tensor factor of type $p^{*}T^{*}B$. This can
occur in a nonlinear covariant derivative (i.e. tangent map as tensor
field) as in $\partial_{B}\Psi\in\Gamma\left(\Psi^{*}TS\otimes p^{*}T^{*}B\right)$
or in a linear covariant derivative, for example, $\nabla_{\partial_{B}}\partial_{\epsilon}\Psi\in\Gamma\left(\Psi^{*}TS\otimes p^{*}T^{*}B\right)$.
Then
\begin{align*}
D\mathcal{L}\left(\phi\right)\cdot\psi & =D\mathcal{L}\left(\phi\right)\cdot\delta_{\epsilon}\Psi & \text{(definition of \ensuremath{\Psi})}\\
 & =\delta_{\epsilon}\left(\mathcal{L}\left(\Psi\right)\right) & \text{(definition of \ensuremath{D\mathcal{L}})}\\
 & =\int_{B}\delta_{\epsilon}\left(L\circ\partial_{B}\Psi\right)\,d\mu_{B}\\
 & =\int_{B}z^{*}\partial_{\epsilon}\left(L\circ\partial_{B}\Psi\right)\,d\mu_{B} & \text{(definition of \ensuremath{\delta_{\epsilon}})}\\
 & =\int_{B}z^{*}\left(\left(\partial_{B}\Psi\right)^{*}dL\cdot_{\left(\partial_{B}\Psi\right)^{*}TE}\partial_{\epsilon}\partial_{B}\Psi\right)\,d\mu_{B} & \text{(definition of \ensuremath{dL})}\\
 & =\int_{B}L_{,\sigma}^{\nabla\phi}\cdot_{\phi^{*}TS}\psi+L_{,\Phi}^{\nabla\phi}\cdot_{\phi^{*}TS\otimes T^{*}B}\nabla\psi\,d\mu_{B}, & \text{(supporting calculations)}
\end{align*}
as claimed. Supporting calculations follow. 

First, note that $\partial_{B}\Psi\circ z=\nabla\phi$, so
\begin{align*}
z^{*}\left(\partial_{B}\Psi\right)^{*}L_{,\sigma} & =\left(\partial_{B}\Psi\circ z\right)^{*}L_{,\sigma} & z^{*}\left(\partial_{B}\Psi\right)^{*}L_{,\beta} & =\left(\partial_{B}\Psi\circ z\right)^{*}L_{,\beta} & z^{*}\left(\partial_{B}\Psi\right)^{*}L_{,\Phi} & =\left(\partial_{B}\Psi\circ z\right)^{*}L_{,\Phi}\\
 & =\left(\nabla\phi\right)^{*}L_{,\sigma} &  & =\left(\nabla\phi\right)^{*}L_{,\beta} &  & =\left(\nabla\phi\right)^{*}L_{,\Phi}\\
 & =L_{,\sigma}^{\nabla\phi}, &  & =L_{,\beta}^{\nabla\phi}, &  & =L_{,\Phi}^{\nabla\phi}.
\end{align*}
Taking the integrand above,
\begin{align}
z^{*}\left(\left(\partial_{B}\Psi\right)^{*}dL\cdot_{\left(\partial_{B}\Psi\right)^{*}TE}\partial_{\epsilon}\partial_{B}\Psi\right) & =z^{*}\left(\left(\partial_{B}\Psi\right)^{*}\left(L_{,\sigma}\cdot_{\pi_{S}^{*}TS}\sigma+L_{,\beta}\cdot_{\pi_{B}^{*}TB}\beta+L_{,\Phi}\cdot_{\pi^{*}E}\Phi\right)\cdot_{\left(\partial_{B}\Psi\right)^{*}TE}\partial_{\epsilon}\partial_{B}\Psi\right)\label{eq:decomposed-dL}
\end{align}
by the partial covariant derivative decomposition of $dL$.

For the term involving $\sigma$ on the right hand side of \eqref{decomposed-dL},
\begin{align*}
\left(\partial_{B}\Psi\right)^{*}\sigma\cdot_{\left(\partial_{B}\Psi\right)^{*}TE}\partial_{\epsilon}\partial_{B}\Psi & =\left(\partial_{B}\Psi\right)^{*}\nabla^{E\to S}\pi_{S}\cdot_{\left(\partial_{B}\Psi\right)^{*}TE}\partial_{\epsilon}\partial_{B}\Psi & \text{(definition of \ensuremath{\sigma})}\\
 & =\partial_{\epsilon}\left(\pi_{S}\circ\partial_{B}\Psi\right) & \text{(definition of \ensuremath{\nabla^{E\to S}\pi_{S}})}\\
 & =\partial_{\epsilon}\Psi,
\end{align*}
and therefore
\begin{align*}
z^{*}\left(\left(\partial_{B}\Psi\right)^{*}\sigma\cdot_{\left(\partial_{B}\Psi\right)^{*}TE}\partial_{\epsilon}\partial_{B}\Psi\right) & =z^{*}\partial_{\epsilon}\Psi\\
 & =\delta_{\epsilon}\Psi & \text{(definition of \ensuremath{\delta_{\epsilon}})}\\
 & =\psi. & \text{(definition of \ensuremath{\Psi})}
\end{align*}
For the term involving $\beta$ on the right hand side of \eqref{decomposed-dL},
\begin{align*}
\left(\partial_{B}\Psi\right)^{*}\beta\cdot_{\left(\partial_{B}\Psi\right)^{*}TE}\partial_{\epsilon}\partial_{B}\Psi & =\left(\partial_{B}\Psi\right)^{*}\nabla^{E\to B}\pi_{B}\cdot_{\left(\partial_{B}\Psi\right)^{*}TE}\partial_{\epsilon}\partial_{B}\Psi & \text{(definition of \ensuremath{\beta})}\\
 & =\partial_{\epsilon}\left(\pi_{B}\circ\partial_{B}\Psi\right) & \text{(definition of \ensuremath{\nabla^{E\to B}\pi_{B}})}\\
 & =\partial_{\epsilon}\text{pr}_{B}^{B\times I}\\
 & =0 & \text{(\ensuremath{\text{pr}_{B}^{B\times I}} is indep. of \ensuremath{\epsilon}),}
\end{align*}
and therefore
\begin{align*}
z^{*}\left(\left(\partial_{B}\Psi\right)^{*}\beta\cdot_{\left(\partial_{B}\Psi\right)^{*}TE}\partial_{\epsilon}\partial_{B}\Psi\right) & =0.
\end{align*}
For the term involving $\Phi$ on the right hand side of \eqref{decomposed-dL},
let $\nabla_{\partial_{B}}^{\Psi^{*}TS}$ denote a covariant derivative
in the $B$ component only. Then 
\begin{align*}
z^{*}\left(\left(\partial_{B}\Psi\right)^{*}\Phi\cdot_{\left(\partial_{B}\Psi\right)^{*}TE}\partial_{\epsilon}\partial_{B}\Psi\right) & =z^{*}\left(\nabla_{\partial_{\epsilon}}^{\left(\pi\circ\partial_{B}\Psi\right)^{*}E}\partial_{B}\Psi\right) & \text{(definition of \ensuremath{\Phi})}\\
 & =z^{*}\left(\nabla_{\partial_{\epsilon}}^{\Psi^{*}TS\otimes T^{*}B}\partial_{B}\Psi\right) & \text{(tensor bundle isomorphism)}\\
 & =z^{*}\left(\nabla_{\partial_{B}}^{\Psi^{*}TS}\partial_{\epsilon}\Psi\right) & \text{(triv. bundle \ensuremath{S\times B\times I\to B\times I} is flat)}\\
 & =z^{*}\nabla^{\Psi^{*}TS}\partial_{\epsilon}\Psi\cdot_{z^{*}T\left(B\times I\right)}\nabla^{B\to B\times I}z & \text{(\ensuremath{\nabla^{B\to B\times I}z} embeds \ensuremath{TB} in \ensuremath{z^{*}T\left(B\times I\right)})}\\
 & =\nabla^{z^{*}\Psi^{*}TS}z^{*}\partial_{\epsilon}\Psi & \text{(definition of \ensuremath{\nabla^{z^{*}\Psi^{*}TS}})}\\
 & =\nabla^{\phi^{*}TS}\delta_{\epsilon}\Psi & \text{(definition of \ensuremath{\delta_{\epsilon}})}\\
 & =\nabla^{\phi^{*}TS}\psi & \text{(definition of \ensuremath{\Psi}).}
\end{align*}
Further detail on the equality $\nabla_{\partial_{\epsilon}}^{\Psi^{*}TS\otimes T^{*}B}\partial_{B}\Psi=\nabla_{\partial_{B}}^{\Psi^{*}TS}\partial_{\epsilon}\Psi$
can be found in \citet[pg. 36]{Dods2022}. Thus
\begin{align*}
z^{*}\left(\left(\partial_{B}\Psi\right)^{*}dL\cdot_{\left(\partial_{B}\Psi\right)^{*}TE}\partial_{\epsilon}\partial_{B}\Psi\right) & =L_{,\sigma}^{\nabla\phi}\cdot_{\phi^{*}TS}\psi+L_{,\beta}^{\nabla\phi}\cdot_{TB}0+L_{,\Phi}^{\nabla\phi}\cdot_{\phi^{*}TS\otimes T^{*}B}\nabla\psi\\
 & =L_{,\sigma}^{\nabla\phi}\cdot_{\phi^{*}TS}\psi+L_{,\Phi}^{\nabla\phi}\cdot_{\phi^{*}TS\otimes T^{*}B}\nabla\psi
\end{align*}
This completes the supporting calculations, and therefore the proof.
\end{proof}
\begin{proof}[Proof of \propref{first-variation-bulk-boundary-form} - First variation
of general $\mathcal{L}$ in bulk + boundary form]
Let $\iota\colon\partial B\to B$ denote the inclusion map, and let
$\nu\in\Gamma\left(\iota^{*}T^{*}B\right)$ be the outward unit conormal
field on $\partial B$. Using the same notation as in the previous
proof, the equality to be proven is
\begin{align*}
D\mathcal{L}\left(\phi\right)\cdot_{\Gamma\left(\phi^{*}TS\right)}\psi & =\int_{B}\left(L_{,\sigma}^{\nabla\phi}-\text{div}L_{,\Phi}^{\nabla\phi}\right)\cdot_{\phi^{*}TS}\psi\,d\mu_{B}+\int_{\partial B}\left(\iota^{*}L_{,\Phi}^{\nabla\phi}\cdot_{\iota^{*}T^{*}B}\nu\right)\cdot_{\iota^{*}\phi^{*}TS}\iota^{*}\psi\,d\mu_{\partial B}.
\end{align*}
Recall that $L_{,\Phi}^{\nabla\phi}\in\Gamma\left(\phi^{*}T^{*}S\otimes TB\right)$,
so
\begin{align*}
\nabla^{\phi^{*}T^{*}S\otimes TB}L_{,\phi}^{\nabla\phi} & \in\Gamma\left(\phi^{*}T^{*}S\otimes TB\otimes T^{*}B\right),
\end{align*}
and therefore
\begin{align*}
\text{div}L_{,\phi}^{\nabla\phi}=\text{tr}\nabla^{\phi^{*}T^{*}S\otimes TB}L_{,\phi}^{\nabla\phi} & \in\Gamma\left(\phi^{*}T^{*}S\right).
\end{align*}
Calculations ensue.
\begin{align*}
D\mathcal{L}\left(\phi\right)\cdot\psi & =\int_{B}L_{,\sigma}^{\nabla\phi}\cdot_{\phi^{*}TS}\psi+L_{,\Phi}^{\nabla\phi}\cdot_{\phi^{*}TS\otimes T^{*}B}\nabla\psi\,d\mu_{B} & \text{(by \propref{first-variation-weak-form})}\\
 & =\int_{B}L_{,\sigma}^{\nabla\phi}\cdot_{\phi^{*}TS}\psi+\text{div}\left(\psi\cdot_{\phi^{*}T^{*}S}L_{,\Phi}^{\nabla\phi}\right)-\left(\text{div}L_{,\Phi}^{\nabla\phi}\right)\cdot_{\phi^{*}TS}\psi\,d\mu_{B} & \text{(supporting calculations)}\\
 & =\int_{B}L_{,\sigma}^{\nabla\phi}\cdot_{\phi^{*}TS}\psi-\left(\text{div}L_{,\Phi}^{\nabla\phi}\right)\cdot_{\phi^{*}TS}\psi\,d\mu_{B}+\int_{\partial B}\iota^{*}\left(\psi\cdot_{\phi^{*}T^{*}S}L_{,\Phi}^{\nabla\phi}\right)\cdot\nu\,d\mu_{\partial B} & \text{(divergence theorem)}\\
 & =\int_{B}\left(L_{,\sigma}^{\nabla\phi}-\text{div}L_{,\Phi}^{\nabla\phi}\right)\cdot_{\phi^{*}TS}\psi\,d\mu_{B}+\int_{\partial B}\left(\iota^{*}L_{,\Phi}^{\nabla\phi}\cdot_{\iota^{*}T^{*}B}\nu\right)\cdot_{\iota^{*}\phi^{*}TS}\iota^{*}\psi\,d\mu_{\partial B},
\end{align*}
as claimed. Supporting calculations follow. Rearrangement of tensor
contractions and use of the product rule gives 
\begin{align*}
L_{,\Phi}^{\nabla\phi}\cdot_{\phi^{*}TS\otimes T^{*}B}\nabla\psi & =\text{tr}\nabla^{TB}\left(\psi\cdot_{\phi^{*}T^{*}S}L_{,\Phi}^{\nabla\phi}\right)-\psi\cdot_{\phi^{*}T^{*}S}\text{tr}\nabla^{\phi^{*}T^{*}S\otimes TB}L_{,\Phi}^{\nabla\phi}\\
 & =\text{div}\left(\psi\cdot_{\phi^{*}T^{*}S}L_{,\Phi}^{\nabla\phi}\right)-\left(\text{div}L_{,\Phi}^{\nabla\phi}\right)\cdot_{\phi^{*}TS}\psi.
\end{align*}
This concludes the proof.
\end{proof}
\begin{proof}[Proof of \propref{second-variation-weak-form} - Second variation
of general $\mathcal{L}$ in weak form]
Let $\phi$ satisfy $D\mathcal{L}\left(\phi\right)=0$. Let $\psi\in\Gamma\left(\phi^{*}TS\right)$.
Let $I,J\subset\mathbb{R}$ each be intervals containing $0$, and
let $\epsilon$ and $\eta$ be the standard coordinates on $I$ and
$J$ respectively. Let
\begin{align*}
z\colon B & \to B\times I\times J,\\
b & \mapsto\left(b,0,0\right),
\end{align*}
denoting \textquotedbl evaluation at $\epsilon=\eta=0$\textquotedbl .
Let
\begin{align*}
\Psi\colon B\times I\times J & \to S
\end{align*}
be a two-parameter variation of $\phi$ such that $z^{*}\partial_{\epsilon}\Psi=\psi$
and that $z^{*}\partial_{\eta}\Psi=\psi$, which in particular means
that $z^{*}\Psi=\phi$. For example, $\Psi\left(b,\epsilon,\eta\right):=\exp\left(\epsilon\psi\left(b\right)+\eta\psi\left(b\right)\right)$.

Let $p\colon B\times I\times J\to B,\,\left(b,\epsilon,\eta\right)\mapsto b$.
As in the proof of \propref{first-variation-weak-form}, let $\partial_{B}$
denote differentiation along the $B$ component, rendering a direction
tensor factor of type $p^{*}T^{*}B$. This can occur in a nonlinear
covariant derivative (i.e. tangent map as tensor field) as in $\partial_{B}\Psi\in\Gamma\left(\Psi^{*}TS\otimes p^{*}T^{*}B\right)$
or in a linear covariant derivative, for example, $\nabla_{\partial_{B}}\partial_{\epsilon}\Psi\in\Gamma\left(\Psi^{*}TS\otimes p^{*}T^{*}B\right)$.
Note that evaluation of such a field at $\epsilon=\eta=0$ causes
the types to collapse accordingly, for example, $z^{*}\partial_{B}\Psi=\nabla\phi\in\Gamma\left(\phi^{*}TS\otimes T^{*}B\right)$
and $z^{*}\nabla_{\partial_{B}}\partial_{\epsilon}\Psi=\nabla\psi\in\Gamma\left(\phi^{*}TS\otimes T^{*}B\right)$.

First, compute the part of integrand of the second variation $\nabla^{2}\mathcal{L}\left(\phi\right):\left(\psi\otimes\psi\right)=\int z^{*}\partial_{\epsilon}\partial_{\eta}\left(L\circ\partial_{B}\Psi\right)\,d\mu_{B}$
before evaluation at $\epsilon=\eta=0$. In particular, using calculations
from the proof of \propref{first-variation-weak-form},
\begin{align*}
\partial_{\epsilon}\partial_{\eta}\left(L\circ\partial_{B}\Psi\right) & =\partial_{\epsilon}\left(\left(\partial_{B}\Psi\right)^{*}L_{,\sigma}\cdot\partial_{\eta}\Psi+\left(\partial_{B}\Psi\right)^{*}L_{,\Phi}:\nabla_{\partial_{B}}^{\Psi^{*}TS}\partial_{\eta}\Psi\right)\\
 & =\left(\nabla_{\partial_{\epsilon}}^{\left(\partial_{B}\Psi\right)^{*}\pi_{S}^{*}T^{*}S}\left(\partial_{B}\Psi\right)^{*}L_{,\sigma}\right)\cdot\partial_{\eta}\Psi+\left(\partial_{B}\Psi\right)^{*}L_{,\sigma}\cdot\nabla_{\partial_{\epsilon}}^{\Psi^{*}TS}\partial_{\eta}\Psi\\
 & +\left(\nabla_{\partial_{\epsilon}}^{\left(\partial_{B}\Psi\right)^{*}\pi^{*}E}\left(\partial_{B}\Psi\right)^{*}L_{,\Phi}\right):\nabla_{\partial_{B}}^{\Psi^{*}TS}\partial_{\eta}\Psi+\left(\partial_{B}\Psi\right)^{*}L_{,\Phi}:\nabla_{\partial_{\epsilon}}^{\Psi^{*}TS\otimes p^{*}T^{*}B}\nabla_{\partial_{B}}^{\Psi^{*}TS}\partial_{\eta}\Psi & \text{(product rule)}\\
 & =\left[\begin{array}{cc}
\partial_{\eta}\Psi\cdot & \nabla_{\partial_{B}}^{\Psi^{*}TS}\partial_{\eta}\Psi:\end{array}\right]\left[\begin{array}{cc}
L_{,\sigma\sigma}^{\partial_{B}\Psi} & L_{,\sigma\Phi}^{\partial_{B}\Psi}\\
L_{,\Phi\sigma}^{\partial_{B}\Psi} & L_{,\Phi\Phi}^{\partial_{B}\Psi}
\end{array}\right]\left[\begin{array}{c}
\cdot\partial_{\epsilon}\Psi\\
:\nabla_{\partial_{B}}^{\Psi^{*}TS}\partial_{\epsilon}\Psi
\end{array}\right] & \text{(supp. calc. 1)}\\
 & +\left(\partial_{B}\Psi\right)^{*}L_{,\sigma}\cdot\nabla_{\partial_{\epsilon}}^{\Psi^{*}TS}\partial_{\eta}\Psi+\left(\partial_{B}\Psi\right)^{*}L_{,\Phi}:\nabla_{\partial_{B}}^{\Psi^{*}TS}\nabla_{\partial_{\epsilon}}^{\Psi^{*}TS}\partial_{\eta}\Psi\\
 & +\left(\partial_{B}\Psi\right)^{*}L_{,\Phi}:\left(\Psi^{*}R^{TS}\vdots\left(\partial_{\eta}\Psi\otimes\partial_{\epsilon}\Psi\otimes\partial_{B}\Psi\right)\right) & \text{(supp. calc. 4).}
\end{align*}
Let $\Theta:=\nabla_{\partial_{\epsilon}}^{\Psi^{*}TS}\partial_{\eta}\Psi\in\Gamma\left(\Psi^{*}TS\right)$
and let $\theta=z^{*}\Theta\in\Gamma\left(z^{*}\Psi^{*}TS\right)\cong\Gamma\left(\phi^{*}TS\right)$.
Note that $\theta$ is a variation of $\phi$. Thus
\begin{align*}
 & \int z^{*}\left(\left(\partial_{B}\Psi\right)^{*}L_{,\sigma}\cdot\nabla_{\partial_{\epsilon}}^{\Psi^{*}TS}\partial_{\eta}\Psi+\left(\partial_{B}\Psi\right)^{*}L_{,\Phi}:\nabla_{\partial_{B}}^{\Psi^{*}TS}\nabla_{\partial_{\epsilon}}^{\Psi^{*}TS}\partial_{\eta}\Psi\right)\,d\mu_{B}\\
 & =\int L_{,\sigma}^{\nabla\phi}\cdot\theta+L_{,\Phi}^{\nabla\phi}:\nabla\theta\,d\mu_{B}\\
 & =D\mathcal{L}\left(\phi\right)\cdot\theta\\
 & =0,
\end{align*}
where the last equality holds because $D\mathcal{L}\left(\phi\right)=0$
by assumption. Thus, evaluating $\partial_{\epsilon}\partial_{\eta}\left(L\circ\partial_{B}\Psi\right)$
at $\epsilon=\eta=0$ and integrating,
\begin{align*}
\nabla^{2}\mathcal{L}\left(\phi\right):\left(\psi\otimes\psi\right) & =\int_{B}z^{*}\partial_{\epsilon}\partial_{\eta}\left(L\circ\partial_{B}\Psi\right)\,d\mu_{B}\\
 & =\int_{B}\left[\begin{array}{cc}
\psi\cdot & \nabla\psi:\end{array}\right]\left[\begin{array}{cc}
L_{,\sigma\sigma}^{\nabla\phi} & L_{,\sigma\Phi}^{\nabla\phi}\\
L_{,\Phi\sigma}^{\nabla\phi} & L_{,\Phi\Phi}^{\nabla\phi}
\end{array}\right]\left[\begin{array}{c}
\cdot\psi\\
:\nabla\psi
\end{array}\right]+L_{,\Phi}^{\nabla\phi}:\left(\phi^{*}R^{TS}\vdots\left(\psi\otimes\psi\otimes\nabla\phi\right)\right)\,d\mu_{B},
\end{align*}
as claimed.

Supporting calculation 1:
\begin{align*}
 & \left(\nabla_{\partial_{\epsilon}}^{\left(\partial_{B}\Psi\right)^{*}\pi_{S}^{*}T^{*}S}\left(\partial_{B}\Psi\right)^{*}L_{,\sigma}\right)\cdot\partial_{\eta}\Psi+\left(\nabla_{\partial_{\epsilon}}^{\left(\partial_{B}\Psi\right)^{*}\pi^{*}E}\left(\partial_{B}\Psi\right)^{*}L_{,\Phi}\right):\nabla_{\partial_{B}}^{\Psi^{*}TS}\partial_{\eta}\Psi\\
 & =\left(\left(\partial_{B}\Psi\right)^{*}L_{,\sigma\sigma}\cdot_{\Psi^{*}TS}\partial_{\epsilon}\Psi+\left(\partial_{B}\Psi\right)^{*}L_{,\sigma\Phi}\cdot_{\Psi^{*}TS\otimes T^{*}B}\nabla_{\partial_{B}}^{\Psi^{*}TS}\partial_{\epsilon}\Psi\right)\cdot\partial_{\eta}\Psi & \text{(supp. calc. 2)}\\
 & +\left(\left(\partial_{B}\Psi\right)^{*}L_{,\Phi\sigma}\cdot_{\Psi^{*}TS}\partial_{\epsilon}\Psi+\left(\partial_{B}\Psi\right)^{*}L_{,\Phi\Phi}\cdot_{\Psi^{*}TS\otimes T^{*}B}\nabla_{\partial_{B}}^{\Psi^{*}TS}\partial_{\epsilon}\Psi\right):\nabla_{\partial_{B}}^{\Psi^{*}TS}\partial_{\eta}\Psi & \text{(supp. calc. 3)}\\
 & =\left[\begin{array}{cc}
\partial_{\eta}\Psi\cdot & \nabla_{\partial_{B}}^{\Psi^{*}TS}\partial_{\eta}\Psi:\end{array}\right]\left[\begin{array}{cc}
L_{,\sigma\sigma}^{\partial_{B}\Psi} & L_{,\sigma\Phi}^{\partial_{B}\Psi}\\
L_{,\Phi\sigma}^{\partial_{B}\Psi} & L_{,\Phi\Phi}^{\partial_{B}\Psi}
\end{array}\right]\left[\begin{array}{c}
\cdot\partial_{\epsilon}\Psi\\
:\nabla_{\partial_{B}}^{\Psi^{*}TS}\partial_{\epsilon}\Psi
\end{array}\right].
\end{align*}

Supporting calculation 2:
\begin{align*}
\nabla_{\partial_{\epsilon}}^{\left(\partial_{B}\Psi\right)^{*}\pi_{S}^{*}T^{*}S}\left(\partial_{B}\Psi\right)^{*}L_{,\sigma} & =\left(\partial_{B}\Psi\right)^{*}\nabla^{\pi_{S}^{*}T^{*}S}L_{,\sigma}\cdot_{\left(\partial_{B}E\right)^{*}TE}\nabla_{\partial_{\epsilon}}^{B\times I\times J\to E}\partial_{B}\Psi\\
 & =\left(\partial_{B}\Psi\right)^{*}\left(L_{,\sigma\sigma}\cdot_{\pi_{S}^{*}TS}\sigma+L_{,\beta}\cdot_{\pi_{B}^{*}TB}\beta+L_{,\Phi}\cdot_{\pi^{*}E}\right)\cdot_{\left(\partial_{B}E\right)^{*}TE}\nabla_{\partial_{\epsilon}}^{B\times I\times J\to E}\partial_{B}\Psi\\
 & =\left(\partial_{B}\Psi\right)^{*}L_{,\sigma\sigma}\cdot_{\Psi^{*}TS}\partial_{\epsilon}\Psi+\left(\partial_{B}\Psi\right)^{*}L_{,\sigma\Phi}\cdot_{\Psi^{*}TS\otimes T^{*}B}\nabla_{\partial_{B}}^{\Psi^{*}TS}\partial_{\epsilon}\Psi.
\end{align*}

Supporting calculation 3:
\begin{align*}
\nabla_{\partial_{\epsilon}}^{\left(\partial_{B}\Psi\right)^{*}\pi^{*}E}\left(\partial_{B}\Psi\right)^{*}L_{,\Phi} & =\left(\partial_{B}\Psi\right)^{*}\nabla^{\pi^{*}E}L_{,\Phi}\cdot_{\left(\partial_{B}E\right)^{*}TE}\nabla_{\partial_{\epsilon}}^{B\times I\times J\to E}\partial_{B}\Psi\\
 & =\left(\partial_{B}\Psi\right)^{*}\left(L_{,\Phi\sigma}\cdot_{\pi_{S}^{*}TS}\sigma+L_{,\Phi\beta}\cdot_{\pi_{B}^{*}TB}\beta+L_{,\Phi\Phi}\cdot_{\pi^{*}E}\Phi\right)\cdot_{\left(\partial_{B}E\right)^{*}TE}\nabla_{\partial_{\epsilon}}^{B\times I\times J\to E}\partial_{B}\Psi\\
 & =\left(\partial_{B}\Psi\right)^{*}L_{,\Phi\sigma}\cdot_{\Psi^{*}TS}\partial_{\epsilon}\Psi+\left(\partial_{B}\Psi\right)^{*}L_{,\Phi\Phi}\cdot_{\Psi^{*}TS\otimes T^{*}B}\nabla_{\partial_{B}}^{\Psi^{*}TS}\partial_{\epsilon}\Psi
\end{align*}

Supporting calculation 4: Note that $\partial_{\epsilon}$ is independent
of \textquotedbl$\partial_{B}$\textquotedbl{} (which represents
a vector field \textquotedbl along $B$\textquotedbl{} independent
of $\epsilon$ and $\eta$), so $\left[\partial_{\epsilon},\partial_{B}\right]=0$,
\begin{align*}
 & \nabla_{\partial_{\epsilon}}^{\Psi^{*}TS\otimes p^{*}T^{*}B}\nabla_{\partial_{B}}^{\Psi^{*}TS}\partial_{\eta}\Psi\\
 & =\nabla_{\partial_{\epsilon}}^{\Psi^{*}TS\otimes p^{*}T^{*}B}\nabla_{\partial_{B}}^{\Psi^{*}TS}\partial_{\eta}\Psi-\nabla_{\partial_{B}}^{\Psi^{*}TS}\nabla_{\partial_{\epsilon}}^{\Psi^{*}TS}\partial_{\eta}\Psi+\nabla_{\partial_{B}}^{\Psi^{*}TS}\nabla_{\partial_{\epsilon}}^{\Psi^{*}TS}\partial_{\eta}\Psi\\
 & =\nabla_{\partial_{\epsilon}}^{\Psi^{*}TS\otimes p^{*}T^{*}B}\nabla_{\partial_{B}}^{\Psi^{*}TS}\partial_{\eta}\Psi-\nabla_{\partial_{B}}^{\Psi^{*}TS}\nabla_{\partial_{\epsilon}}^{\Psi^{*}TS}\partial_{\eta}\Psi-\nabla_{\left[\partial_{\epsilon},\partial_{B}\right]}^{\Psi^{*}TS}\partial_{\eta}\Psi+\nabla_{\partial_{B}}^{\Psi^{*}TS}\nabla_{\partial_{\epsilon}}^{\Psi^{*}TS}\partial_{\eta}\Psi\\
 & =R_{\text{op}}^{\Psi^{*}TS}\left(\partial_{\epsilon},\partial_{B}\right)\partial_{\eta}\Psi+\nabla_{\partial_{B}}^{\Psi^{*}TS}\nabla_{\partial_{\epsilon}}^{\Psi^{*}TS}\partial_{\eta}\Psi\\
 & =\Psi^{*}R^{TS}\vdots\left(\partial_{\eta}\Psi\otimes\partial_{\epsilon}\Psi\otimes\partial_{B}\Psi\right)+\nabla_{\partial_{B}}^{\Psi^{*}TS}\nabla_{\partial_{\epsilon}}^{\Psi^{*}TS}\partial_{\eta}\Psi & \text{(supp. calc. 5)}\\
\implies & \left(\partial_{B}\Psi\right)^{*}L_{,\Phi}:\nabla_{\partial_{\epsilon}}^{\Psi^{*}TS\otimes p^{*}T^{*}B}\nabla_{\partial_{B}}^{\Psi^{*}TS}\partial_{\eta}\Psi\\
 & =\left(\partial_{B}\Psi\right)^{*}L_{,\Phi}:\nabla_{\partial_{B}}^{\Psi^{*}TS}\nabla_{\partial_{\epsilon}}^{\Psi^{*}TS}\partial_{\eta}\Psi+\left(\partial_{B}\Psi\right)^{*}L_{,\Phi}:\Psi^{*}R^{TS}\vdots\left(\partial_{\eta}\Psi\otimes\partial_{\epsilon}\Psi\otimes\partial_{B}\Psi\right).
\end{align*}

Supporting calculation 5: Note that
\begin{align*}
R^{\Psi^{*}TS}=\Psi^{*}R^{TS}:\left(\nabla^{B\times I\times J\to S}\Psi\boxtimes\nabla^{B\times I\times J\to S}\Psi\right) & \in\Gamma\left(\Psi^{*}TS\otimes\Psi^{*}T^{*}S\otimes T^{*}\left(B\times I\times J\right)\otimes T^{*}\left(B\times I\times J\right)\right),
\end{align*}
where if $A$ and $B$ are each 2-tensor {[}fields{]}, then $A\boxtimes B$
is a 4-tensor {[}field{]} defined by $\left(A\boxtimes B\right)^{ijk\ell}:=A^{ik}B^{j\ell}$
(see \citet[pg. 38]{Dods2022}). Embedding $X\in\Gamma\left(TB\right)$
into $\Gamma\left(T\left(B\times I\times J\right)\right)\cong\Gamma\left(TB\oplus TI\oplus TJ\right)$
as $Y\left(b,\epsilon,\eta\right):=X\left(b\right)\oplus0\oplus0$,
it follows that $\left[\partial_{\epsilon},Y\right]=0$, and
\begin{align*}
R_{\text{op}}^{\Psi^{*}TS}\left(\partial_{\epsilon},Y\right)\partial_{\eta}\Psi & =R^{\Psi^{*}TS}\vdots\left(\partial_{\eta}\Psi\otimes\partial_{\epsilon}\otimes Y\right)\\
 & =\left(\Psi^{*}R^{TS}:\left(\nabla^{B\times I\times J\to S}\Psi\boxtimes\nabla^{B\times I\times J\to S}\Psi\right)\right)\vdots\left(\partial_{\eta}\Psi\otimes\partial_{\epsilon}\otimes Y\right)\\
 & =\left(\Psi^{*}R^{TS}:\left(\partial_{\epsilon}\Psi\otimes\nabla_{Y}\Psi\right)\right)\cdot\partial_{\eta}\Psi\\
 & =\Psi^{*}R^{TS}\vdots\left(\partial_{\eta}\Psi\otimes\partial_{\epsilon}\Psi\otimes\nabla_{Y}\Psi\right)\\
 & =\Psi^{*}R^{TS}\vdots\left(\partial_{\eta}\Psi\otimes\partial_{\epsilon}\Psi\otimes\partial_{B}\Psi\right)\cdot Y\\
\implies R_{\text{op}}^{\Psi^{*}TS}\left(\partial_{\epsilon},\partial_{B}\right)\partial_{\eta}\Psi & =\Psi^{*}R^{TS}\vdots\left(\partial_{\eta}\Psi\otimes\partial_{\epsilon}\Psi\otimes\partial_{B}\Psi\right)
\end{align*}
This concludes the supporting calculations, and therefore the proof.
\end{proof}

\subsection{Proofs for \Subsecref{Hyperelastic-Mechanics} - Hyperelastic Mechanics\protect\label{subsec:Proofs-for-Hyperelastic-Mechanics}}
\begin{proof}[Proof of \corref{mechanical-first-variation-in-weak-form} - First
variation of specific $\mathcal{L}$ in weak form]
Using the typed subscript notation for natural pairings on tensor
factors, the equality to prove is
\begin{align*}
D\mathcal{L}\left(\phi\right)\cdot\psi & =\int_{B}\left(W_{,\sigma}^{\nabla\phi}+\rho\phi^{*}dU\right)\cdot_{\phi^{*}TS}\psi+W_{,\Phi}^{\nabla\phi}\cdot_{\phi^{*}TS\otimes T^{*}B}\nabla\psi\,d\mu_{B},
\end{align*}
where
\begin{align*}
L & :=W+V\circ\pi.
\end{align*}
This follows easily from \propref{first-variation-weak-form} after
some straightforward calculations to determine $L_{,\sigma}$ and
$L_{,\Phi}$. It will help to express $L$ in a form more conducive
to the tensor calculus formalism discussed in \subsecref{Riemannian-Calculus-of-Variations}.
\begin{align*}
L & =W+V\circ\pi\\
 & =W+\pi^{*}V\\
 & =W+\left(\pi_{S}^{*}U\right)\left(\pi_{B}^{*}\rho\right).
\end{align*}
Starting with the partial covariant derivative decomposition of $L$,
\begin{align*}
L_{,\sigma}\cdot\sigma+L_{,\beta}\cdot\beta+L_{,\Phi}:\Phi & =dL\\
 & =dW+d\left(\left(\pi_{S}^{*}U\right)\left(\pi_{B}^{*}\rho\right)\right)\\
 & =dW+\left(\pi_{B}^{*}\rho\right)\left(\pi_{S}^{*}dU\cdot\nabla^{E\to S}\pi_{S}\right)+\left(\pi_{S}^{*}U\right)\left(\pi_{B}^{*}d\rho\cdot\nabla^{E\to B}\pi_{B}\right)\\
 & =W_{,\sigma}\cdot\sigma+W_{,\beta}\cdot\beta+W_{,\Phi}:\Phi+\left(\pi_{B}^{*}\rho\right)\left(\pi_{S}^{*}dU\cdot\sigma\right)+\left(\pi_{S}^{*}U\right)\left(\pi_{B}^{*}d\rho\cdot\beta\right)\\
 & =\left(W_{,\sigma}+\left(\pi_{B}^{*}\rho\right)\left(\pi_{S}^{*}dU\right)\right)\cdot\sigma+\left(W_{,\beta}+\left(\pi_{S}^{*}U\right)\left(\pi_{B}^{*}d\rho\right)\right)\cdot\beta+W_{,\Phi}:\Phi.
\end{align*}
Because $\sigma\oplus\beta\oplus\Phi$ is a vector bundle isomorphism,
this implies that
\begin{align*}
L_{,\sigma} & =W_{,\sigma}+\left(\pi_{B}^{*}\rho\right)\left(\pi_{S}^{*}dU\right), & L_{,\beta} & =W_{,\beta}+\left(\pi_{S}^{*}U\right)\left(\pi_{B}^{*}d\rho\right), & L_{,\Phi} & =W_{,\Phi},
\end{align*}
and therefore
\begin{align*}
L_{,\sigma}^{\nabla\phi} & =W_{,\sigma}^{\nabla\phi}+\left(\nabla\phi\right)^{*}\left(\left(\pi_{B}^{*}\rho\right)\left(\pi_{S}^{*}dU\right)\right) & \text{(definition of \ensuremath{\cdot^{\nabla\phi}} superscript)}\\
 & =W_{,\sigma}^{\nabla\phi}+\left(\left(\nabla\phi\right)^{*}\pi_{B}^{*}\rho\right)\left(\left(\nabla\phi\right)^{*}\pi_{S}^{*}dU\right) & \text{(distribute pullback)}\\
 & =W_{,\sigma}^{\nabla\phi}+\left(\left(\pi_{B}\circ\nabla\phi\right)^{*}\rho\right)\left(\left(\pi_{S}\circ\nabla\phi\right)^{*}dU\right) & \text{(pullback type isomorphism)}\\
 & =W_{,\sigma}^{\nabla\phi}+\rho\phi^{*}dU & \text{(\ensuremath{\pi_{B}\circ\nabla\phi=\text{Id}_{B},} \ensuremath{\pi_{S}\circ\nabla\phi=\phi}),}\\
L_{,\Phi}^{\nabla\phi} & =W_{,\Phi}^{\nabla\phi}.
\end{align*}
Therefore, by \propref{first-variation-weak-form},
\begin{align*}
D\mathcal{L}\left(\phi\right)\cdot_{\Gamma\left(\phi^{*}TS\right)}\psi & =\int_{B}L_{,\sigma}^{\nabla\phi}\cdot_{\phi^{*}TS}\psi+L_{,\Phi}^{\nabla\phi}\cdot_{\phi^{*}TS\otimes T^{*}B}\nabla\psi\,d\mu_{B}\\
 & =\int_{B}\left(W_{,\sigma}^{\nabla\phi}+\rho\phi^{*}dU\right)\cdot_{\phi^{*}TS}\psi+W_{,\Phi}^{\nabla\phi}\cdot_{\phi^{*}TS\otimes T^{*}B}\nabla\psi\,d\mu_{B},
\end{align*}
as claimed.
\end{proof}
\begin{proof}[Proof of \corref{mechanical-first-variation-in-bulk-boundary-form}
- First variation of specific $\mathcal{L}$ in bulk + boundary form]
This follows immediately, having already computed $L_{,\sigma}^{\nabla\phi}$
and $L_{,\Phi}^{\nabla\phi}$ in the previous proof.
\end{proof}
\begin{proof}[Proof of \corref{mechanical-second-variation-in-weak-form} - Second
variation of specific $\mathcal{L}$ in weak form]
This involves a straightforward set of calculations. Recall that
$L_{,\sigma}=W_{,\sigma}+\left(\pi_{B}^{*}\rho\right)\left(\pi_{S}^{*}dU\right).$
Thus 
\begin{align*}
L_{,\sigma\sigma}\cdot\sigma+L_{,\sigma\beta}\cdot\beta+L_{,\sigma\Phi}:\Phi & =\nabla L_{,\sigma}\\
 & =\nabla W_{,\sigma}+\nabla\left(\left(\pi_{B}^{*}\rho\right)\left(\pi_{S}^{*}dU\right)\right)\\
 & =\nabla W_{,\sigma}+\left(\pi_{S}^{*}dU\right)\nabla\left(\pi_{B}^{*}\rho\right)+\left(\pi_{B}^{*}\rho\right)\nabla\left(\pi_{S}^{*}dU\right)\\
 & =\nabla W_{,\sigma}+\left(\pi_{S}^{*}dU\right)\pi_{B}^{*}d\rho\cdot\nabla\pi_{B}+\left(\pi_{B}^{*}\rho\right)\pi_{S}^{*}\nabla^{2}U\cdot\nabla\pi_{S}\\
 & =W_{,\sigma\sigma}\cdot\sigma+W_{,\sigma\beta}\cdot\beta+W_{,\sigma\Phi}:\Phi+\left(\pi_{S}^{*}dU\right)\pi_{B}^{*}d\rho\cdot\beta+\left(\pi_{B}^{*}\rho\right)\pi_{S}^{*}\nabla^{2}U\cdot\sigma\\
 & =\left(W_{,\sigma\sigma}+\left(\pi_{B}^{*}\rho\right)\pi_{S}^{*}\nabla^{2}U\right)\cdot\sigma+\left(W_{,\sigma\beta}+\left(\pi_{S}^{*}dU\right)\pi_{B}^{*}d\rho\right)\cdot\beta+W_{,\sigma\Phi}:\Phi,
\end{align*}
implying that
\begin{align*}
L_{,\sigma\sigma} & =W_{,\sigma\sigma}+\left(\pi_{B}^{*}\rho\right)\pi_{S}^{*}\nabla^{2}U,\\
L_{,\sigma\Phi} & =W_{,\sigma\Phi}.
\end{align*}
By the symmetry of $\nabla^{2}L$, it follows that $L_{,\Phi\sigma}=W_{,\Phi\sigma}$.
Finally, recalling that $L_{,\Phi}=W_{,\Phi}$, it follows that
\begin{align*}
L_{,\Phi\Phi} & =W_{,\Phi\Phi}.
\end{align*}
Taking the pullback of these fields by $\nabla\phi$,
\begin{align*}
L_{,\sigma\sigma}^{\nabla\phi} & =W_{,\sigma\sigma}^{\nabla\phi}+\left(\left(\nabla\phi\right)^{*}\pi_{B}^{*}\rho\right)\left(\left(\nabla\phi\right)^{*}\pi_{S}^{*}\nabla^{2}U\right)\\
 & =W_{,\sigma\sigma}^{\nabla\phi}+\rho\phi^{*}\nabla^{2}U,\\
L_{,\sigma\Phi}^{\nabla\phi} & =W_{,\sigma\Phi}^{\nabla\phi},\\
L_{,\Phi\sigma}^{\nabla\phi} & =W_{,\Phi\sigma}^{\nabla\phi},\\
L_{,\Phi\Phi}^{\nabla\phi} & =W_{,\Phi\Phi}^{\nabla\phi}.
\end{align*}
Therefore the second variation of $\mathcal{L}$ is
\begin{align*}
\nabla^{2}\mathcal{L}\left(\phi\right):\left(\psi\otimes\psi\right) & =\int_{B}\left[\begin{array}{cc}
\psi\cdot & \nabla\psi:\end{array}\right]\left[\begin{array}{cc}
W_{,\sigma\sigma}^{\nabla\phi}+\rho\phi^{*}\nabla^{2}U & W_{,\sigma\Phi}^{\nabla\phi}\\
W_{,\Phi\sigma}^{\nabla\phi} & W_{,\Phi\Phi}^{\nabla\phi}
\end{array}\right]\left[\begin{array}{c}
\cdot\psi\\
:\nabla\psi
\end{array}\right]+W_{,\Phi}^{\nabla\phi}:\left(\phi^{*}R^{TS}\vdots\left(\psi\otimes\psi\otimes\nabla\phi\right)\right)\,d\mu_{B},
\end{align*}
as claimed.
\end{proof}
\begin{proof}[Proof of \propref{W-is-spatially-invariant-and-homogeneous} - Spatial
invariance and homogeneity of uni-constant compressible Neo-Hooke
material]
Let
\begin{align*}
R & \in\Gamma\left(\pi^{*}E\right)\cong\Gamma\left(\pi_{S}^{*}TS\otimes\pi_{B}^{*}T^{*}B\right)
\end{align*}
be the canonical radial vector field on $\pi^{*}E$ defined by $R\left(F\right):=\left(F,F\right)$.
By construction, $R$ takes values in {[}the space isomorphic to{]}
the vertical subbundle of $TE$, so it follows that the $\sigma$
and $\beta$ partial covariant derivatives (both measuring horizontal
changes) are zero, and because $R$ is the identity map on fibers,
its $\Phi$ partial covariant derivative is the identity. Concretely,
\begin{align*}
R_{,\sigma} & \equiv0\in\Gamma\left(\pi_{S}^{*}TS\otimes\pi_{S}^{*}T^{*}S\right), & R_{,\beta} & \equiv0\in\Gamma\left(\pi_{B}^{*}TB\otimes\pi_{B}T^{*}B\right), & R_{,\Phi} & =\mathbb{I}_{\pi^{*}E}\in\Gamma\left(\pi^{*}E\otimes\pi^{*}E^{*}\right).
\end{align*}
$R$ can be used to define $C$ as a kind of endomorphism field over
$E$. In particular,
\begin{align*}
C & =\pi_{B}^{*}g^{-1}\cdot_{\pi_{B}^{*}T^{*}B}R^{\left(1\,2\right)}\cdot_{\pi_{S}^{*}T^{*}S}\pi_{S}^{*}h\cdot_{\pi_{S}^{*}TS}R\in\Gamma\left(\pi_{B}^{*}\left(TB\otimes T^{*}B\right)\right),
\end{align*}
noting that $R^{\left(1\,2\right)}\in\Gamma\left(\pi_{B}^{*}T^{*}B\otimes\pi_{S}TS\right)$
is the tensor transpose of $R$. Let $X\in\Gamma\left(TE\right)$
such that $\Phi\cdot X=0$. Then
\begin{align*}
\nabla_{X}R & =\nabla R\cdot X\\
 & =\left(R_{,\sigma}\cdot\sigma+R_{,\beta}\cdot\beta+R_{,\Phi}\cdot\Phi\right)\cdot X\\
 & =0\cdot\sigma+0\cdot\beta+R_{,\Phi}\cdot0\\
 & =0.
\end{align*}
Then, by metric compatibility,
\begin{align*}
\nabla_{X}C & =\pi_{B}^{*}g^{-1}\cdot\nabla_{X}R^{\left(1\,2\right)}\cdot\pi_{S}^{*}h\cdot R+\pi_{B}^{*}g^{-1}\cdot R^{\left(1\,2\right)}\cdot\pi_{S}^{*}h\cdot\nabla_{X}R\\
 & =\pi_{B}^{*}g^{-1}\cdot0\cdot\pi_{S}^{*}h\cdot R+\pi_{B}^{*}g^{-1}\cdot R^{\left(1\,2\right)}\cdot\pi_{S}^{*}h\cdot0\\
 & =0,
\end{align*}
showing that $C_{,\sigma}\equiv0$ and $C_{,\beta}\equiv0$. Because
$W=f\circ C$, where $f\left(C\right):=\alpha\left(\text{tr}\left(C-I\right)-\log\det C\right)$,
it follows that
\begin{align*}
W_{,\sigma}=C^{*}df\cdot C_{,\sigma} & =0, & W_{,\beta}=C^{*}df\cdot C_{,\beta} & =0,
\end{align*}
establishing the claim.
\end{proof}
\nocite{Abraham&Marsden,Antman,brenner_scott_2008,Dods2012,Dods2013,Dods2022,Eells&LeMaire,fike_2013_hyperdual,Gelfand&Fomin,Hetnarski&Ignaczak,JeffMLee,Klein_Roth_Valizadeh_Weeger_2023,Kolar&Michor&Slovak,kolev_desmorat_2023_souriau,kupferman_shamai_2012,Marsden&Hughes,Marsden&Ratiu,martins_sturdza_alonso_2003,peon_escalante_cantun_avila_carvente_espinosa_romero_penunuri,Plaza&Vallejo,rehner2021,RiemannianLee,Smith,YangGao_Neff_Roventa_Thiel}

\bibliographystyle{plainnat}
\addcontentsline{toc}{section}{\refname}\bibliography{references}

\end{document}